\documentclass[11pt]{amsart}
\usepackage{lmodern}
\usepackage[T1]{fontenc}
\usepackage[utf8]{inputenc}
\usepackage{amsmath,amssymb,amsthm}
\usepackage[a4paper,margin=1in]{geometry}
\usepackage{microtype}
\allowdisplaybreaks
\usepackage{xcolor}
\usepackage[colorlinks=true,linkcolor=blue!55!black,citecolor=blue!55!black,urlcolor=blue!55!black]{hyperref}
\hypersetup{pdftitle={Singularity Models of Finite-Time K\"ahler-Ricci Flows},pdfauthor={Frederick Tsz-Ho Fong, Hung Tran}}

\theoremstyle{plain}
\numberwithin{equation}{section}
\newtheorem{theorem}{Theorem}[section]
\newtheorem{proposition}[theorem]{Proposition}
\newtheorem{lemma}[theorem]{Lemma}
\newtheorem{corollary}[theorem]{Corollary}
\theoremstyle{definition}
\newtheorem{definition}[theorem]{Definition}

\theoremstyle{remark}
\newtheorem*{remarkx}{Remark}

\title{Singularity Models of Finite-Time K\"ahler-Ricci Flows}
\author{Frederick Tsz-Ho Fong}
\address{Frederick T.-H. Fong, Department of Mathematics, Hong Kong University of Science and Technology, Clear Water Bay, Kowloon, Hong Kong SAR}
\email{frederick.fong@ust.hk}
\author{Hung Tran}
\address{Hung Tran, Department of Mathematics and Statistics,
	Texas Tech University, Lubbock, TX 79409}
\email{hung.tran@ttu.edu}
\date{3 September 2026}
\newtheorem*{thmA}{Theorem A}
\newtheorem*{thmB}{Theorem B}
\newtheorem*{thmC1}{Theorem C1}
\newtheorem*{thmC2}{Theorem C2}
\newtheorem*{thmD}{Theorem D}
\newtheorem*{thmE}{Theorem E}

\begin{document}

\begin{abstract}
We study the singularity type and models of the K\"ahler--Ricci flow on compact manifolds constructed from the 1-parameter foliation of a circle-bundle over a product of K\"ahler--Einstein manifolds $N := N_1 \times \cdots \times N_r$, with metric constructed using the ansatz considered in \cite{DW2011}, \cite{WW} et. al.

In the earlier work \cite{FT} by the authors, we considered the ``two-bolt'' case where both ends of the foliation close with the ``bolt'' $N$. The compactification $\widehat{M}$ is then a $\mathbb{CP}^1$-bundle over $N$. We have already proved there that such a flow must encounter a Type I singularity.

In this article, we continue our work on the more subtle ``nut-bolt'' and ``two-nut'' cases. The former has one end of the interval closes with a nut-type collapse (i.e. $N' := N_2 \times \cdots \times N_r$) and the other with a bolt (i.e. $N$). The compactification $\widehat{M}$ is then a $\mathbb{CP}^{m+1}$-bundle over $N'$.  The ``two-nut'' case is one that both ends close with nut-type collapses, necessarily two of the $N_i$'s must be $\mathbb{CP}^{m_0}$ and $\mathbb{CP}^{m_\ell}$, and the compactification $\widehat{M}$ is a $\mathbb{CP}^{m_0+m_\ell+1}$-bundle over $\prod_{k\geq 3}N_k$. We proved that in both ``nut-bolt'' and ``two-nut'' caess the singularity must be of Type I.

Furthermore, we study the pointed Cheeger-Gromov limit of the rescaled and dilated sequence of the flow in all of the ``two-bolt'', ``nut-bolt'', and ``two-nut'' cases, and prove that the limit model must be $(\Sigma^{m+1}, g_\Sigma(t)) \times (\mathbb{C}^{k}, \textrm{flat})$ with $m, k \geq 0$, where $\Sigma$ is one of the following: $\mathbb{CP}^{m+1}$, $\textrm{Tot}(\mathcal{L}^{\oplus(m+1)})$, or a projectivization $\mathbb{P}\big(\mathcal{O}^{\oplus(m_0+1)} \oplus \mathcal{L}^{\oplus(m_\ell+1)}\big)$ with $m_0 + m_\ell = m$, and $\mathcal{L}$ is a line bundle over the product of \emph{some} of the $N_1, \cdots, N_r$ factors. The metric $g_\Sigma(t)$ is a K\"ahler-Ricci shrinker satisfying the circle-bundle ansatz.

\end{abstract}

\maketitle
\section{Introduction}\label{sec:intro}

\subsection{Background}
A solution $g ( t )$ of the Ricci flow $\partial_t g = - 2 \operatorname{Ric} ( g )$ starting from a K\"ahler metric on a compact complex manifold stays K\"ahler, and its maximal existence time is dictated by cohomology: by Tian--Zhang \cite{TZ}, the flow exists precisely until the first time $T$ at which the evolved K\"ahler class $[ \omega ( t ) ] = [ \omega_0 ] - 2 t \pi c_1$ leaves the K\"ahler cone. When $T < \infty$, the geometry at the singular time is probed by parabolic rescaling according to the blow-up rate of curvature. Letting $t_j \to T$, one would like to determine the Cheeger--Gromov convergence of the rescaled sequence $g_j(t) := K_j g(t_j + K_j^{-1}t)$ where $K_j := \sup_{\{t_j\} \times M} |\textrm{Rm}|_{g(t_j)}$. A finite-time singularity is of \emph{Type I} if $\sup_M ( T - t )\, | \mathrm{Rm} |$ stays bounded, and of \emph{Type II} otherwise. For Type I singularities, blow-up limits of $g_j(t)$ are nontrivial gradient shrinking solitons, by the theorem of Enders--M\"uller--Topping \cite{EMT} and Naber \cite{Naber}. The fundamental noncompact examples of such limit models in K\"ahler geometry are the shrinkers by Feldman--Ilmanen--Knopf \cite{FIK} on the line bundles $\mathcal{O} ( - k ) \to \mathbb{CP}^{n - 1}$ for $0 < k < n$, modelling the contraction of a divisor, together with their generalizations on bundles over a \emph{product} of K\"ahler--Einstein manifolds by Dancer--Wang \cite{DW2011}.

The Einstein or soliton metrics mentioned above belong to a classical cohomogeneity-one construction, which includes Calabi's $\mathrm{U} ( n )$-symmetric K\"ahler metrics \cite{Calabi}, B\'erard-Bergery's Einstein metrics on bundles over a K\"ahler--Einstein base \cite{BB} (including Page's metric), the Einstein metrics of Wang--Wang \cite{WW} and Wang--Ziller \cite{WZ} over \emph{products} of K\"ahler--Einstein factors, Cao's and Koiso's compact K\"ahler--Ricci solitons \cite{Cao96, Koiso}, and the cohomogeneity-one Einstein and soliton analysis of Dancer--Wang \cite{DW1998,DW2009,DW2011}. All of these can be unified into one ansatz constructed by taking a product $M_0 := I \times P$ of an open interval $I$ with a circle-bundle $P$ where $\pi : P \to N = N_1 \times \cdots \times N_r$ is a principal circle bundle over a product of compact K\"ahler--Einstein manifolds $(N_i, g_i, \omega_i)$ with Euler class $\sum_i q_i [\omega_i]$, $q_i \not= 0$, and equipping it with a metric of the multiply-warped product form
\begin{equation}
\label{eq:ansatz}\tag{*}
g \;=\; ds \otimes ds \;+\; H ( s )^2\, \theta \otimes \theta \;+\; \sum_{i=1}^r F_i ( s )^2\, \pi^* g_i \,, \qquad d\theta = \sum_{i=1}^r q_i\pi^*\omega_i,
\end{equation}
and closing the ends of $I$ by adjoining $N$ (\emph{bolt-type}), or, in the case $(N_1, g_1, p_1, q_1) = (\mathbb{CP}^m, 2g_{\textrm{FS}}, m+1, 1)$, by adjoining $N' = N_2 \times \cdots \times N_r$ (\emph{nut-type}). From now on, we will call this construction the \textbf{circle-bundle ansatz}. In our earlier paper \cite{FT}, we proved that this ansatz is preserved by the Ricci flow in both the bolt-type and nut-type closings, and no K\"ahler condition is needed there.

The circle-bundle ansatz includes an important subclass called \emph{Calabi's ansatz/symmetry}, in which the metric is determined by the complex Hessian of a $\mathrm{U}(n)$-symmetric K\"ahler potential (the conversion rule between \eqref{eq:ansatz} and the Calabi's ansatz can be found in \cite[Section 2.3]{FT}). In the case $r = 1$ and $N = \mathbb{CP}^n$ (i.e. the base is a \emph{single} K\"ahler-Einstein manifold), the compactification $\widehat{M}$ with both ends adjoined by $\mathbb{CP}^n$'s (which will be called the \emph{two-bolt} case) is a $\mathbb{CP}^1$-bundle over $\mathbb{CP}^n$. Assuming Calabi's symmetry, it was studied by Song-Weinkove in \cite{SW} where Gromov-Hausdorff convergence was established, and by Fong \cite{F}, Song \cite{S}, Guo-Song \cite{GS} who proved the singularity must be of Type I, and the singularity models are either the product $\mathbb{CP}^1 \times \mathbb{C}^n$, the Feldman-Ilmanen-Knopf shrinking K\"ahler-Ricci soliton on the total space of a line bundle $O(-k)$ over $\mathbb{CP}^n$, or in the canonical class case -- the compact Feldman-Ilmanen-Knopf shrinking K\"ahler-Ricci soliton on $\widehat{M}$.

For higher rank $\mathbb{CP}^{m+1}$-bundles over a \emph{single} K\"ahler-Einstein manifold $Z^n$ where $m, n \geq 1$ with Calabi's symmetry, it is interesting to note that it is also a subclass of the circle-bundle ansatz \eqref{eq:ansatz}. It is the special case when $r = 2$, $(N_1, g_1, p_1, q_1) = (\mathbb{CP}^m, 2g_{\textrm{FS}}, m+1, 1)$ and $N_2 = Z$, and the compactified bundle $\widehat{M}$ is obtained by adjoining $I \times P$ by one end with $Z$ (i.e. \emph{nut-type}), and the other end with $\mathbb{CP}^m \times Z$ (i.e. \emph{bolt-type}) -- we shall call such a compactification the \emph{nut-bolt closing}. In this case, Song-Yuan proved in \cite{SY} the Gromov-Hausdorff convergence; and recently in Jian-Song-Tian \cite{JST}, it was proved that the singularity is of Type I, and the singularity models are either $\mathbb{CP}^{m+1} \times \mathbb{C}^n$, the shrinking K\"ahler-Ricci soliton on the total space of a rank-$(m+1)$ complex vector bundle over $N^n$, or in the canonical class case -- the compact shrinking K\"ahler-Ricci soliton on $\widehat{M}$.

Without any symmetry assumption, it is very challenging to determine the singularity type and the Cheeger--Gromov convergence limit of the finite-time K\"ahler-Ricci flows. Conlon--Hallgren--Ma proved in \cite{CHM} that any non-collapsed finite time singularity of the Ricci flow on a compact K\"ahler surface is of Type I, and hence from Cifarelli--Conlon--Deruelle \cite{CD} the singularity is modelled on the Feldman--Ilmanen--Knopf's shrinker on $\mathcal{O}(-1)$-bundle over $\mathbb{CP}^1$. Recently, Xu--Zhang proved in \cite{XZ} that with the Type I singularity as the hypothesis, then if on the blow-up $\pi : \textrm{Bl}_p(Y) \to Y$ of a compact K\"ahler manifold $Y^n$ at a point $p \in Y$, if allow the K\"ahler-Ricci flow the K\"ahler class $[\omega(t)]$ degenerates to $\pi^*[\omega_Y]$ for some K\"ahler metric $\omega_Y$ on $Y$, then the blow-up limit is given by the Feldman--Ilmanen--Knopf's shrinker on the total space $\operatorname{Tot}\big(\mathcal{O}_{\mathbb{CP}^{n-1}}(-1)\big)$. In a very recent preprint \cite{XZ2}, Xu--Zhang also proved that any finite-time collapsing K\"ahler-Ricci flow on ruled surfaces must develop a Type I singularity, such singularity is modelled on the standard product $\mathbb{CP}^1 \times \mathbb{C}$.  Very recently, Jian--Song \cite{JS26} also considered the K\"ahler-Ricci flow on Fano bundles $\pi: X^n \to Y^m$ so that when the limiting class of the flow is $\pi^*[\omega_Y]$, i.e. the fibers collapse. They proved that the tangent flow splits as $Z = \mathbb{C}^k \times Z'$ for some $Z'$ where is some normal analytic variety. When the fiber has complex dimension one, the singularity is of Type I and every tangent flow at any fixed limiting point is $\mathbb{C}^k \times \mathbb{CP}^1$.

The authors now study the K\"ahler-Ricci flow under the circle-bundle ansatz \eqref{eq:ansatz} in full generality -- allowing the base manifold to be any product of (finitely many) K\"ahler-Einstein manifolds, and the compactification could be of \emph{two-bolt}, \emph{nut-bolt}, or \emph{two-nut} type. The two-bolt case was studied by the authors in \cite{FT}, and it was proved that in the K\"ahler case the singularity must be of Type I. The key ingredients include the splitting theorem obtained by analyzing the O'Neill's tensors, and the use of Perelman's local non-collapsing theorem to rule out Type II singularity. The classification of singularity models in this \emph{two-bolt} case was briefly discussed in \cite{FT} but not fully classified yet. We have remarked in \cite{FT} that the classification of singularity models in the general (i.e. $r \geq 2$) circle-bundle ansatz would be more diverse than the case of $\mathbb{CP}^1$-bundle over a single K\"ahler-Einstein manifold, since on the closing some but not all of the K\"ahler-Einstein factors may contract to a point.

The present paper will fully address the classification problem of singularity models described in the previous paragraph, and will further study the case with nut-bolt and two-nut closings. The former has one end closing with a nut $N' := N_2 \times \cdots \times N_r$ where $\{N_i\}_{i\geq 2}$ are K\"ahler-Einstein manifolds, while the other end closes with a bolt $N := \mathbb{CP}^m \times N'$ and the resulting manifold is a $\mathbb{CP}^{m+1}$-bundle over a \emph{product} of K\"ahler--Einstein manifolds:
\[
\widehat{M} \;\cong\; \mathbb{P} \big( \mathcal{O} \oplus E' \big) \;\longrightarrow\; N' \,, \qquad E' \;=\; \mathbb{C}^{m+1} \otimes L' \,,
\]
for a line bundle $L' \to N'$ determined by the charges $( q_2, \dots, q_r )$.  For $m = 0$ this degenerates to the two-bolt case by regarding $\mathbb{CP}^0$ as a point, so that $N = \{\textrm{pt}\} \times N' \cong N'$. Therefore, one can regard the two-bolt case as a special case of the nut-bolt case.

The remaining \emph{two-nut} closing --- both ends of $I$ closing with nut-type collapses --- admits the same analysis with the two ends playing symmetric roles. In the K\"ahler case, a two-nut closing is only possible when $N_1 = \mathbb{CP}^{m_0}$ and $N_2 = \mathbb{CP}^{m_l}$ and that $q_1, q_2$ have opposite charges. In this case, the lower end of $(0,l) \times P$ is adjoined by a copy of $Q_0 \cong \mathbb{CP}^{m_l} \times N_3 \times \cdots \times N_r$ whereas the upper end is adjoint by $Q_l \cong \mathbb{CP}^{m_0} \times N_3 \times \cdots \times N_r$. The compactification $\widehat{M}$ is then the projectivized bundle $\mathbb{P}(\mathcal{O}^{\oplus(m_0+1)} \oplus \mathcal{L}^{\oplus(m_l+1)})$ over $N'' := N_3 \times \cdots \times N_r$.

For simplicity, we will mainly present the nut-bolt case, whose results also hold in the two-bolt closing case, and we will point out in Section \ref{sec:twonut} the necessary modifications for the two-nut case, whose main results are stated as Theorem E below. Our main results are Theorems A to E below.

\subsection{Type I singularity}

\begin{thmA}[Type I; Theorem \ref{thm:nut-bolt-typeI}]
Let $\widehat{M}$ be a nut--bolt compactification of $I \times P$ with the circle-bundle ansatz \eqref{eq:ansatz}, and let $g ( t )$, $t \in [ 0, T )$, be a K\"ahler-Ricci flow of ansatz metrics with maximal existence time $T < \infty$. Then, the singularity at $T$ is of Type I.
\end{thmA}

A special case of Theorem A was known when $m = 0$ and $r = 1$, as proved by the first-named author in \cite{F} for the fiber-collapsing case, and Song \cite{S} for the contracting divisor case. Their proofs both used de Rham splitting and Perelman's local non-collapsing theorem, but the details were slightly different. Recently, the authors have proved in \cite{FT} a more general case of Theorem A with $m = 0$ and any $r \geq 1$, also using de Rham splitting and Perelman's local non-collapsing theorem, but in a more unified way without distinguishing the fiber-collapsing and the contracting divisor cases. Furthermore, in the case $m \geq 1$ and $r = 2$ (i.e. $\mathbb{CP}^{m+1}$-bundle over a \emph{single} K\"ahler-Einstein manifold), Theorem A was proved by Jian--Song--Tian in \cite{JST} using a completely different approach from \cite{F, S, FT}. Very recently, in the case $m = 1$ while without any symmetry assumption, Jian--Song \cite{JS26} proved the K\"ahler-Ricci flow must encounter Type I singularity on any Fano bundle $X^{n+1} \to Y^n$ in the fiber-collapsing case.

Hence, Theorem A pushes the Type I singularity results of former works \cite{F, S, GS, JST, FT} to full generality under this circle-bundle ansatz. In \cite{FT} the O'Neill's tensor associated with the submersion $\pi : I \times P \to N$ was used. In this case, the submersion map extends smoothly to both closings to become a submersion $\widehat{\pi} : \widehat{M} \to N$. The $\mathbb{CP}^1$-fibers are totally geodesic, hence the O'Neill's $T$-tensor vanishes. If the singularity were of Type II, we have proved in \cite{FT} that the O'Neill's $A$-tensor would vanish in the limit model, and hence the limit model would split off a $2$-dimensional factor, which would be the cigar steady soliton according to Hamilton's classification theorem. This then violates Perelman's local non-collapsing theorem, hence proving Type I singularity in \cite{FT}. Now that the fibers are $( 2 m + 2 )$-dimensional, the projection map $\pi : I \times P \to N$ does not extend to a smooth submersion to the whole $\widehat{M}$. Instead, we need to consider the map $\pi' := \textrm{pr}_{N'} \circ \pi : I \times P \to N' := N_2 \times \cdots \times N_r$, where $\textrm{pr}_{N'} : N \to N'$ is another projection map. Such a submersion $\pi'$ would extend smoothly to a submersion $\widehat{\pi}' : \widehat{M} \to N'$ with totally geodesic fibers. The Type I singularity will be proved by a chain of results including a suitable lower estimate for the holomorphic bisectional curvature of the fiber block, and hence giving an eternal Type II limit with \emph{nonnegative bisectional curvature}. Cao's theorem on eternal solutions \cite{Cao97} then shows it is a steady gradient K\"ahler--Ricci soliton; and the rigidity theorem of Deng--Zhu \cite{DZ18, DZ20} forces such a $\kappa$-noncollapsed soliton to be flat --- giving a contradiction to the Type II singularity. The excluded solitons are precisely the higher-dimensional avatars of the cigar: Cao's steady solitons on $\mathbb{C}^{m+1}$ and on $K_{\mathbb{CP}^m}$ \cite{Cao96}, all $\kappa$-collapsed. Overall, the proof is in a similar spirit to \cite{F, FT} with more technical analysis, but the approach is very different from that in \cite{JST} where the notion of Ricci vertex and a partial Type I estimate of the scalar curvature were used.

\subsection{Singularity models}

Apart from proving the singularity is of Type I, we further give the full classification of singularity models of the circle-bundle ansatz in the present paper under the finite-time K\"ahler-Ricci flow. As in previous works on the finite-time singularity of the K\"ahler-Ricci flow, there are two major regimes: (a) volume collapsing (i.e. $\textrm{Vol}_{g(t)}(M) \to 0$ as $t \to T$), or (b) volume non-collapsing (i.e. $\inf_{M \times [0,T)} \textrm{Vol}_{g(t)}(M) > 0$). For $\mathbb{CP}^{m+1}$-bundles, Case (a) corresponds to the collapsing of $\mathbb{CP}^{m+1}$-fibers, and Case (b) corresponds to the contraction of the some K\"ahler-Einstein factors of the zero/infinity-section.

\subsubsection{Fiber-collapsing case} For Case (a), we will prove:

\begin{thmB}[fiber collapse; Theorem \ref{thm:fiber-collapse-model}]
When the $\mathbb{CP}^{m+1}$-fibers collapse at the first singular time $T < \infty$ while the volume of the closing $Q_0 \cong N' = N_2 \times \cdots \times N_r$ is bounded away from zero, the blow-up limit is the product of $\mathbb{CP}^{m+1}$ with the shrinking Fubini--Study metric, and the flat metric on $\mathbb{C}^{\dim_{\mathbb{C}}N'}$.
\end{thmB}

Two ``orthogonal'' special cases of Theorem B were previously known: (i) when the base manifold is a \emph{single} K\"ahler-Einstein manifold and the fibers are $\mathbb{CP}^{m+1}$ where $m \geq 0$ -- as proved by \cite{F} for $m = 0$, and \cite{JST} for $m \geq 1$; or (ii) when $m = 0$ and $r \geq 1$ -- as proved by \cite{FT}. Now we have a complete result for all $m \geq 0$ and $r \geq 1$.

 Very recently Jian--Song proved in \cite{JS26} that when $m = 1$ in the fiber-collapsing case, Theorem B holds without any symmetry/ansatz assumption. For $m \geq 2$, they obtained that the singularity model splits as $\mathbb{C}^k \times Z^{m+1}$ for some normal analytic variety $Z^{m+1}$. It concurs with our Theorem B, which also identities $Z^{m+1}$ as $\mathbb{CP}^{m+1}$ under the our circle-bundle ansatz assumption.

\subsubsection{Contraction case}

As remarked, the interesting and challenging case is the contraction of zero/infinity-section of the bundle. In our general circle-bundle ansatz, it is possible that on the closing, some but not all of the K\"ahler-Einstein factors contract. Also, in the nut-bolt closing, the zero- and infinity-sections are also topologically different, which diversifies the types of singularity models.

\begin{thmC1}[contraction at the nut-type stratum, i.e. zero-section; Theorem \ref{thm:contracting_Q0}]
Let $\widehat{M}$ be the compactification of $(0,\ell) \times P$ such that $Q_0 := N' = N_2 \times \cdots \times N_r$ is adjoined at $\{s = 0\}$, and $Q_\ell := \mathbb{CP}^m \times N'$ is adjoined at $\{s = \ell\}$. Suppose the total volume of $\widehat{M}$ stays positive at the first singular time $T < \infty$, and the singularity occurs at $Q_0$. Let $\mathcal{I}_0$ be the subset of $\{2, \cdots, r\}$ such that $i_0 \in \mathcal{I}_0$ if and only if the $N_{i_0}$-factor of $Q_0$ contracts to a point at $T$. Then, the blow-up limit splits as $( \operatorname{Tot} ( E_{\mathcal{I}_0} ), g_{\mathrm{sol}} ( t ) ) \times \mathbb{C}^k$, where $k = \sum_{i_0 \not\in \mathcal{I}_0,\, i_0 \not= 1} \dim_{\mathbb{C}} N_{i_0}$, $E_{\mathcal{I}_0} = \mathcal{L}_{\mathcal{I}_0}^{\oplus ( m + 1 )} \to \prod_{i_0 \in \mathcal{I}_0} N_{i_0}$ is the vector bundle obtained by restricting the normal bundle of $Q_0 \subset \widehat{M}$ to $\prod_{i_0 \in \mathcal{I}_0} N_{i_0}$, with $c_1 ( \mathcal{L}_{\mathcal{I}_0} ) = - \sum_{i_0 \in \mathcal{I}_0} q_{i_0} [\omega_{i_0}]$, and $g_{\mathrm{sol}}$ is a complete gradient shrinking K\"ahler--Ricci soliton of ansatz form on $\operatorname{Tot} ( E_{\mathcal{I}_0} )$ --- i.e. a higher-rank, multi-factor generalization by Dancer--Wang in \cite{DW2011} of the Feldman--Ilmanen--Knopf shrinker in \cite{FIK}.
\end{thmC1}

Theorem C1, together with the forthcoming variants, addresses the questions asked in \cite{FT} about the singularity model in the case of contraction of some, not necessarily all, K\"ahler-Einstein factors. Special cases of Theorem C1 were proven in \cite{S, GS} when $m = 0$, $r = 2$, and $\mathcal{I}_0 = \{2\}$; and in \cite{JST} when $m \geq 1$, $r = 2$, and $\mathcal{I}_0 = \{2\}$. We will prove Theorem C1 by explicitly constructing the Cheeger-Gromov diffeomorphism sequence $\Phi_j : W_j \subset M_\infty \to U_j \subset \widehat{M}$. Among all technical steps in the whole proof, the most important step is to establish the smooth convergence of the pullback sequence $\Phi_j^* g_j(t)$, where $g_j(t) := K_jg(t_j + K_j^{-1}t)$. In this regard, we need $C^k$-estimates of each metric component of $\Phi_j^*g_j(t)$. This will be done using the uniform bounds on $|\nabla^k \textrm{Rm}|_{g_j(t)}$ obtained from Shi's derivative estimates \cite{Shi}. The key idea of this step is the relation between the sectional curvature of $g_j(t)$ (resp. its derivatives) along some directions, and the second-order (resp. higher-order) derivatives of some metric components of $\Phi_j^*g_j(t)$. After establishing $C^k$-estimates of some metric components of $\Phi_j^*g_j(t)$, the smooth convergence is hence achieved by a standard Arzel\`a--Ascoli-type argument.

By modifying the proof of Theorem C1, \emph{mutatis mutandis}, one can also obtain a similar result if the singularity occurs at the bolt closing (i.e. infinity-section), which differs from Theorem C1 only in the rank and the sign of the vector bundle.

\begin{thmC2}[contraction at the bolt-type stratum; Theorem \ref{thm:contracting_Ql}]
Let $\widehat{M}$ be as in Theorem C1, and again the total volume stays positive at the first singular time $T < \infty$. Now assume the singularity occurs at a bolt $Q_\ell := N$. Let $\mathcal{I}_\infty$ be the subset of $\{1, 2, \cdots, r\}$ such that $i_\infty \in \mathcal{I}_\infty$ if and only if the $N_{i_\infty}$-factor of $Q_\ell$ contracts to a point at the singular time $T$ of the K\"ahler-Ricci flow. Then, the blow-up limit splits as $( \operatorname{Tot} ( \mathcal{L}_{\mathcal{I}_\infty} ), g_{\mathrm{sol}} ( t ) ) \times \mathbb{C}^k$, where $k = \sum_{i_\infty \not\in \mathcal{I}_\infty,\, i_\infty \not= 1} \dim_{\mathbb{C}} N_{i_\infty}$, $\mathcal{L}_{\mathcal{I}_\infty} \to \prod_{i_\infty \in \mathcal{I}_\infty} N_{i_\infty}$ is the \emph{line} bundle obtained from restricting the normal bundle of $Q_\ell \subset \widehat{M}$ to $\prod_{i_\infty \in \mathcal{I}_\infty} N_{i_\infty}$, with $c_1 ( \mathcal{L}_{\mathcal{I}_\infty} ) = \sum_{i_\infty\in \mathcal{I}_\infty} q_{i_\infty} [\omega_{i_\infty}]$, and $g_{\mathrm{sol}}$ is a complete gradient shrinking K\"ahler--Ricci soliton of ansatz form on the line bundle total space $\operatorname{Tot} ( \mathcal{L}_{\mathcal{I}_\infty} )$ --- the Feldman--Ilmanen--Knopf shrinker in \cite{FIK}, or its multi-base generalization by Dancer--Wang in \cite{DW2011}.

\end{thmC2}

Note that Theorems C1 and C2 also contain the two-bolt case (with $m = 0$) in which the compactification $\widehat{M}$ is obtained by adjoining $N$ on both ends.

\subsubsection{Borderline cases}
With multiple K\"ahler-Einstein factors on the base manifold $N$, the borderline cases -- meaning that both the collapsing of $\mathbb{CP}^{m+1}$-fibers and the contraction of zero/infinity-section happen simultaneously at the singular time $T < \infty$ -- are more diversified than the case of a $\mathbb{CP}^{m+1}$-bundle over a \emph{single} K\"ahler-Einstein manifold. In the single-factor case, this borderline case refers to the Fano K\"ahler-Ricci flow in the canonical class, where the K\"ahler class $[\omega(t)]$ stays a constant multiple of the first Chern class $c_1(\widehat{M})$. Convergence is relatively better understood and the limit model is a compact shrinking K\"ahler-Ricci soliton. Now with multi-factors on the base $N$, it could happen that some but not all K\"ahler-Einstein factors contract to a point, and the $\mathbb{CP}^{m+1}$-fibers collapse at the same singular time $T$. It is not the canonical class case as the manifold does not become extinct. Fortunately, Theorem C1 was proved using an explicit construction of diffeomorphism sequence $\Phi_j$, so with some slight modification of the proof, we can address the limit model of these ``borderline'' cases:

\begin{thmD}
Let $\widehat{M}$ be as in Theorems C1 and C2. Suppose at the first singular time $T < \infty$ that both the $\mathbb{CP}^{m+1}$-fibers collapse, and some of the $N_i$-factors of $Q_0$ (hence also of $Q_\ell)$ contract to a point (denote the index set of such $N_i$-factors by $\mathcal{I} \subset \{2, \cdots, r\}$). Then, the blow-up limit splits as $( \mathbb{P}(\mathcal{O} \oplus E), g_{\mathrm{sol}} ( t ) ) \times (\mathbb{C}^k, \textrm{flat})$, where $E = \mathcal{L}_{\mathcal{I}}^{\oplus (m+1)} \to \prod_{i \in \mathcal{I}} N_i$ is the normal bundle of $Q \subset \widehat{M}$ restricted to $\prod_{i \in \mathcal{I}} N_i$ with $c_1(\mathcal{L}_{\mathcal{I}}) = \mp\sum_{i\in\mathcal{I}}q_i[\omega_i]$, $g_{\mathrm{sol}}(t)$ is the unique compact shrinking K\"ahler-Ricci soliton on the projectivized $\mathbb{CP}^{m+1}$-bundle $\mathbb{P}(\mathcal{O} \oplus E)$ over $\prod_{i \in \mathcal{I}} N_i$, and $k = \sum_{i \not\in \mathcal{I}} \dim_{\mathbb{C}} N_{i}$. Note that it is possible that $k = 0$, when all $N_i$-factors collapse, and it would become the Fano canonical class case and hence the limit manifold is $\widehat{M}$ itself.
\end{thmD}

\subsubsection{The two-nut case}
The trichotomy of possible closings is completed by the \emph{two-nut} case, where both ends of $I$ close with nut-type collapses. The K\"ahler condition forces two of the factors to be complex projective spaces $(N_1, N_2) = (\mathbb{CP}^{m_0}, \mathbb{CP}^{m_\ell})$ with unit charges of opposite signs $q_1 = +1$ and $q_2 = -1$, so that $\widehat{M} \cong \mathbb{P}\big(\mathcal{O}^{\oplus(m_0+1)} \oplus (\mathcal{L}'')^{\oplus(m_\ell+1)}\big)$ is a $\mathbb{CP}^{m_0+m_\ell+1}$-bundle over $N''' := N_3 \times \cdots \times N_r$. All results above admit two-nut analogues, with the two ends of $I$ playing symmetric roles:

\begin{thmE}[the two-nut case; Theorems \ref{thm:twonut_typeI} and \ref{thm:twonut_models}]
Let $\widehat{M}$ be a compact two-nut closing of $I \times P$ with the circle-bundle ansatz \eqref{eq:ansatz}, and let $g(t)$, $t \in [0, T)$, be a K\"ahler-Ricci flow of ansatz metrics with maximal existence time $T < \infty$. Then, the singularity at $T$ is of Type I, and the blow-up limit splits as $(\Sigma, g_\Sigma(t)) \times (\mathbb{C}^k, \textrm{flat})$, where $(\Sigma, g_\Sigma(t))$ is one of the following:
\begin{itemize}
\item $\mathbb{CP}^{m_0+m_\ell+1}$ with the shrinking Fubini--Study metric, when the fibers collapse;
\item $\operatorname{Tot}\big(\mathcal{L}_{\mathcal{I}_0}^{\oplus(m_0+1)}\big)$, the total space of a rank-$(m_0+1)$ vector bundle over the product of the $N_k$-factors ($k \geq 3$) contracting at $Q_0$, with the Dancer--Wang complete shrinking K\"ahler--Ricci soliton --- or its mirror $\operatorname{Tot}\big(\mathcal{L}_{\mathcal{I}_\infty}^{\oplus(m_\ell+1)}\big)$ for a contraction at $Q_\ell$;
\item the projectivization $\mathbb{P}\big(\mathcal{O}^{\oplus(m_0+1)} \oplus \mathcal{L}_{\mathcal{I}}^{\oplus(m_\ell+1)}\big)$ over the product of the contracting $N_k$-factors, with its unique compact shrinking K\"ahler--Ricci soliton, in the borderline case where the fiber collapse and the contraction happen simultaneously.
\end{itemize}
\end{thmE}

\subsubsection{Summary of all possible singularity models}
To summarize, the possible singularity models of the K\"ahler-Ricci flow on $\widehat{M}$ under our circle-bundle ansatz are given by $(\Sigma^{m+1}, g_\Sigma(t)) \times (\mathbb{C}^k, \textrm{flat})$ with $m, k \geq 0$ such that $m+k+1 = \dim_{\mathbb{C}}\widehat{M}$, where $(\Sigma^{m+1}, g_\Sigma(t))$ is one of the following:

\begin{itemize}
\item $\mathbb{CP}^{m+1}$ with the shrinking Fubini-Study metric, or
\item $\textrm{Tot}(E)$, the total space of a rank-$(m+1)$ vector bundle $E$ over a product of \emph{some} of the $N_1, \cdots, N_r$ factors, with $g_\Sigma(t)$ being the Feldman-Ilmanen-Knopf/Dancer-Wang's complete shrinking K\"ahler-Ricci soliton, or
\item the projectivization $\mathbb{P}(\mathcal{O} \oplus E)$ of a rank-$(m+1)$ vector bundle $E$ described as the above, with the unique compact shrinking K\"ahler-Ricci soliton --- enlarged, in the two-nut case, to the projective bundles $\mathbb{P}\big(\mathcal{O}^{\oplus(m_0+1)} \oplus \mathcal{L}^{\oplus(m_\ell+1)}\big)$ for a line bundle $\mathcal{L}$ (Theorem \ref{thm:twonut_models}).
\end{itemize}

\medskip
\noindent\textbf{Organization.} Section \ref{sec:prelim} collects the basic facts that this paper rests on: the ansatz and its closing conditions, the K\"ahler structure and Calabi symmetry, the O'Neill tensors, the curvature and radial calculus, the periods of the first Chern class, the Ricci flow equations and the preservation of the ansatz, and the two-bolt Type I theorem, Li--Yau and entropy estimates obtained from \cite{FT} and from the classical references indicated there. Section \ref{sec:typeI} sets up the nut--bolt case --- the residual bundle $\widehat{\pi}' : \widehat{M} \to N'$, the decay of the residual O'Neill tensor, and the bisectional curvature of the fiber block --- and proves Theorem A. Section \ref{sec:limits} identifies the blow-up limits: Theorem B (Theorem \ref{thm:fiber-collapse-model}) and Theorem C1 (Theorem \ref{thm:contracting_Q0}). The ideas of the proofs of Theorems C2 and D will be explained based on the proof of Theorem C1. Sections \ref{sec:prelim}--\ref{sec:limits} present mainly the nut-bolt case; Section \ref{sec:twonut} completes the two-nut closing, where the two ends are adjoined by nut-type closings. The proof is a slightly modification from theorems in the bolt-nut case, so we will sketch its proof and points out the necessary modifications for the Type I theorem (Theorem \ref{thm:twonut_typeI}) and for the classification of the singularity models (Theorem \ref{thm:twonut_models}).

\medskip
\noindent\textbf{Acknowledgements.} The first-named author is partially supported by the General Research Fund \#16305625 by the Hong Kong Research Grants Council.

\medskip

\noindent\textbf{Declaration of AI usage:} The whole train-of-thought in the proof of Theorem \ref{thm:nut-bolt-typeI} in Section \ref{sec:typeI} of using bisectional curvature $-C(T-t)$ lower bound, the use of splitting theorems, the use of Cao's Theorem on eternal solutions, and the use of Deng--Zhu's non-existence theorem of steady gradient K\"ahler-Ricci solitons to rule Type II singularity are all the \textbf{authors' original ideas}. However, some technical detail such as the observation of viewing $h$ is a function of $f_1$ in Proposition \ref{prop:reduced-flow}, and subsequently leading to the estimates about $h$ and $f_1$ in Lemma \ref{lem:h2pp-controls-bisec} to simplify the sectional curvature expression, and derivation of the upper bound of $(h^2)''$ in Theorem \ref{thm:h2pp-upper-bound} leading to the $-C(T-t)$ lower bound of the bisectional curvature, are the contributions by Claude Fable 5. All detail steps have been carefully checked, and the whole proofs were completely written by the authors.

In Sections \ref{sec:limits} and \ref{sec:twonut}, the major contribution from Claude Fable 5 is the great suggestion of using Shi's derivative estimates and the Poincar\'e lemma to establish the smooth convergence of $\Phi_j^*g_j(t)$. Based on the ideas suggested by Claude Fable 5, the detail proofs are completely written by the authors.

\section{Preliminaries}\label{sec:prelim}

This section fixes notation and recalls preliminary definitions and statements. The main references are our companion paper \cite{FT} along with \cite{DW1998, DW2011, VZ}.

\subsection{The circle-bundle ansatz and the K\"{a}hler condition}\label{subsec:ansatz}

The \emph{circle-bundle ansatz} is a metric ansatz of {multiple warped product} type. It goes back to Bérard-Bergery's construction \cite{BB} of Einstein metrics on bundles over a Kähler--Einstein base (with Page's metric \cite{Page} as the four-dimensional prototype), was extended to circle bundles over products of Kähler--Einstein factors by Wang--Wang \cite{WW}, and was employed by Dancer--Wang \cite{DW1998,DW2011} in their study of cohomogeneity-one Einstein metrics and Ricci soliton. 

The underlying manifold contains a one-parameter foliation by copies of a principal circle bundle $P$ over a product of Kähler--Einstein manifolds, together with the associated family of metrics. Specifically, we fix an integer $r \geq 1$. For each $i = 1, \dots, r$, let $(N_i, J_i, g_i, \omega_i)$ be a compact K\"{a}hler--Einstein manifold of real dimension $2n_i$, with complex structure $J_i$, Kähler metric $g_i$, and Kähler form $\omega_i = g_i(J_i \cdot\,, \cdot)$, normalized so that
\[
\operatorname{Ric}(g_i) = p_i\, g_i, \qquad p_i \in \mathbb{Z}.
\]
Write the first Chern class as $c_1(N_i, J_i) = p_i\, a_i$, where $a_i \in H^2(N_i; \mathbb{Z})$ is an {indivisible} integral class. Since the Ricci form $\rho_i = p_i\, \omega_i$ represents $2\pi\, c_1(N_i)$, the Kähler class satisfies $[\omega_i] = 2\pi\, a_i$.

Let
\[
N := N_1 \times \cdots \times N_r, \qquad \dim_{\mathbb{R}} N = 2n, \quad n := \sum_{i=1}^r n_i,
\]
equipped with the product complex structure. Each $\mathrm{pr}_i : N \to N_i$ denotes the standard projection and we write $\omega_i$ and $a_i$ also for the pullbacks $\mathrm{pr}_i^*\, \omega_i$ and $\mathrm{pr}_i^*\, a_i$ on $N$. Let $\Pi : P \to N$ be a principal $\mathrm{U}(1)$-bundle whose Euler class is
\[
e ( P ) \;=\; \sum_{i=1}^r q_i\, a_i \,, \qquad q_i \in \mathbb{Z} \,.
\]
The existence of $P$ is due to that $\big[ \sum_i q_i\, \omega_i \big] = 2 \pi \sum_i q_i\, a_i$ represents an integral cohomology class and the classification of principal $\mathrm{U}(1)$-bundles over a smooth manifold \cite{Hatcher2002}. By Chern–Weil theory \cite[Chapter~XII]{KobayashiNomizu1969}, there exists a principal connection $\theta \in \Omega^1 ( P )$ with curvature
\[
d \theta \;=\; \sum_{i=1}^r q_i \Pi^* \omega_i \,.
\]
Let $\xi$ denote the generator of the $\mathrm{U}(1)$-action, so that $\theta ( \xi ) = 1$, and let $\mathcal{H} := \ker \theta$ be the horizontal distribution, on which $\pi_*$ is an isomorphism onto $TN$ block by block: $\mathcal{H} = \mathcal{H}_1 \oplus \cdots \oplus \mathcal{H}_r$ with $\Pi_* \mathcal{H}_i = \mathrm{pr}_i^* T N_i$.

\begin{definition}[Circle-bundle ansatz]\label{def:ansatz}
	Given smooth positive functions $H, F_1, \dots, F_r : I \to (0, \infty)$, the \emph{circle-bundle ansatz} is the Riemannian metric on $M^0 = I \times P$ given by
	\begin{equation}\label{eq:ansatz-metric}
		g \;=\; ds^2 \;+\; H(s)^2\, \theta \otimes \theta \;+\; \sum_{i=1}^r F_i(s)^2\, \pi^* g_i .
	\end{equation}
\end{definition}


\noindent It follows that $\pi_s : (P_s, g_s) \to \big(N, \sum_i F_i(s)^2 g_i\big)$ is a Riemannian submersion with totally geodesic circle fibers of length $2\pi H(s)$. Naturally, one defines the horizontally conformal submersion \cite{OR93}:
\[\pi: M^0\to \big(N, \sum_i F_i(s)^2 g_i\big) \text{   by  } \pi_{\mid P_s}= \pi_s.\] 
With respect to the unit normal $\nu = \partial_s$, the shape operator $L_s$ of the leaf $P_s$ is diagonal with respect to the splitting $TP = \mathbb{R}\xi \oplus \bigoplus_i \mathcal{H}_i$ (where $\mathcal{H}_i$ is the horizontal lift of $TN_i$):

\begin{equation}\label{eq:shape-operator}
	L_s \;=\; \frac{H'(s)}{H(s)}\, \mathrm{Id}_{\mathbb{R}\xi} \;\oplus\; \bigoplus_{i=1}^r \frac{F_i'(s)}{F_i(s)}\, \mathrm{Id}_{\mathcal{H}_i},
\end{equation}


	

For the circle-bundle ansatz with metric \eqref{eq:ansatz-metric}, the complex structure $J$ is given by 
\[
J \partial_s \;=\; \frac{1}{H(s)}\, \xi, \qquad J \xi \;=\; -H(s)\, \partial_s, \qquad J X \;=\; \widetilde{J_N\, \pi_* X} \quad \text{for } X \in \mathcal{H}.
\]
The K\"{a}hler form $\omega := g(J \cdot\,, \cdot)$ is
\begin{equation}\label{eq:kahler-form}
	\omega \;=\; H(s)\, ds \wedge \theta \;+\; \sum_{i=1}^r F_i(s)^2\, \pi^* \omega_i.
\end{equation}
$(M^0, J, g)$ is a Kähler manifold precisely when $d\omega = 0$. That is, 


\begin{equation}\label{eq:kahler-condition}
	q_i\, H(s) \;=\; \frac{d}{ds} F_i(s)^2 \;=\; 2\, F_i(s)\, F_i'(s) \qquad \text{for all } i = 1, \dots, r \text{ and } s \in I.
\end{equation}

The $r = 1$ instance of the ansatz over $\mathbb{CP}^{n-1}$ is exactly Calabi's classical $\mathrm{U} ( n )$-symmetry \cite{Calabi}. Furthermore, we let \[\kappa_i := (F_i^2)' - q_i H\] denote the \emph{Kähler-defect functions}. In case (2) of Lemma \ref{lem:bolt-nut-dichotomy} with $r \geq 2$, the K\"{a}hler condition induces the following relation between the warping functions:
\begin{align}
	\label{eq:F1-int-H}
	F_1^2 ( s ) &\;=\; \int_0^s H ( \sigma )\, d\sigma \,;\\
	\label{eq:Fk-affine}
	F_k^2 &\;=\; F_k^2 \big|_{Q_0} \;+\; q_k\, F_1^2 \qquad ( k \geq 2 ) \,.
\end{align}

\subsection{Closing conditions at the boundary of $I$}\label{subsec:closing}
A compactification of $M^0$ with the equipped metric at a finite endpoint of $I$ is obtained by adjoining a suitable set to it and demanding that $g$ extend smoothly across the closing stratum. Without loss of generality, we assume the endpoint under consideration is $s = 0$ and $I = (0, \ell)$ with $0 < \ell \leq \infty$. Immediately, the compactification forces some of the warping functions to degenerate as $s \rightarrow 0^+$. 

Let $Q$ be the set adjoined to $M^0$ at $s=0$.  Following the terminology from the physics literature, $Q$ is called a \emph{bolt} when $\dim Q > 0$ and a \emph{nut} when $Q$ is a point. We first observe the dichotomy of degeneration.  

\begin{lemma}\label{lem:bolt-nut-dichotomy}
	Suppose $(M^0, g)$ closes smoothly at $s = 0$ and that $H, F_1, \dots, F_r$ extend continuously to $s = 0$. Let $J := \{\, j : F_j(0) = 0 \,\}$. Then necessarily $H(0) = 0$, and exactly one of the following holds:
	
	\begin{enumerate}
		\item $J = \varnothing$ (\emph{bolt of codimension two}): only the circle fibers collapse, and the sub-manifold $Q$ at $\{s=0\}$ is diffeomorphic to $N$.
		\item $J = \{ j_0 \}$ (\emph{nut-type collapse}): $(N_{j_0}, g_{j_0})$ is holomorphically isometric to $\mathbb{CP}^{n_{j_0}}$ with the Fubini--Study metric, $|q_{j_0}| = 1$, and $Q \cong \prod_{k \neq j_0} N_k$ is a bolt of codimension $2 n_{j_0} + 2$. When $r = 1$ this means $P \cong S^{2n+1}$ and the entire leaf collapses to a \emph{nut} $Q = \{ \mathrm{pt} \}$; for $r \geq 2$ the collapse is a nut \emph{fiberwise over} the bolt $Q$.
	\end{enumerate}
	
\end{lemma}
 A proof of Lemma \ref{lem:bolt-nut-dichotomy} will be provided in the Appendix. The metric is of cohomogeneity one and the topology of $P$ is given, so the closing conditions are well-known (see, for example, \cite{VZ, DW1998}) and recalled below. 
\begin{proposition}[Closing conditions for a bolt $Q \cong N$]\label{prop:bolt-closing}
	In case (1) of Lemma \ref{lem:bolt-nut-dichotomy}, $(M^0, g)$ closes smoothly at $s = 0$ if and only if
	
	\begin{enumerate}
		\item $H$ extends to a smooth \textbf{odd} function of $s$ with $H'(0) = 1$; equivalently $H(s) = s\, \varphi(s^2)$ for a smooth function $\varphi$ with $\varphi(0) = 1$; and
		\item each $F_i$ extends to a smooth \textbf{even} function of $s$ with $F_i(0) > 0$.
	\end{enumerate}
\end{proposition}
\noindent In this case $\widehat{M}$ is, near $Q$, diffeomorphic to a neighbourhood of the zero section of the complex line bundle $L := P \times_{\mathrm{U}(1)} \mathbb{C}$ over $N$, with $Q$ the zero section.\\

In case (2) of Lemma \ref{lem:bolt-nut-dichotomy}, for convenience, we consistently relabel so that $j_0 = 1$ and set $m := n_1$.
\begin{proposition}[Closing conditions for a nut-type collapse]\label{prop:nut-closing}
	In case (2) of Lemma \ref{lem:bolt-nut-dichotomy}, a smooth closing at $s = 0$ requires:
	
	\begin{enumerate}
		\item $H$ and $F_1$ extend to smooth \textbf{odd} functions of $s$, and each $F_k$ with $k \geq 2$ extends to a smooth \textbf{even} function with $F_k(0) > 0$;
		\item $H'(0) = 1$ and $F_1'(0) = \tfrac{1}{\sqrt{2}}$.
	\end{enumerate}

\end{proposition}
In this case, if $r = 1$, $\widehat{M}$ near $Q$ is a ball $B^{2n+2}$, and adjoining the point yields $M^0 \cup Q \cong \mathbb{C}^{n+1}$ topologically. For $r \geq 2$, the map $P \to Q = \prod_{k \geq 2} N_k$ is a fiber bundle with fiber $S^{2m+1}$ (the circle bundle over $\mathbb{CP}^m$ with $|q_1| = 1$), and $\widehat{M}$ near $Q$ is the total space of the associated rank-$(2m+2)$ vector bundle obtained by coning off the sphere fibers.

\begin{remarkx}To obtain compact examples, one requires that both endpoints of $I = (0, \ell)$ are finite and $(M^0, g)$ closes smoothly at both. For instance: closing with bolts $Q \cong N$ at both ends yields $\widehat{M} \cong P \times_{\mathrm{U}(1)} \mathbb{CP}^1$, the $S^2$-bundle over $N$ associated to $P$; for $r = 1$, $N_1 = \mathbb{CP}^n$, $q_1 = 1$, a nut at one end and a bolt at the other yields $\widehat{M} \cong \mathbb{CP}^{n+1}$, while nuts at both ends yield $\widehat{M} \cong S^{2n+2}$. 
\end{remarkx}
\begin{remarkx}
	The submersion $\pi : M^0 \to N$ is defined a priori only on the open manifold $M^0$. In case (1) of Lemma \ref{lem:bolt-nut-dichotomy}, with the identification $Q \cong N$, it extends to the vector-bundle projection $\widehat{\pi} : \widehat{M} \to N$ with $\widehat{\pi}|_Q = \mathrm{id}_N$. In the second case, there is no such extension. Instead, one writes $N = \mathbb{CP}^{n_{1}} \times N'$ with $N' := \prod_{k \geq 2} N_k$. When $r \geq 2$, by a natural composition of maps, $\widehat{M}$ can be viewed as a fiber bundle over $N'$ and this perspective will be crucial in our analysis. 
\end{remarkx}

\subsection{The O'Neill tensors}\label{subsec:oneill}

For a submersion with vertical distribution $\mathcal{V}$ and horizontal one $\mathcal{H}$, O'Neill's fundamental tensors \cite{ONeill66} are
\[
T_E F \;:=\; \mathcal{H} \nabla_{\mathcal{V} E}\, \mathcal{V} F \;+\; \mathcal{V} \nabla_{\mathcal{V} E}\, \mathcal{H} F \,, \qquad
A_E F \;:=\; \mathcal{V} \nabla_{\mathcal{H} E}\, \mathcal{H} F \;+\; \mathcal{H} \nabla_{\mathcal{H} E}\, \mathcal{V} F \,:
\]
$T$ is the obstruction to the fibers being totally geodesic, and $A$ the obstruction to the integrability of $\mathcal{H}$. Throughout, $\xi^* := \xi / H$ denotes the unit vertical (Reeb) field.  For $\check{g}_s = \sum_i F_i ( s )^2 g_i$, the leaf submersion $\pi_s : ( P_s, g_s ) \to ( N, \check{g}_s )$ is Riemannian while the total submersion $\pi : ( M^0, g ) \to (N, \check{g}_s)$ is horizontally conformal, both with totally geodesic fibers. Thus, each tensor $T$ is vanishing while $A$ for each map is computed in \cite{FT}. In particular, for basic horizontal $X, Y$,

\begin{equation}\label{eq:A-leaf}
	A^{\pi_s}_X Y \;=\; -\frac{H}{2}\sum_{i=1}^r q_i\,\pi^*\omega_i(X,Y)\;\xi^*, \qquad A^{\pi_s}_X\, \xi^* \;=\; \frac{H}{2}\sum_{i=1}^r \frac{q_i}{F_i^2}\; J_i X_i \, .
\end{equation}
Equivalently, one can write that $A^{\pi_s}_X Y = -\tfrac{1}{2}\, d\theta(X,Y)\,\xi$. \\ 

For the total submersion, 
since $A_X Y \in \mathcal{V} = \mathbb{R}\partial_s \oplus \mathbb{R}\xi$ and $\{\partial_s,\, \xi/H\}$ is a $g$-orthonormal basis of $\mathcal{V}$, for basic horizontal $X, Y$:
\begin{align}\label{eq:A-total}
	A^{\pi}_X Y =\;\;\; & -\frac{H}{2}\sum_{i=1}^r q_i\,\pi^*g_i(X,Y)\;\partial_s \;-\; \frac{1}{2}\sum_{i=1}^r \kappa_i\,\pi^*g_i(X,Y)\;\partial_s \;-\; \frac{H}{2}\sum_{i=1}^r q_i\,\pi^*\omega_i(X,Y)\;\xi^*\\
	\label{eq:A-total-kahler}
	\underset{K\"{a}hler}{=}& -\frac{H}{2}\sum_i q_i\,\pi^*g_i(X,Y)\;\partial_s \;-\; \frac{H}{2}\sum_i q_i\,\pi^*\omega_i(X,Y)\;\xi^*.
\end{align}
The mixed components follow by metric-duality such that  $g(A^{\pi}_X\partial_s, Y) = -g(A^{\pi}_XY,\partial_s)$. There is a similar identity for $\xi^*$. \\



The formulas above are phrased in the polar-type frame $\{\partial_s, \xi^*\}$, which does not extend across a closing stratum $Q$. 
Thus, to study the behavior of those tensors at a closing we introduce the even variable $\rho := s^2$ (writing $\dot{\varphi} := d\varphi / d\rho$ for $\rho$-derivatives) and 
\[
E \;:=\; s\, \partial_s \;=\; \tfrac{1}{2}\, \nabla \rho \,.
\]
It is immediate that $E$ and $\zeta$ are smooth vector fields on $\widehat{M}$, both vanish along $Q$, and together they span the vertical distribution $\mathcal{V} = \mathbb{R}\partial_s \oplus \mathbb{R}\xi$ over $M^0$.\\

If we are in case (1) of Lemma \ref{lem:bolt-nut-dichotomy}, at the compactification of $M^0$, write $H(s) = s\, \varphi(s^2)$ with $\varphi(0) = 1$ and $F_i(s)^2 = \psi_i(s^2)$ with $\psi_i(0) > 0$. Then 
for basic horizontal $X, Y$,
\begin{align}
	\label{eq:A-cartesian}
	A^{\pi}_X Y &\;=\; -\sum_{i=1}^r \dot{\psi}_i\; \pi^* g_i(X_i, Y_i)\; E \;-\; \frac{1}{2} \sum_{i=1}^r q_i\; \pi^* \omega_i(X_i, Y_i)\; \xi \,,\\
	\label{eq:A-cartesian-mixed}
	A^{\pi}_X E &\;=\; \rho \sum_{i=1}^r \frac{\dot{\psi}_i}{\psi_i}\; X_i \,, \qquad A^{\pi}_X\, \xi \;=\; \frac{\rho\, \varphi^2}{2} \sum_{i=1}^r \frac{q_i}{\psi_i}\; J_i X_i \,.
\end{align}

Next, we assume case (2) of Lemma \ref{lem:bolt-nut-dichotomy} with $r \geq 2$. The submersion $\widehat{\pi}' : \widehat{M} \to N'$ has, over $M^0$, vertical and horizontal distributions
\[
\mathcal{V}' \;=\; \mathbb{R}\partial_s \,\oplus\, \mathbb{R}\xi \,\oplus\, \mathcal{H}_1, \qquad \mathcal{H}' \;=\; \bigoplus_{k \geq 2} \mathcal{H}_k \, .
\]
We record the following whose proof is given in the Appendix.
\begin{proposition}\label{prop:residual-oneill}
	For the submersion $\widehat{\pi}' : \widehat{M} \to N'$, we have:
	
	\emph{(i)} $T^{\widehat{\pi}'} \equiv 0$.
	
	\emph{(ii)} For basic $X, Y \in \mathcal{H}'$, the $A$-tensor has no $\mathcal{H}_1$-component and is the $k \geq 2$ part of \eqref{eq:A-total}.
\end{proposition}
\noindent Decomposing a horizontal vector $X = \sum_i X_i$ into its $\mathcal{H}_i$-components and noting that $g(X_i, Y_i) = F_i^2\, \pi^*g_i(X,Y)$, we can rewrite 
\begin{equation}
	\label{eq:A-residual}
	A^{\widehat{\pi}'}_X Y \underset{K\"{a}hler}{=}\;  - \sum_{k \geq 2} g ( X_k, Y_k )\; \nabla \log F_k \;-\; \frac{H}{2}\, d\theta' ( X, Y )\; \xi^* \,,
\end{equation}
where $d\theta' := \sum_{k \geq 2} q_k\, \pi^* \omega_k$ is the residual curvature and $\nabla \log F_k = \tfrac{q_k h}{2 F_k^2}\, \partial_s$. 

At the compactification of $M^0$ (see Proposition \ref{prop:nut-closing}), write $H(s) = s\, \varphi(s^2)$, $F_1(s)^2 = s^2\, \psi_1(s^2)$, and $F_k(s)^2 = \psi_k(s^2)$ for $k \geq 2$. Thus, we have
\begin{align*}
	\varphi(0) &= 1,~~~ \psi_1(0) = \tfrac{1}{2}, ~~~\psi_k(0) > 0,\\
	\dot{\psi}_k &= \tfrac{1}{2}\, q_k\, \varphi ~~~ \text{for $k \geq 2$}. 
\end{align*} 
Consequently, for basic $X, Y \in \mathcal{H}'$,
\begin{align}
	A^{\widehat{\pi}'}_X Y &\;=\; -\sum_{k \geq 2} \dot{\psi}_k\; \pi^* g_k(X_k, Y_k)\; E \;-\; \frac{1}{2} \sum_{k \geq 2} q_k\; \pi^* \omega_k(X_k, Y_k)\; \xi \,\nonumber\\
\label{eq:A-residual-cartesian}
	&\;=\; -\frac{\varphi}{2} \sum_{k \geq 2} q_k\, \pi^* g_k ( X_k, Y_k )\; E \;-\; \frac{s\, \varphi}{2} \sum_{k \geq 2} q_k\, \pi^* \omega_k ( X_k, Y_k )\; \xi^* \,.
\end{align}

The following will be important for our analysis. For convenience, let $A'=A^{\widehat{\pi}'}$.
\begin{corollary}\label{cor:norm-A-residual}
	Assuming case (2) of Lemma \ref{lem:bolt-nut-dichotomy} and the K\"{a}hler condition, we have
	\begin{equation}\label{eq:norm-A-residual}
		\big| A' \big|^2_g \;\leq\; C ( r, n_k ) \sum_{k \geq 2} \frac{\big| \nabla F_k^2 \big|^2}{F_k^4} \qquad \text{on } \widehat{M} \,,
	\end{equation}
\end{corollary}


\subsection{Curvature of the ansatz metric}\label{subsec:curvature}
Here we record the curvature computation. The calculation is done via the submersion toolkit (O'Neill tensors), Gauss, Codazzi, and Riccati equations (see \cite{FT} for more details).

All non-base curvature of $g_s$ turns out to be carried by the single $2$-form
\begin{equation}\label{eq:Psi}
	\Psi \;:=\; H\, d\theta \;=\; \sum_{i=1}^r q_i H\; \pi^*\omega_i. 
\end{equation}
Then we have the curvature of a leaf $(P, g_s)$ given by:
\begin{equation}\label{eq:leaf-curvature}
	\begin{aligned}
		&\mathrm{Rm}_{g_s}(X, Y, Z, W) \;=\; \sum_{i=1}^r F_i^2\; \pi^*\mathrm{Rm}_{g_i}(X_i, Y_i, Z_i, W_i) \\
		&\qquad \;+\; \frac{1}{4}\Big( 2\, \Psi(X,Y)\, \Psi(Z,W) \;-\; \Psi(Y,Z)\, \Psi(X,W) \;+\; \Psi(X,Z)\, \Psi(Y,W) \Big).
	\end{aligned}
\end{equation}

\begin{proposition}[Riemann curvature of the ansatz metric]\label{prop:riemann-ansatz}
	For the metric \eqref{eq:ansatz-metric}, with $X, Y, Z, W$ basic horizontal, $\xi^* = \xi/H$ the unit Reeb field, and $\nu = \partial_s$:
	
	\emph{(i) Horizontal components.}
	\begin{equation}\label{eq:Rm-horizontal}
		\begin{aligned}
			&\mathrm{Rm}_g(X, Y, Z, W) \;=\; \mathrm{Rm}_{g_s}(X, Y, Z, W) \\
			& \;+\; \frac{1}{4} \sum_{i,j=1}^r (q_i H + \kappa_i)(q_j H + \kappa_j) \Big( \pi^*g_i(X_i, Z_i)\, \pi^*g_j(Y_j, W_j) \;-\; \pi^*g_i(X_i, W_i)\, \pi^*g_j(Y_j, Z_j) \Big),
		\end{aligned}
	\end{equation}
	with $\mathrm{Rm}_{g_s}$ given by \eqref{eq:leaf-curvature}; moreover $\mathrm{Rm}_g(X, Y, Z, \xi^*) = 0$.
	
	\emph{(ii) Vertizontal components.}
	\begin{equation}\label{eq:Rm-vertizontal}
		\mathrm{Rm}_g(X, \xi^*, \xi^*, Y) \;=\; \frac{1}{4} \sum_{i=1}^r \frac{(q_i H)^2}{F_i^2}\; \pi^*g_i(X_i, Y_i) \;-\; \frac{H'}{2H} \sum_{i=1}^r \big( q_i H + \kappa_i \big)\, \pi^*g_i(X_i, Y_i) .
	\end{equation}
	
	\emph{(iii) Codazzi components.} The only non-vanishing components with exactly one $\nu$-slot are
		\begin{equation}\label{eq:Rm-codazzi}
		\mathrm{Rm}_g(X, Y, \xi^*, \nu) \;=\; 2\, \mathrm{Rm}_g(X, \xi^*, Y, \nu) \;=\; \sum_{i=1}^r q_i H \left( \frac{F_i'}{F_i} - \frac{H'}{H} \right) \pi^*\omega_i(X_i, Y_i) ;
	\end{equation}
	in particular $\mathrm{Rm}_g(X, Y, Z, \nu) = 0$ and $\mathrm{Rm}_g(X, \xi^*, \xi^*, \nu) = 0$.
	
	\emph{(iv) Radial components.}
	\begin{align}
			\mathrm{Rm}_g(X, \nu, \nu, Y) &\;=\; -\sum_{i=1}^r F_i F_i''\; \pi^*g_i(X_i, Y_i), \nonumber \\
			\label{eq:Rm-radial}
			\mathrm{Rm}_g(\xi^*, \nu, \nu, \xi^*) &\;=\; -\frac{H''}{H}, \\
			\mathrm{Rm}_g(X, \nu, \nu, \xi^*) &\;=\; 0 .\nonumber
	\end{align}
	
	Together with the symmetries of the curvature tensor, (i)--(iv) determine $\mathrm{Rm}_g$ completely.
	
\end{proposition}

\begin{corollary}[sectional curvatures]\label{cor:sectional-curvatures}
	For the metric \eqref{eq:ansatz-metric}, let $X_i \in \mathcal{H}_i$ and $Y_j \in \mathcal{H}_j$ be $g$-unit horizontal vectors, orthogonal to one another when the blocks coincide, and set	
	\[
	\tau_i \;:=\; F_i^2\; \pi^*\omega_i(X_i, Y_i) \;=\; g(J X_i, Y_i) \;\in\; [-1, 1] \,.
	\]
	Thus, $\tau_i = \pm 1$ iff $Y_i = \pm J X_i$, and $\tau_i = 0$ iff the plane is totally real. Then the sectional curvatures of the frame planes are:
	
	\emph{(i) Radial planes.}
	\[
	K(\nu, \xi^*) \;=\; -\frac{H''}{H} \,, \qquad K(\nu, X_i) \;=\; -\frac{F_i''}{F_i} \,.
	\]
	
	\emph{(ii) Vertizontal planes.}
	\[
	K(X_i, \xi^*) \;=\; \frac{q_i^2 H^2}{4 F_i^4} \;-\; \frac{H'}{H}\,\frac{F_i'}{F_i} \,.
	\]
	
	\emph{(iii) Horizontal planes within a block.}
		\[
	K(X_i, Y_i) \;=\; \frac{K_{g_i}(\pi_* X_i, \pi_* Y_i)}{F_i^2} \;-\; \frac{3\, q_i^2 H^2}{4 F_i^4}\;\tau_i^2 \;-\; \frac{F_i'^2}{F_i^2} \,,
	\]
		where $K_{g_i}$ denotes the sectional curvature of $(N_i, g_i)$.
	
	\emph{(iv) Horizontal planes across blocks.} For $i \neq j$,
	\[
	K(X_i, Y_j) \;=\; -\,\frac{F_i'}{F_i}\;\frac{F_j'}{F_j} \,,
	\]
		the classical doubly-warped-product value: the connection form is invisible on cross-block planes.
	
	
\end{corollary}

\begin{proposition}[Ricci and scalar curvature]\label{prop:ricci-ansatz}
	For the metric \eqref{eq:ansatz-metric}, with $X, Y \in \mathcal{H}$ basic horizontal (block decomposition $X = \sum_i X_i$), $\xi^* = \xi/H$ the unit Reeb field, $\nu = \partial_s$, 
	
	\emph{(i) Vertical components.}
	\begin{align}
		\label{eq:Ric-nu}
		\operatorname{Ric}(\nu, \nu) &\;=\; -\frac{H''}{H} \;-\; \sum_{i=1}^r n_i \left( \frac{q_i H' + \kappa_i'}{F_i^2} \;-\; \frac{(q_i H + \kappa_i)^2}{2 F_i^4} \right),\\
	\label{eq:Ric-xi}
		\operatorname{Ric}(\xi^*, \xi^*) &\;=\; -\frac{H''}{H} \;+\; \sum_{i=1}^r n_i \left( \frac{q_i^2 H^2}{2 F_i^4} \;-\; \frac{H'}{H}\, \frac{q_i H + \kappa_i}{F_i^2} \right).
	\end{align}
	
	\emph{(ii) Mixed components vanish:} $\operatorname{Ric}(\nu, \xi^*) = 0$ and $\operatorname{Ric}(X, \nu) = \operatorname{Ric}(X, \xi^*) = 0$.
	
	\emph{(iii) Horizontal components.} $\operatorname{Ric}(X, Y) = \sum_{i=1}^r \mathcal{R}_i\; \pi^*g_i(X_i, Y_i)$, where
	\begin{equation}\label{eq:Ric-horizontal}
		\mathcal{R}_i \;=\; p_i \;-\; \frac{q_i H' + \kappa_i'}{2} \;-\; \frac{H'}{2H}\,(q_i H + \kappa_i) \;+\; \frac{\kappa_i\,(2 q_i H + \kappa_i)}{2 F_i^2} \;-\; \frac{q_i H + \kappa_i}{2} \sum_{j=1}^r \frac{n_j\,(q_j H + \kappa_j)}{F_j^2} \,.
	\end{equation}
	
	\emph{(iv) Scalar curvature.}
		\begin{equation}\label{eq:scalar-curvature}
		\begin{aligned}
			&\mathrm{R}_g \;=\; -\frac{2H''}{H} \;+\; \sum_{i=1}^r n_i \left( \frac{2 p_i}{F_i^2} \;-\; \frac{q_i^2 H^2}{2 F_i^4} \;-\; \frac{2\,(q_i H' + \kappa_i')}{F_i^2} \;-\; \frac{2H'}{H}\, \frac{q_i H + \kappa_i}{F_i^2} \;+\; \frac{3\,(q_i H + \kappa_i)^2}{2 F_i^4} \right) \\
			&\qquad \;-\; \left( \sum_{i=1}^r \frac{n_i\,(q_i H + \kappa_i)}{F_i^2} \right)^{2} .
		\end{aligned}
	\end{equation}
	
\end{proposition}

\begin{corollary}[The Kähler case]\label{cor:ricci-kahler}
	If $\kappa_i \equiv 0$ for all $i$ (the Kähler condition \eqref{eq:kahler-condition}), then
	
	\begin{equation}\label{eq:ricci-kahler}
		\begin{gathered}
			\operatorname{Ric}(\nu, \nu) \;=\; \operatorname{Ric}(\xi^*, \xi^*) \;=\; -\frac{H''}{H} \;+\; \sum_{i=1}^r n_i \left( \frac{q_i^2 H^2}{2 F_i^4} \;-\; \frac{q_i H'}{F_i^2} \right), \\
			\mathcal{R}_i \;=\; p_i \;-\; q_i H' \;-\; \frac{q_i H^2}{2} \sum_{j=1}^r \frac{n_j\, q_j}{F_j^2} \,,
		\end{gathered}
	\end{equation}
	
	and the scalar curvature becomes
	
	\[
	\mathrm{R}_g \;=\; -\frac{2H''}{H} \;+\; \sum_{i=1}^r n_i \left( \frac{2 p_i}{F_i^2} \;-\; \frac{4\, q_i H'}{F_i^2} \;+\; \frac{q_i^2 H^2}{F_i^4} \right) \;-\; \left( \sum_{i=1}^r \frac{n_i\, q_i H}{F_i^2} \right)^{2} .
	\]
	
\end{corollary}

\subsection{Periods and the first Chern class of the closings}\label{subsec:periods-chern}

The topological probe transverse to the strata is a single $2$-sphere. Fix $x \in P$ and let $C \subset \widehat{M}$ be the closure of the annulus $(0, \ell) \times \big( \mathrm{U}(1) \cdot x \big)$ swept out by the circle orbit through $x$. At a bolt end the orbit bounds the polar disc of the totally geodesic fiber $\widehat{\Sigma}_x$ of \S\ref{subsec:closing}; at a nut-type end, it is a Hopf fiber of the collapsing $S^{2m+1}$, whose radial cone is a complex line through the nut point of the fiberwise $\mathbb{C}^{m+1}$. Hence $C \cong S^2$ is an embedded sphere meeting each stratum in exactly one point.

Define the \emph{Ricci potential}
	\[
	\Theta \;:=\; -\,H' \;-\; \frac{H^2}{2} \sum_{j=1}^r \frac{n_j\, q_j}{F_j^2} \,,
	\]
and give each end the integer weight $d_0$ (resp. $d_\ell$), equal to $1$ at a bolt and to $m + 1$ at a nut-type closing with collapsing factor $\mathbb{CP}^m$. Write $J_0, J_\ell \subseteq \{ 1, \dots, r \}$ for the sets of blocks collapsed at each end ($\varnothing$ at a bolt). The K\"{a}hler components \eqref{eq:ricci-kahler} of Corollary \ref{cor:ricci-kahler} read, verbatim,
	\[
	\operatorname{Ric}(\nu, \nu) \;=\; \operatorname{Ric}(\xi^*, \xi^*) \;=\; \frac{\Theta'}{H} \,, \qquad \mathcal{R}_i \;=\; p_i \;+\; q_i\, \Theta \,.
	\]
Consequently, on $M^0$, the Ricci form is
	\[
	\varrho \;=\; \sum_{i=1}^r p_i\, \pi^* \omega_i \;+\; d \big( \Theta\, \theta \big) \,.
	\]
	$\Theta$ extends continuously to $[0, \ell]$ with
		\[
	\Theta(0) \;=\; -\, d_0 \,, \qquad \Theta(\ell) \;=\; +\, d_\ell \,.
	\]
	Correspondingly, the horizontal boundary values at the two strata become, under \eqref{eq:kahler-condition}, the {integers}
		\[
	\mathcal{R}_i(0) \;=\; p_i \;-\; d_0\, q_i \quad ( i \notin J_0 ) \,, \qquad \mathcal{R}_i(\ell) \;=\; p_i \;+\; d_\ell\, q_i \quad ( i \notin J_\ell ) \,.
	\]
	Therefore, the first Chern class pairs with the sphere $C$ and restricts to the strata by
		\begin{align*}
				\big\langle c_1(\widehat{M}),\, [C] \big\rangle &\;=\; d_0 + d_\ell \,, \\
			c_1(\widehat{M}) \big|_{Q_0} &\;=\; \sum_{i \notin J_0} \big( p_i - d_0\, q_i \big)\, a_i \,, \\
			c_1(\widehat{M}) \big|_{Q_\ell} &\;=\; \sum_{i \notin J_\ell} \big( p_i + d_\ell\, q_i \big)\, a_i \,,
	\end{align*}
	with the classes $a_i$ being pulled back to the stratum. Comparing with adjunction 
	$$c_1(\widehat{M})|_Q = c_1(Q) + c_1(\mathcal{N}_Q)~~ \text{ and } ~~c_1(Q_0) = \sum_{i \notin J_0} p_i\, a_i$$
	identifies the normal bundles of the strata: 
	$$c_1(\mathcal{N}_{Q_0}) = -\, d_0 \sum_{i \notin J_0} q_i\, a_i ~~\text{ and  } ~~c_1(\mathcal{N}_{Q_\ell}) = +\, d_\ell \sum_{i \notin J_\ell} q_i\, a_i.$$
	These periods determine $c_1(\widehat{M})$ completely.
	
	\begin{remarkx} Writing $[Q]$ for the Poincaré dual of a stratum that is a divisor. For {two bolts},
	\[
	c_1(\widehat{M}) \;=\; [Q_0] \;+\; [Q_\ell] \;+\; \widehat{\pi}^*\, c_1(N) \;=\; 2\, [Q_0] \;+\; \widehat{\pi}^* \sum_{i=1}^r \big( p_i + q_i \big)\, a_i \,.
	\]
	For {a nut-type closing and a bolt},
	\[
	c_1(\widehat{M}) \;=\; \big( m + 2 \big)\, [Q_\ell] \;+\; \big( \widehat{\pi}' \big)^* \sum_{k \geq 2} \big( p_k - ( m + 1 )\, q_k \big)\, a_k \,.
	\]
	Finally, for {two nut-type closings} $c_1(\widehat{M})$ is the unique class with the periods described above-- in particular for $r = 2$, where $\widehat{M} \cong \mathbb{CP}^{m_0 + m_\ell + 1}$, it is $m_0 + m_\ell + 2$ times the hyperplane class: the maximal Fano index, recovered from the Ricci tensor of an arbitrary Kähler ansatz metric.
\end{remarkx}

\subsection{The Ricci flow of ansatz metrics}\label{subsec:ricci-flow-ansatz}
The curvature formulas of \S\ref{subsec:curvature} refer to the arclength parameter $s$ of one fixed metric. Under a geometric evolution the metric --- and with it the arclength --- changes in time, so $s$ is no longer available as a time-independent coordinate. We therefore fix once and for all a background coordinate $u$ on $I$ and consider one-parameter families of ansatz metrics
\begin{equation}\label{eq:ansatz-flow-form}
	g(u, t) \;=\; a(u,t)^2\, du \otimes du \;+\; h(u,t)^2\, \theta \otimes \theta \;+\; \sum_{i=1}^r f_i(u,t)^2\, \pi^* g_i \,,
\end{equation}
with smooth positive coefficients $a, h, f_1, \dots, f_r$. At each fixed time, $s(u,t) := \int^u a(w,t)\, dw$ is the arclength, and \eqref{eq:ansatz-flow-form} is the metric \eqref{eq:ansatz-metric} with $H = h$, $F_i = f_i$ regarded as functions of $s$; conversely, every ansatz metric takes the form \eqref{eq:ansatz-flow-form} with $a = \partial s / \partial u$ once the coordinate $u$ is fixed. The dictionary between the two descriptions is the first-order operator
\[
D \;:=\; \frac{1}{a}\, \frac{\partial}{\partial u} \;=\; \frac{\partial}{\partial s} \,, \qquad D^2 \varphi \;=\; \frac{\varphi_{uu}}{a^2} \;-\; \frac{a_u\, \varphi_u}{a^3} \,,
\]
so that $\nu = a^{-1} \partial_u$ is the unit normal of the leaves, $\xi^* = \xi/h$ is the unit Reeb field, and the Kähler deficit reads $\kappa_i = D(f_i^2) - q_i h$. 

Now let the family \eqref{eq:ansatz-flow-form} evolve by the Ricci flow,
\[
\frac{\partial g}{\partial t} \;=\; -2\, \operatorname{Ric}\big( g(t) \big) \,.
\]
By the block-diagonality of the Ricci tensor in these variables (Proposition \ref{prop:ricci-ansatz} rewritten in the fixed coordinate $u$; see \cite{FT}), the right-hand side is block-diagonal of the same shape as \eqref{eq:ansatz-flow-form}. Just matching the coefficients converts the tensor equation into a system of scalar equations, see \cite{FT} for the full justification,
\begin{align}
		\partial_t a \;&=\; \left( \frac{D^2 h}{h} \;+\; \sum_{i=1}^r 2 n_i\, \frac{D^2 f_i}{f_i} \right) a \,,\nonumber\\
		\label{eq:flow-system}
		\partial_t h \;&=\; D^2 h \;+\; 2\, Dh \sum_{i=1}^r n_i\, \frac{D f_i}{f_i} \;-\; \sum_{i=1}^r \frac{n_i\, q_i^2\, h^3}{2 f_i^4} \,,\\
		\partial_t f_i \;&=\; D^2 f_i \;+\; D f_i \left( \frac{Dh}{h} \;+\; \sum_{j=1}^r 2 n_j\, \frac{D f_j}{f_j} \right) \;-\; \frac{(D f_i)^2}{f_i} \;-\; \frac{p_i}{f_i} \;+\; \frac{q_i^2\, h^2}{2 f_i^3} \,, \qquad 1 \leq i \leq r \,.\nonumber
\end{align}

Every scalar quantity naturally attached to the ansatz --- warping functions, curvatures, soliton potentials --- depends only on the coordinate of $I$ and, along a flow, on time. Call a function $Q$ on $M^0$ (or on $M^0 \times [0,T)$) \textbf{radial} if $Q = Q(s, t)$. We record the three basic differential operators of $(M^0, g)$ on radial functions --- gradient, squared gradient norm, Laplacian --- first for a general ansatz metric \eqref{eq:ansatz-metric}, in terms of the deficit functions $\kappa_i$, then in the Kähler case as a corollary. Throughout, all operators are those of a fixed metric $g = g(t)$ of the form \eqref{eq:ansatz-metric}, so time enters only as a parameter and ${}'$ denotes $\partial / \partial s$; we use the analyst's sign convention $\Delta Q = \operatorname{div}( \nabla Q ) = \operatorname{tr} \nabla^2 Q$. 

\begin{proposition}[Radial calculus]\label{prop:radial-calculus}
	Let $Q = Q(s)$ be a radial function on $(M^0, g)$ with $g$ as in \eqref{eq:ansatz-metric}, and let $\nu = \partial_s$ be the unit normal of the leaves. Then
	\begin{align}\label{eq:radial-gradient}
		\nabla Q &\;=\; Q'\, \nu \,, \qquad \big| \nabla Q \big|^2 \;=\; \big( Q' \big)^2 \,,\\
		\label{eq:radial-laplacian}
		\Delta Q &\;=\; Q'' \;+\; Q' \left( \frac{H'}{H} \;+\; \sum_{i=1}^r \frac{n_i\, \big( q_i H + \kappa_i \big)}{F_i^2} \right) \;=\; \frac{\big( V Q' \big)'}{V} \,, \qquad V \;:=\; H \prod_{i=1}^r F_i^{2 n_i} \,.
	\end{align}	
\end{proposition}

\begin{remarkx}
	The function $V$ is the \emph{relative volume density} of the leaves: $d\mathrm{vol}_{g_s} = \tfrac{V(s)}{V(s_0)}\, d\mathrm{vol}_{g_{s_0}}$ under the identification $P_s \cong P \cong P_{s_0}$. Thus, on radial functions, $\Delta$ is the one-dimensional Sturm--Liouville operator with weight $V$. All the geometry enters through the drift $(\log V)'$: the leaves' mean curvature.
\end{remarkx}

\begin{corollary}[Kähler case]\label{cor:kahler-laplacian}
	If $\kappa_i \equiv 0$ (the Kähler condition \eqref{eq:kahler-condition}), then for radial $Q$
	\begin{equation}\label{eq:kahler-laplacian}
		\Delta Q \;=\; Q'' \;+\; Q' \left( \frac{H'}{H} \;+\; H \sum_{i=1}^r \frac{n_i\, q_i}{F_i^2} \right).
	\end{equation}
\end{corollary}

The following was observed in \cite{FT}.
\begin{proposition}[The heat equation for $f_i^2$ in the Kähler case]\label{prop:heat-equation-fi}
	Suppose $\kappa_i \equiv 0$ for all $i$, and let $g(t)$ be a Ricci flow preserving the ansatz with fixed coordinate $u$. 
	Then each squared warping function evolves by the {linear heat equation with constant forcing}
	\begin{equation}\label{eq:heat-fi}
		\partial_t \big( f_i^2 \big) \;=\; \Delta f_i^2 \;-\; 2 p_i \,, \qquad i = 1, \dots, r \,.
	\end{equation}
\end{proposition}
Immediately, there is a gradient estimate of Li--Yau type \cite{LY} for the fiber potentials. 

\begin{proposition}[evolution of the gradient; Li--Yau estimate]\label{prop:li-yau}
	Along the flow, each $u := f_i^2$ satisfies
		\begin{equation}\label{eq:gradient-evolution}
		\big( \partial_t - \Delta \big)\, \big| \nabla u \big|^2 \;=\; -\, 2\, \big| \nabla^2 u \big|^2 \,,
	\end{equation}
	and consequently the Li--Yau quantity \cite{LY} $| \nabla u |^2 / u$ is uniformly bounded:
		\begin{equation}\label{eq:li-yau}
		\sup_{\widehat{M}} \frac{| \nabla u |^2}{u} (\cdot, t) \;\leq\; C_0 \;:=\; \max \left\{ \sup_{\widehat{M}} \frac{| \nabla u |^2}{u} (\cdot, 0) \,, \; 4\, p_i \right\} \qquad \text{for all } t \in [0, T) \,.
	\end{equation}
\end{proposition}
We now recall the main theorem from \cite{FT}. $\widehat{M}$ is a {compact two-bolt closing} of \S\ref{subsec:closing} --- the $S^2$-bundle over $N$, with bolts at both finite ends, which we denote $Q_0 = \{ s = 0 \}$ and $Q_\ell = \{ s = \ell \}$ --- and $g(t)$, $t \in [0, T)$, is a Ricci flow of ansatz metrics \eqref{eq:ansatz-metric} satisfying the Kähler condition \eqref{eq:kahler-condition}, which is preserved by the flow (\S\ref{subsec:ricci-flow-ansatz}), with $T < \infty$ the maximal existence time.


\begin{definition}[Type I and Type II]\label{def:type-I-II}
	A Ricci flow $g(t)$ on a compact manifold with finite maximal existence time $T < \infty$ develops a \emph{Type I} singularity if
	\[
	\sup_{\widehat{M} \times [0, T)} \big( T - t \big)\, \big| \mathrm{Rm}\big( g(t) \big) \big| \;<\; \infty \,,
	\]
	and a \emph{Type II} singularity if this supremum is infinite. 
\end{definition}

\begin{theorem}[the K\"ahler two-bolt flow is Type I, \cite{FT}]\label{thm:two-bolt-typeI}
	Let $\widehat{M}$ be a compact two-bolt closing and let $g(t)$, $t \in [0, T)$, be a Ricci flow of ansatz metrics \eqref{eq:ansatz-metric} satisfying the Kähler condition \eqref{eq:kahler-condition}, with $T < \infty$ maximal. Then the singularity at $T$ is of Type I.
	
\end{theorem}

The proof in \cite{FT} has four components: exact affine laws for the fiber potentials along the bolts; a Li--Yau-type gradient bound; their combination into the decay $| A |^2 \leq C / ( T - t )$ of the O'Neill tensor, which kills $A$ along any Type II rescaling; and the exclusion of the resulting cigar limit by Perelman's no-local-collapsing theorem. 

\section{Type I singularity of the nut--bolt flow}\label{sec:typeI}
We now investigate the other compact Kähler closing of \S\ref{subsec:closing}. Throughout this section $\widehat{M}$ is a \textit{compact nut--bolt closing}: a {nut-type collapse} at $s = 0$ (the sphere $S^{2m+1}$ collapses there) and a \emph{bolt} at $s = \ell$. 
Thus, $( N_1, g_1 ) = ( \mathbb{CP}^m, 2\, g_{FS} )$ with $p_1 = m + 1$ and $| q_1 | = 1$; we write $N' := \prod_{k \geq 2} N_k$, so that the closing strata are $Q_0 = \{ s = 0 \} \cong N'$ and $Q_\ell = \{ s = \ell \} \cong N = \mathbb{CP}^m \times N'$, and $\widehat{M} \cong \mathbb{P} ( \mathcal{O} \oplus E )$ is the $\mathbb{CP}^{m+1}$-bundle over $N'$ of \S\ref{subsec:closing}. We assume $r \geq 2$; for $r = 1$ ($N'$ a point, $\widehat{M} \cong \mathbb{CP}^{m+1}$) everything below degenerates gracefully to the classical Calabi-symmetric flow. 

The goal of the section is to obtain an analogue of Theorem \ref{thm:two-bolt-typeI}: \emph{the singularity at} $T$ \emph{is of Type I}. The structural novelty is the submersion of record. Since $\pi$ does not extend across $Q_0$ (\S\ref{subsec:closing}), its role is taken by the {residual bundle} of \S\ref{subsec:oneill},
\[
\widehat{\pi}' \;:\; \widehat{M} \;\longrightarrow\; N' \,.
\]

\subsection{Affine laws at the two strata and the linear lower bound}\label{subsec:affine-laws}

Crucial to our analysis is the evolution of warping functions at each closing stratum. Recall from Proposition \ref{prop:heat-equation-fi} that in the Kähler case each $u = f_i^2$ solves the linear heat equation with constant forcing, $( \partial_t - \Delta )\, f_i^2 = -2 p_i$. At a bolt this equation closes up into an ODE:

\begin{proposition}[affine laws at both strata]\label{prop:affine-laws}
	Along the flow, for every $i$:
	
	{(a) At the bolt,} 
	\[
	f_i^2 \big|_{Q_\ell} (t) \;=\; f_i^2 \big|_{Q_\ell} (0) \;-\; 2 \big( q_i + p_i \big)\, t \,.
	\]
	In particular, for $i = 1$, by \eqref{eq:F1-int-H} the circle mass is exactly affine,
	\begin{equation}\label{eq:circle-mass-affine}
		L (t) \;=\; L (0) \;-\; 2 ( m + 2 )\, t \,, \qquad \text{so} \quad T \;\leq\; T_F' \;:=\; \frac{L (0)}{2 ( m + 2 )} \,.
	\end{equation}
	
	{(b) At the nut-type closing,}
	\begin{equation}\label{eq:nut-affine}
		f_i^2 \big|_{Q_0} (t) \;=\; f_i^2 \big|_{Q_0} (0) \;+\; 2 \big( ( m + 1 )\, q_i \;-\; p_i \big)\, t \,.
	\end{equation}
	
	{(c) (Linear lower bound for the surviving blocks)} There is $c > 0$ with
	\begin{equation}\label{eq:linear-lower-bound}
		f_k^2 (\cdot, t) \;\geq\; c\, \big( T - t \big) \qquad \text{on } \widehat{M} \times [0, T) \,, \quad \text{for every } k \geq 2 \,,
	\end{equation}
	
	
\end{proposition}

\begin{proof}
	\emph{(a)} The calculation is local near a bolt and carried out in \cite{FT}.
	
	
	\emph{(b)} Each $f_i^2$ is radial and constant on the slice $Q_0$, so $\partial_t \big( f_i^2 |_{Q_0} \big) = \big( \Delta f_i^2 \big) \big|_{Q_0} - 2 p_i$ by \eqref{eq:heat-fi}, and \eqref{eq:kahler-laplacian} gives, in the coordinate $s$, 	
	\[
	\Delta f_i^2 \;=\; q_i \left( 2\, h' \;+\; h^2 \sum_{j = 1}^r \frac{n_j\, q_j}{f_j^2} \right).
	\]
	The right-side expression is smooth across $Q_0$ by the closing condition of \S\ref{subsec:oneill}. Along $Q_0$, $h' \to \varphi ( 0 ) = 1$; for the collapsing block ($n_1 = m$, $q_1 = 1$), $h^2 / f_1^2 = \varphi^2 / \psi_1 \to 2$ 
	while $h^2 / f_j^2 \to 0$ for $j \geq 2$. Hence
	
	\[
	\big( \Delta f_i^2 \big) \big|_{Q_0} \;=\; q_i\, \big( 2 \;+\; 2\, m \big) \;=\; 2 ( m + 1 )\, q_i \,,
	\]
	
	which integrates to \eqref{eq:nut-affine}. 
	
	\emph{(c)} For $k\geq 2$, by equation \ref{eq:Fk-affine}, each $f_k^2$ is monotone in $s$, so
	
	\[
	\inf_{\widehat{M}} f_k^2 (\cdot, t) \;=\; \min \Big\{ f_k^2 \big|_{Q_0} (t) \,, \; f_k^2 \big|_{Q_\ell} (t) \Big\} \,.
	\]
	This is a minimum of two affine functions of $t$, each positive on $[0, T)$. An affine function $a (t) > 0$ on $[0, T)$ satisfies $a (t) \geq c\, ( T - t )$ there: if $\dot{a} \geq 0$, take $c = a (0) / T$; if $\dot{a} < 0$, then $a (t) = a (T) + | \dot{a} |\, ( T - t ) \geq | \dot{a} |\, ( T - t )$, with $a (T) := \lim_{t \to T^-} a (t) \geq 0$. Take for $c$ the least of the resulting constants over both strata and the finitely many $k \geq 2$.
\end{proof}

\begin{remarkx}
		
		Cohomologically, the affine law is the pointwise shadow of the linear motion of the Kähler class. With our convention $\partial_t g = -2 \operatorname{Ric}$, one has $[ \omega(t) ] = [ \omega_0 ] - 4 \pi t\, c_1 ( \widehat{M} )$. Furthermore, $2 \pi L (t) = [ \omega (t) ] \cdot [ \Lambda ]$ for a projective line $\Lambda$ in a $\mathbb{CP}^{m+1}$-fiber joining the two strata; the normal bundle of a fiber is trivial, so $c_1 ( \widehat{M} ) \cdot [ \Lambda ] = c_1 ( \mathbb{CP}^{m+1} ) \cdot [ \Lambda ] = m + 2$, and $[ \omega (t) ] = [ \omega_0 ] - 4 \pi t\, c_1 ( \widehat{M} )$ returns the rate $- 2 ( m + 2 )$. 
		
\end{remarkx}
\subsection{The residual O'Neill tensor along a Type II rescaling}\label{subsec:A-residual-decay}


Here we give a type I estimate for $|A'|$, the tensor discussed in \ref{subsec:oneill}, and the residual-base curvature.
\begin{proposition}\label{prop:A-residual-decay}
	There is a constant $C = C( g_0, T )$ such that
	\begin{equation}\label{eq:A-residual-typeI}
		\big| A' \big|^2_{g (t)} \;\leq\; \frac{C}{T - t} \qquad \text{on } \widehat{M} \times [0, T) \,.
	\end{equation}
	And the same bound holds for the sectional curvatures of $\mathcal{H}'$-planes of $g (t)$ and for those of the evolving residual base metric $\check{g}'_t = \sum_{k \geq 2} f_k^2\, g_k$ on $N'$. Consequently, along any rescaling $g_j (t) = K_j\, g ( t_j + K_j^{-1} t )$ with $K_j ( T - t_j ) \to \infty$,
	
	\begin{equation}\label{eq:A-residual-rescaled}
		\big| A' \big|^2_{g_j (t)} \;=\; K_j^{-1}\, \big| A' \big|^2_{g ( t_j + K_j^{-1} t )} \;\leq\; \frac{C}{K_j ( T - t_j ) - t} \;\longrightarrow\; 0 \qquad ( j \to \infty ) \,,
	\end{equation}
	locally uniformly in $t \in \mathbb{R}$, and likewise for the rescaled $\mathcal{H}'$-sectional and residual-base curvatures.
	
\end{proposition}

\begin{proof}
	The heat equation \eqref{eq:heat-fi} holds for every block regardless of the closing type, so the Li--Yau estimate \eqref{eq:li-yau} applies to $u = f_k^2$ for each $k \geq 2$: $| \nabla f_k^2 |^2 \leq C_0\, f_k^2$. Inserting this into \eqref{eq:norm-A-residual} and using the linear lower bound \eqref{eq:linear-lower-bound},
	\[
	\big| A' \big|^2_g \;\leq\; C \sum_{k \geq 2} \frac{\big| \nabla f_k^2 \big|^2}{f_k^4} \;\leq\; C\, C_0 \sum_{k \geq 2} \frac{1}{f_k^2} \;\leq\; \frac{C}{T - t} \,.
	\]
	
	For the $\mathcal{H}'$-curvatures, restricted to the blocks $k \geq 2$: the base contribution is blockwise $f_k^2\, \pi^* \mathrm{Rm}_{g_k}$, with sectional curvatures $K_{g_k} / f_k^2 \leq C / ( T - t )$ by \eqref{eq:linear-lower-bound}, and the $\Psi$- and $\mathrm{II}$-corrections are quadratic in $q_k h / f_k^2$, of size $| \nabla f_k^2 |^2 / f_k^4 \leq C_0 / f_k^2 \leq C / ( T - t )$. Finally $A'$ carries one derivative of the metric, so $| A' |^2_{K g} = K^{-1} | A' |^2_g$, and sectional curvatures scale likewise. Finally, writing $T - ( t_j + K_j^{-1} t ) = K_j^{-1} \big( K_j ( T - t_j ) - t \big)$ gives \eqref{eq:A-residual-rescaled}.
\end{proof}

Suppose the singularity is Type II. By Hamilton's point-picking for Type II singularities, there are $( x_j, t_j )$ with $t_j \to T$ and $K_j := | \mathrm{Rm} | ( x_j, t_j ) \to \infty$ such that $K_j ( T - t_j ) \to \infty$ and the rescaled flows $g_j (t) = K_j\, g ( t_j + K_j^{-1} t )$ satisfy $\sup | \mathrm{Rm} |_{g_j} \leq 1 + o(1)$ on time intervals $( - \alpha_j, \beta_j )$ with $\alpha_j, \beta_j \to \infty$, with $| \mathrm{Rm} |_{g_j} ( x_j, 0 ) = 1$. Since $\widehat{M}$ is compact and $T < \infty$, Perelman's no-local-collapsing theorem provides a uniform $\kappa > 0$ such that all the $g_j$ are $\kappa$-non-collapsed at all scales $\lesssim 1$; injectivity radius bounds follow. Consequently, Hamilton's compactness theorem yields a sub-sequential pointed $C^\infty$ Cheeger--Gromov limit
\[
\big( \widehat{M}, \; g_j (t), \; x_j \big) \;\longrightarrow\; \big( M_\infty, \; g_\infty (t), \; x_\infty \big) \,.
\]
This is a complete {eternal} Ricci flow, $t \in \mathbb{R}$, with $| \mathrm{Rm} |_{g_\infty} \leq 1$ on all of space-time, equality at $( x_\infty, 0 )$, and $\kappa$-noncollapsed at all scales.

	
\begin{proposition}[Type II blow-up limits split along $\mathcal{V}' \oplus \mathcal{H}'$]\label{prop:typeII-splitting}
	Suppose the singularity at $T$ is of Type II. 
	
	{(i) The limit splits along} $\mathcal{V}' \oplus \mathcal{H}'$\emph{.} The distributions subconverge to complementary orthogonal distributions $\mathcal{V}'_\infty$ (rank $2 m + 2$) and $\mathcal{H}'_\infty$ (rank $2 n'$, where $n' := \sum_{k \geq 2} n_k$) whose O'Neill tensors vanish: both are parallel, and by the de Rham decomposition theorem the universal cover splits isometrically and holomorphically, for each $t$,
	
	\begin{equation}\label{eq:typeII-splitting}
		\widetilde{M}_\infty \;=\; \Phi \times Y \,, \qquad \widetilde{g}_\infty (t) \;=\; g^\Phi (t) \,\oplus\, g^Y (t) \,, \qquad \dim_{\mathbb{R}} \Phi \;=\; 2 m + 2 \,,
	\end{equation}
	
	with each factor a complete eternal Kähler--Ricci flow.\\
	
	{(ii) The} $\mathcal{H}'$\emph{-factor is flat.} $( Y, g^Y (t) )$ is flat and static, hence $Y = \mathbb{C}^{n'}$, being simply connected; all curvature is carried by the fiber factor, $| \mathrm{Rm}_{g^\Phi} | \leq 1$ with equality at $( x_\infty, 0 )$, so $\Phi$ is nonflat.\\
	
	{(iv) The fiber factor is a limit of Calabi fibers.} Under \eqref{eq:typeII-splitting}, $\Phi$ is the leaf of $\mathcal{V}'_\infty$ through $x_\infty$: the pointed limit of the rescaled totally geodesic fibers $\big( ( \widehat{\pi}' )^{-1} \big( \widehat{\pi}' ( x_j ) \big), \, g_j (t) \big)$ --- each a Calabi-symmetric Kähler metric on $\mathbb{CP}^{m+1}$.
	
\end{proposition}

\begin{proof}
\emph{(i)} The $g$-orthogonal splitting $T \widehat{M} = \mathcal{V}' \oplus \mathcal{H}'$ is orthogonal for every ansatz metric, hence for every $g_j (t)$, and is scale-invariant. We represent the sub-bundle by its orthogonal projection operator, which is a smooth $(1,1)$-tensor field:
\begin{equation*}
	P \in \Gamma(T^*\mathcal{V}' \otimes T\mathcal{V}')\text{ satisfying } P^2 = P \quad \text{and} \quad g(P \cdot, \cdot) = g(\cdot, P \cdot).
\end{equation*}
 The full second fundamental form of the splitting is the pair $( T, A )$. Specifically, we have the point-wise identity:
\begin{equation}
	\label{eq:projection-norm}
	\|\nabla P\|^2 = 2 \left( \|T\|^2 + \|A\|^2 \right)
\end{equation}

Under smooth pointed Cheeger–Gromov convergence, the pulled-back metrics $\bar{g}_j = \phi_j^* g_j$ and their Levi-Civita connections $\bar{\nabla}^j$ converge in $C^\infty$ on compact subsets of $M_\infty$. The convergence of the distributions $\mathcal{V}_j$ is mathematically equivalent to the convergence of the projection tensors $P_j$ as smooth tensor fields. Here $T \equiv 0$ identically and $| A |_{g_j} \to 0$ locally uniformly by \eqref{eq:A-residual-rescaled}. 


Via Shi's estimates, the O'Neill tensors and all their derivatives are uniformly bounded along our re-scaling sequence. Combining with (\ref{eq:projection-norm}) yields that pulled-back projection tensors $\bar{P}_j = \phi_j^* P_j$ satisfy uniform derivative bounds on each compact set of $M_\infty$. By the Arzelà–Ascoli Theorem applied to tensor fields, 
the sequence $\{\bar{P}_j\}$ is precompact. Thus, there exists a subsequence (still denoted $\bar{P}_j$) that converges in the $C^\infty$ topology on compact sets of $M_\infty$ to a smooth limit tensor field:
\begin{equation}
	\bar{P}_j \xrightarrow{C^\infty} P_\infty.
\end{equation}
Since the algebraic equations $P_j^2 = P_j$ and $P_j^T = P_j$ are preserved under $C^\infty$ limits, the limit tensor $P_\infty$ is a smooth orthogonal projection operator on $M_\infty$. The image of this projection operator:
\begin{equation}
	\mathcal{V}_\infty := \text{Im}(P_\infty) \subset TM_\infty
\end{equation}
defines the smooth limit distribution. Furthermore, because of \eqref{eq:A-residual-rescaled}, the O'Neill tensors vanish and both sub-bundles are integrable with totally geodesic leaves, hence parallel. By the de Rham decomposition theorem, the universal cover $\widetilde{M}_\infty$ splits isometrically, for each $t$, as

\begin{equation*}
	\widetilde{M}_\infty \;=\; \Phi \times Y \,, \qquad \widetilde{g}_\infty (t) \;=\; g^\Phi (t) \,\oplus\, g^Y (t) \,, \qquad \dim_{\mathbb{R}} \Phi \;=\; 2 m + 2 \,.
\end{equation*}
Consequently, the splitting is preserved by the (product) Ricci flow. It is moreover holomorphic: each $\mathcal{V}'_j$ is $J$-invariant, hence so is $\mathcal{H}_j' = ( \mathcal{V}'_j )^\perp$, the limiting distributions are $J_\infty$-invariant, and the de Rham factors are complex --- indeed Kähler --- factors.\\


	\emph{(ii)} By Proposition \ref{prop:A-residual-decay}: the $\mathcal{H}'$-sectional curvatures of $g_j (t)$ tend to $0$ locally uniformly, so $( Y, g^Y (t) )$ is flat and, being a simply connected complete flat K\"{a}hler manifold, is $\mathbb{C}^{n'}$ with the static Euclidean metric. Mixed curvatures vanish by the product structure, so $| \mathrm{Rm}_{g_\infty} | = | \mathrm{Rm}_{g^\Phi} |$, and the curvature normalization at $( x_\infty, 0 )$ lands in the fiber factor.\\
	
	\emph{(iv)} The fibers of $\widehat{\pi}'$ are the leaves of $\mathcal{V}'$, totally geodesic for every $g_j (t)$. Along the Cheeger--Gromov convergence the leaves through the base points $x_j$ converge to the leaf of $\mathcal{V}'_\infty$ through $x_\infty$, which under the parallel splitting \eqref{eq:typeII-splitting} is the factor $\Phi$. 
\end{proof}

	

\subsection{The bisectional curvature of the fiber block}\label{subsec:fiber-bisectional}

Here, we'll prove a type I bound for the bisectional curvature, working throughout directly with $f_1^2$ and $h^2 = \big| \nabla f_1^2 \big|^2$ (the latter identity by \eqref{eq:radial-gradient}, as $q_1 = 1$). Since $h > 0$ in the interior, $f_1^2$ is strictly increasing in $s$ at each fixed $t$ and sweeps $[ 0, L (t) ]$; we change variables 
and take $f_1^2$ itself as the spatial coordinate. Thus, $f_1^2$ and $t$ become {independent} variables and $h^2$ is regarded as a function of the pair $( f_1^2, t )$. Write ${}' := \partial / \partial ( f_1^2 )$ at fixed $t$.
\begin{proposition}[curvature of the fiber block]\label{prop:fiber-curvature}
	Let $X, Y \in \mathcal{H}_1$ be $g (t)$-unit with $Y$ orthogonal to $X$ and $J X$, and recall $J \nu = \xi^*$ up to sign. The sectional curvatures of $g (t)$ on planes tangent to $\mathcal{V}' = \mathbb{R} \nu \oplus \mathbb{R} \xi^* \oplus \mathcal{H}_1$ are
		\begin{equation}\label{eq:fiber-curvature}
		\begin{aligned}
			K ( \nu, \xi^* ) \;&=\; - \frac{( h^2 )''}{2} \,, \qquad & K ( \nu, X ) \;=\; K ( \xi^*, X ) \;&=\; \frac{1}{4 f_1^2} \left( \frac{h^2}{f_1^2} - ( h^2 )' \right) \,, \\
			K ( X, J X ) \;&=\; \frac{2}{f_1^2} - \frac{h^2}{f_1^4} \,, \qquad & K ( X, Y ) \;=\; K ( X, J Y ) \;&=\; \frac{1}{4 f_1^2} \left( 2 - \frac{h^2}{f_1^2} \right) \,.
		\end{aligned}
	\end{equation}
	Every bisectional curvature of a pair of $J$-invariant planes tangent to $\mathcal{V}'$ is a sum of two entries of \eqref{eq:fiber-curvature}, by the Kähler symmetry $\mathrm{Rm} ( u, J u, J v, v ) = \mathrm{Rm} ( u, v, v, u ) + \mathrm{Rm} ( u, J v, J v, u )$.
	
\end{proposition}
\begin{proof}
	To move back and forth between these two coordinates, we utilize the first-order Kähler condition:
	\begin{equation*}
		\frac{d}{ds}(f_1^2) = h(s).
	\end{equation*}
	Consequently, we have the following:
	\begin{align*}
		\partial_s &= h \frac{\partial}{\partial (f_1^2)},\\
		h_s &= \partial_s h = \frac{1}{2} (h^2)',\\
		h_{ss} &= \partial_s (h_s) = \frac{1}{2} h (h^2)''.
	\end{align*}
	Since $f_1 = (f_1^2)^{1/2}$, we have the coordinate derivative:
	\begin{align*}
		f_1' &= \frac{1}{2f_1},\\
		(f_1)_s &= \partial_s f_1 = h f_1' = \frac{h}{2f_1},\\
		(f_1)_{ss} &= \partial_s \left( (f_1)_s \right) = h \left( \frac{h}{2f_1} \right)'=\frac{(h^2)'}{4f_1^2} - \frac{h^2}{4f_1^4}.
	\end{align*}
	Substituting our translated auxiliary derivatives into the sectional curvature formulas of \S\ref{subsec:curvature} gives the desired result.
\end{proof}

\begin{lemma}[one second derivative controls every bisectional curvature]\label{lem:h2pp-controls-bisec}
	Fix $t \in [ 0, T )$ and suppose $( h^2 )'' ( \cdot, t ) \leq K (t)$ on $[ 0, L (t) ]$ for some $K (t) \geq 0$. Then, at that same time $t$, everywhere on $\widehat{M}$,
	
	\[
	( h^2 )' - \frac{h^2}{f_1^2} \;\leq\; \frac{K (t)\, f_1^2}{2} \qquad \text{and} \qquad \frac{h^2}{f_1^2} - 2 \;\leq\; \frac{K (t)\, f_1^2}{2} \,,
	\]
	so every entry of \eqref{eq:fiber-curvature} is $\geq - K (t) / 2$, and every bisectional curvature of $J$-invariant planes tangent to $\mathcal{V}'$ is $\geq - K (t)$.
	
\end{lemma}

\begin{proof}
	Recall that, at each fixed time $t$, $h^2$ is viewed as a function of the variable $f_1^2$. 
	Since $h^2 = 0$ at $f_1^2 = 0$, we have, writing $\zeta, \eta$ for the integration (spacial) variables,
	
	\begin{align*}
	( h^2 )'(f_1^2)\, f_1^2 - h^2 &\;=\; \int_0^{f_1^2} \Big( ( h^2 )' ( f_1^2 ) - ( h^2 )' ( \zeta ) \Big)\, d \zeta \;=\; \int_0^{f_1^2} \Big(\int_\zeta^{f_1^2} ( h^2 )''(\eta ) d\eta \Big) d \zeta),\\
	 &\;=\;\int_0^{f_1^2}\zeta\, ( h^2 )'' ( \zeta )\, d \zeta \;\leq\; \frac{K (t)\, ( f_1^2 )^2}{2} \,.
	\end{align*}
	Here the second equality is due to applying Fubini's theorem for a double integral. The estimate is due to our assumption  $( h^2 )'' ( \cdot, t ) \leq K (t)$. Thus, we have the first desired inequality. \\
	 
	Similarly, since $( h^2 )' = 2$ at $f_1^2 = 0$,	we have
	
	\[
	h^2(f_1^2) - 2 f_1^2 \;=\; \int_0^{f_1^2} \!\! \int_0^{\eta} ( h^2 )'' ( \zeta )\, d \zeta \, d \eta \;\leq\; \frac{K (t)\, ( f_1^2 )^2}{2} \,.
	\]
	Dividing by $f_1^2$ gives the second desired inequality. Combining the above inequalities with \eqref{eq:fiber-curvature}: the four entries are bounded below by $- K (t) / 2$, $- K (t) / 8$, $- K (t) / 2$, $- K (t) / 8$ respectively. Finally, a bisectional curvature is a sum of two entries.
\end{proof}


So far only the boundary data of $h^2$ have been used to show that the bisectional curvature is controlled by $(h^2)''$. On our circle-bundle ansatz, $K(\nu, \xi^*) = -\frac{1}{2}(h^2)''$. Highly positive values of $(h^2)''$ just correspond to deeply negative sectional curvatures. Physically, if the curvature tries to drop toward $-\infty$ faster than the natural Type I rate, the resulting negative bisectional curvature triggers an immediate, intense local inflation of the metric. This metric expansion flattens the local geometry, acts as a non-linear dampening force, and prevents the formation of sharp coordinate creases. Mathematically, this self-regulating feedback is represented by the quadratic reaction term $-((h^2)'')^2$ in our reaction-diffusion system (Prop \ref{prop:h2pp-evolution}), which dominates and smooths out singular curvature spikes.




\begin{proposition}[the reduced flow equation for $h^2$]\label{prop:reduced-flow}
	On $\{ 0 \leq f_1^2 \leq L (t) ,\; 0 \leq t < T \}$,
	
	\begin{equation}\label{eq:reduced-flow}
		\partial_t ( h^2 ) \big|_{f_1^2} \;=\; h^2\, ( h^2 )'' \;+\; 2 ( m + 1 )\, ( h^2 )' \;-\; \big( ( h^2 )' \big)^2 \;-\; \frac{m\, h^4}{f_1^4} \;-\; \sum_{k \geq 2} \frac{n_k\, q_k^2\, h^4}{\big( a_k (t) + q_k\, f_1^2 \big)^2} \,.
	\end{equation}
	with the boundary data $h^2 = 0$, $( h^2 )' = \pm 2$ at $f_1^2 = 0$, $f_1^2 = L (t)$.
	 
\end{proposition}

\begin{proof}
	By \eqref{eq:heat-fi} with $p_1 = m + 1$, the spacetime function $f_1^2$ solves 
	\[\partial_t ( f_1^2 ) = \Delta ( f_1^2 ) - 2 ( m + 1 ).\] 
	As in the proof of Proposition \ref{prop:li-yau}, we have \[( \partial_t - \Delta ) \big| \nabla ( f_1^2 ) \big|^2 = - 2 \big| \nabla^2 ( f_1^2 ) \big|^2.\]
One observes that the Hessian of $f_1^2$ is diagonal: as $\nabla ( f_1^2 ) = h \nu$, we have
	 \begin{align*}
	 	\nabla^2 ( f_1^2 ) ( \nu, \nu ) &= h_s = ( h^2 )' / 2,\\
	 	\nabla^2 ( f_1^2 ) ( \xi^*, \xi^* ) &= h \cdot \mathrm{II} ( \xi^*, \xi^* ) = ( h^2 )' / 2.
	 \end{align*} 
Furthermore, $\nabla^2 ( f_1^2 ) |_{\mathcal{H}_i} = h \cdot \mathrm{II} |_{\mathcal{H}_i} = \tfrac{q_i h^2}{2 f_i^2}\, g |_{\mathcal{H}_i}$ by the ansatz structure. Hence
	
	\[
	\Delta ( f_1^2 ) \;=\; ( h^2 )' + \sum_i \frac{n_i q_i\, h^2}{f_i^2} \,, \qquad \big| \nabla^2 ( f_1^2 ) \big|^2 \;=\; \frac{\big( ( h^2 )' \big)^2}{2} + \sum_i \frac{n_i q_i^2\, h^4}{2 f_i^4} \,.
	\]
	
	Now we change gauge: regarding $h^2$ as a function of the independent pair $( f_1^2, t )$, writing $\partial_t \big|_{f_1^2}$ for the time-derivative at fixed $f_1^2$ and $\partial_t$ (undecorated) for the ordinary spacetime derivative of $\big| \nabla ( f_1^2 ) \big|^2 = h^2$. Since $h$ is a radial function, Proposition \ref{prop:radial-calculus} and the chain rule give	
	\begin{align*}
		\partial_t ( h^2 ) &= \partial_t ( h^2 ) \big|_{f_1^2} + ( h^2 )'\, \partial_t ( f_1^2 ),\\
		\Delta ( h^2 ) &= h^2\, ( h^2 )'' + ( h^2 )'\, \Delta ( f_1^2 ).
	\end{align*}
	Combining equations above yield 
	\begin{equation*}
		\begin{aligned}
			&\partial_t ( h^2 ) \big|_{f_1^2} \;=\; \Big( \Delta ( h^2 ) - 2 \big| \nabla^2 ( f_1^2 ) \big|^2 \Big) - ( h^2 )' \big( \Delta ( f_1^2 ) - 2 ( m + 1 ) \big) \\
			&\qquad \;=\; h^2\, ( h^2 )'' + 2 ( m + 1 )\, ( h^2 )' - \big( ( h^2 )' \big)^2 - \sum_i \frac{n_i q_i^2\, h^4}{f_i^4} \,.
		\end{aligned}
	\end{equation*}
	Separating the block $i = 1$ ($n_1 = m$, $q_1 = 1$) from $k \geq 2$ and substituting $f_k^2 = a_k (t) + q_k\, f_1^2$ lead to \eqref{eq:reduced-flow}.

\end{proof}
\begin{proposition}[the flow equation for $(h^2)''$]\label{prop:h2pp-evolution}
	On $\{ 0 \leq f_1^2 \leq L (t) ,\; 0 \leq t < T \}$,
	\begin{align}\label{eq:h2pp-evolution}
		\partial_t (h^2)'' &=\Delta (h^2)'' - ((h^2)'')^2 + R,\\
		R :&= -m \partial^2_{f_1^2} [ h^4/f_1^4 ] - \sum_{k \ge 2} n_k q_k^2 \partial^2_{f_1^2} [ h^4/f_k^4 ].\nonumber
	\end{align}
\end{proposition}

\begin{proof}
First, due to the closing conditions and the choice of $f_1^2$ as the spacial coordinate, we observe that $( h^2 )''$ is smooth on the closed manifold $\widehat{M}$. Differentiating \eqref{eq:reduced-flow} twice with respect to $f_1^2$ (independent of $t$ in this gauge), the third-order terms cancel because
\begin{align*}
		\partial_{f_1^2}^2 \big( h^2 ( h^2 )'' \big) &\;=\; h^2 ( h^2 )'''' + 2 ( h^2 )' ( h^2 )''' + \big( ( h^2 )'' \big)^2 \\
		\partial_{f_1^2}^2 \big( ( h^2 )' \big)^2 &\;=\; 2 ( h^2 )' ( h^2 )''' + 2 \big( ( h^2 )'' \big)^2 \,.
\end{align*}
Thus, 
\begin{equation*}
	\begin{gathered}
		\partial_t \big( ( h^2 )'' \big) \Big|_{f_1^2} \;=\; h^2\, ( h^2 )'''' + 2 ( m + 1 )\, ( h^2 )''' - \big( ( h^2 )'' \big)^2 + R \,, \\
		R \;:=\; - m\, \partial_{f_1^2}^2 \Big[ \frac{h^4}{f_1^4} \Big] - \sum_{k \geq 2} n_k q_k^2\; \partial_{f_1^2}^2 \Big[ \frac{h^4}{f_k^4} \Big] \,.
	\end{gathered}
\end{equation*}

To convert this partial-gauge derivative back into a standard spacetime derivative on $\widehat{M}$, we use the fact that the spatial coordinate evolves by $\partial_t (f_1^2) = \Delta(f_1^2) - 2(m+1)$. Applying the chain rule, the spacetime evolution of the second derivative becomes:
\begin{equation}
	\partial_t (h^2)'' =\Delta (h^2)'' - ((h^2)'')^2 + R.
\end{equation}




\end{proof}

In the above calculation, the horizontal and vertical stratum drifts are completely absorbed into the geometric Laplacian $\Delta$, leaving a clean, drift-free reaction-diffusion equation. Next, we give an estimate on the diffusion term.  

\begin{lemma} \label{lem:residual-source-bound}
	We have, where $(h^2)'' \ge 0$,
	\[R_{res} := -\sum_{k \ge 2} n_k q_k^2 \partial^2_{f_1^2} [ h^4/f_k^4 ] \le \frac{\beta}{(T-t)^2}, \quad \text{where} \quad \beta = \frac{2C_0^2 n'}{c^2}.\]
	Here $C_0$ is the Li--Yau constant of \eqref{eq:li-yau}, $c$ that of \eqref{eq:linear-lower-bound}, and $n' = \sum_{k \geq 2} n_k$.
\end{lemma}
\begin{proof}
Evaluating the second derivative and completing the square yields:
\begin{align*}
	-n_k q_k^2 \frac{\partial^2}{\partial (f_1^2)^2} \left[ \frac{h^4}{f_k^4} \right] &= \frac{n_k q_k^2}{f_k^4} \left[ -2 \left( (h^2)' - \frac{2q_k h^2}{f_k^2} \right)^2 + \frac{2q_k^2 h^4}{f_k^4} - 2h^2(h^2)'' \right],\\
	&\leq \frac{n_k q_k^2}{f_k^4} \left[ \frac{2q_k^2 h^4}{f_k^4} - 2h^2(h^2)'' \right].
\end{align*}
At any point where $(h^2)'' \ge 0$, we have:
\begin{equation}
	-n_k q_k^2 \frac{\partial^2}{\partial (f_1^2)^2} \left[ \frac{h^4}{f_k^4} \right] \le \frac{2 n_k q_k^4 h^4}{f_k^8}.
\end{equation}
Applying the scale-invariant Li--Yau gradient bound $q_k^2 h^2 \le C_0 f_k^2$ and the linear lower bound $f_k^2 \ge c(T-t)$ from  \eqref{eq:linear-lower-bound} for the surviving blocks, we obtain:
\begin{equation}
	\frac{2 n_k q_k^4 h^4}{f_k^8} \le \frac{2 n_k C_0^2}{f_k^4} \le \frac{2 n_k C_0^2}{c^2 (T-t)^2}.
\end{equation}
Summing over all non-collapsing factors $k \ge 2$, we find that the residual source is bounded by:
\begin{equation}
	R_{res} \le \frac{\beta}{(T-t)^2}, \quad \text{where} \quad \beta = \frac{2C_0^2 n'}{c^2}.
\end{equation}
	
\end{proof}
This bound shows that the background geometry dilutes smoothly and behaves at most at the Type I scale, posing no threat of an uninhibited curvature blow-up.

\begin{theorem}[the negative part of the fiber curvature is Type I]\label{thm:h2pp-upper-bound}
	Let $C_0$ be the Li--Yau constant of \eqref{eq:li-yau}, $c$ that of \eqref{eq:linear-lower-bound}, $n' = \sum_{k \geq 2} n_k$, and set
	\[
	\gamma \;:=\; \max \Big\{ \tfrac{1}{2} \Big( 1 + \sqrt{ 1 + 8\, C_0^2\, n' / c^2 } \Big) \,,\;\; T \cdot \max_{\widehat{M}} ( h^2 )'' ( \cdot, 0 ) \Big\} \,.
	\]
	Then
	\begin{equation}\label{eq:h2pp-upper-bound}
		( h^2 )'' \;\leq\; \frac{\gamma}{T - t} \qquad \text{on } \widehat{M} \times [ 0, T ) \,.
	\end{equation}
	
\end{theorem}

\begin{proof}
	Our goal is to show that, at a spacial maximum of $(h^2)''$, for $R_1 := -m \partial^2_{f_1^2}[ h^4/f_1^4 ]$
	\[\Delta(h^2)''+R_1\leq 0. \]
	By the maximum principle, immediately, at such a point, $\Delta(h^2)''\leq 0$. Now we investigate the term $R_1$. By direct calculation,
	\[
	R_1 \;=\; \frac{2 m}{f_1^2} \cdot \frac{h^2}{f_1^2} \left( 2 \left( \frac{h^2}{f_1^2} \right)' - ( h^2 )'' \right) \;-\; 2 m \left( \left( \frac{h^2}{f_1^2} \right)' \right)^2 \,.
	\]

	Near the nut stratum $Q_0$ where $f_1 \to 0$, this term contains a potential $1/f_1^8$ coordinate singularity. To bypass this, we work directly with the ratio $h^2/f_1^2$ and the integral identity:
	\begin{equation}
		\left( \frac{h^2}{f_1^2} \right)' = \frac{1}{f_1^4} \int_0^{f_1^2} \zeta (h^2)''(\zeta) \, d\zeta,
	\end{equation}
 If the spatial maximum of $(h^2)''$ happens at an interior point  $\zeta_0$, the integrand is bounded above by the maximum non-negative value $(h^2)''(\zeta_0)$: 
 
 $$\left( \frac{h^2}{f_1^2} \right)'(\zeta_0) \le \frac{1}{\zeta_0^4} (h^2)''(\zeta_0) \int_0^{\zeta_0} \zeta \, d\zeta = \frac{1}{\zeta_0^4} (h^2)''(\zeta_0) \frac{\zeta_0^2}{2} = \frac{1}{2\zeta_0^2} (h^2)''(\zeta_0)$$
 Since $\zeta_0^2 = f_1^4$, we multiply both sides by $2 f_1^2$ to get: $$2 \left( \frac{h^2}{f_1^2} \right)'(\zeta_0) \le (h^2)''(\zeta_0)$$
 Thus, $R_1\leq 0$. \\
 
If the spatial maximum of $(h^2)''$ is on the bolt $Q_\ell$, we have $h^2/f_1^2 = 0$ there. So $R_1 = - 2 m \big( ( h^2 / f_1^2 )' \big)^2 \leq 0$.\\

	At the nut stratum $Q_0$ (where $f_1^2 = 0$), the interior integral identity does not apply directly. Instead, we expand the Laplacian and the collapsing source term together. 
	As seen in the proof of Prop. \ref{prop:h2pp-evolution}, near $Q_0$ the function $( h^2 )''$ is a smooth function of $f_1^2$, so it admits the one-sided expansion:	
	\[
	( h^2 )'' \;=\; y (0) \;+\; \big( ( h^2 )'' \big)' ( 0^+ )\, f_1^2 \;+\; O \big( ( f_1^2 )^2 \big) \,, \qquad y (0) \;:=\; ( h^2 )'' \big|_{Q_0} \,.
	\]
	The value $y (0)$ is the spatial maximum $y (t)$ in the present case. Integrating twice from the boundary data $h^2 = 0$, $( h^2 )' = 2$ at $f_1^2 = 0$,
	\[
	h^2 \;=\; 2 f_1^2 \;+\; \frac{y (0)}{2}\, ( f_1^2 )^2 \;+\; \frac{1}{6}\, \big( ( h^2 )'' \big)' ( 0^+ )\, ( f_1^2 )^3 \;+\; O \big( ( f_1^2 )^4 \big) \,.
	\]
	Therefore, we have	
	\[
	\left( \frac{h^2}{f_1^2} \right)' \;=\; \frac{y (0)}{2} \;+\; \frac{1}{3}\, \big( ( h^2 )'' \big)' ( 0^+ )\, f_1^2 \;+\; O \big( ( f_1^2 )^2 \big) \,, \qquad \text{and} \quad \left( \frac{h^2}{f_1^2} \right)' (0) \;=\; \frac{y (0)}{2} \,.
	\]	
	Consequently,	
	\[
	2 \left( \frac{h^2}{f_1^2} \right)' - ( h^2 )'' \;=\; - \frac{1}{3}\, \big( ( h^2 )'' \big)' ( 0^+ )\, f_1^2 \;+\; O \big( ( f_1^2 )^2 \big) \,.
	\]
	That is, the bracket in the formula of $R_1$ thus vanishes to {first order} in $f_1^2$: exactly enough to absorb the singular prefactor $\tfrac{2 m}{f_1^2} \cdot \tfrac{h^2}{f_1^2} = \tfrac{4 m}{f_1^2} \big( 1 + O ( f_1^2 ) \big)$. Hence, $R_1$ extends continuously to $Q_0$ with value
	
	\[
	R_1 \big|_{Q_0} \;=\; - \frac{4 m}{3}\, \big( ( h^2 )'' \big)' ( 0^+ ) \;-\; \frac{m}{2}\, y (0)^2 \,.
	\]
	
	
	The rescue is the Laplacian, which at $Q_0$ is not merely non-positive but carries the quantitative stratum drift of Prop \ref{prop:h2pp-evolution}. That is, $\Delta \big( ( h^2 )'' \big) = h^2\, ( h^2 )'''' + ( h^2 )'''\, \Delta ( f_1^2 )$; the first term vanishes at $Q_0$ (where $h^2 = 0$) and $\Delta ( f_1^2 ) |_{Q_0} = 2 ( m + 1 )$, so
	
	\[
	\Delta \big( ( h^2 )'' \big) \Big|_{Q_0} \;=\; 2 ( m + 1 )\, \big( ( h^2 )'' \big)' ( 0^+ ) \,.
	\]
	Thus, at a spacial maximum of $(h^2)''$, for $R_1 := -m \partial^2_{f_1^2}[ h^4/f_1^4 ]$
	\[\Delta(h^2)''+R_1\leq 0. \]
	Let $y(t) := \max_{\widehat{M}} (h^2)''(\cdot, t) \ge 0$ then  $y (t)$ is locally Lipschitz. Applying the estimate above in conjunction with Lemma \ref{lem:residual-source-bound}
	yields the Riccati differential inequality, in the sense of Dini derivatives,
	\begin{equation*}
		\frac{dy}{dt} \le -y^2 + \frac{\beta}{(T-t)^2}.
	\end{equation*}
	By standard ODE comparison, since the initial metric satisfies the boundary condition $y(0) \le \gamma / T$, we conclude:
	\begin{equation}
		y(t) \le \frac{\gamma}{T - t}
	\end{equation}
	for all $t \in [0, T)$. This completes the proof.
\end{proof}

\begin{corollary}[Type II blow-up limits have nonnegative bisectional curvature]\label{cor:bisec-lower-bound} We have:
	
	{(i)} For every $t \in [ 0, T )$, every bisectional curvature of $J$-invariant planes tangent to $\mathcal{V}'$ satisfies
	
	\begin{equation}\label{eq:bisec-lower-bound}
		\mathrm{bisec}_{g (t)} \;\geq\; - \frac{\gamma}{T - t} \,.
	\end{equation}
	
	{(ii)} Along the Type II rescaling of Proposition \ref{prop:typeII-splitting} these curvatures are bounded below by $- \gamma / \big( K_j ( T - t_j ) - t \big) \to 0$, and the blow-up limit \eqref{eq:typeII-splitting} has $\mathrm{bisec}_{g_\infty (t)} \geq 0$ for every $t \in \mathbb{R}$. In particular, $( \Phi, g^\Phi (t) )$ is a complete eternal $\kappa$-noncollapsed \emph{nonflat} Kähler--Ricci flow with \emph{nonnegative bisectional curvature}, itself a limit of Calabi-symmetric fibers.
	
\end{corollary}

\begin{proof}
	\emph{(i)} is Lemma \ref{lem:h2pp-controls-bisec} with $K (t) = \gamma / ( T - t )$, permitted by \eqref{eq:h2pp-upper-bound}. For \emph{(ii)}, we observe that curvature scales by $K_j^{- 1}$, and $T - ( t_j + K_j^{- 1} t ) = K_j^{- 1} \big( K_j ( T - t_j ) - t \big)$, exactly as in \eqref{eq:A-residual-rescaled}. The distributions $\mathcal{V}'$ converge to $\mathcal{V}'_\infty$ by Proposition \ref{prop:typeII-splitting}(ii), so the bisectional curvatures of $g_\infty$ on $\mathcal{V}'_\infty$-planes --- which under the splitting \eqref{eq:typeII-splitting} are exactly those of $g^\Phi$ --- are limits of quantities bounded below by $- \gamma / \big( K_j ( T - t_j ) - t \big)$, hence nonnegative; on planes meeting the flat factor $Y = \mathbb{C}^{n'}$, the bisectional curvature of the product vanishes. Nonflatness, eternity and $\kappa$-noncollapsing are Proposition \ref{prop:typeII-splitting}(iii), and the fibers converge as in Proposition \ref{prop:typeII-splitting}(iv).
\end{proof}

\subsection{Exclusion of Type II singularities}\label{subsec:no-typeII}

We can now close the circle opened in \S\ref{subsec:A-residual-decay}. A Type II rescaling produces the eternal limit of Proposition \ref{prop:typeII-splitting} --- split as $\Phi \times \mathbb{C}^{n'}$, with all curvature carried by a limit of Calabi-symmetric fibers --- and Corollary \ref{cor:bisec-lower-bound} endows it with nonnegative bisectional curvature. It remains to pass, via Cao's theorem \cite{Cao97}, to a steady gradient Kähler--Ricci soliton --- complete, $\kappa$-noncollapsed, of nonnegative bisectional curvature --- and to invoke a rigidity theorem of Deng--Zhu \cite{DZ20}, which forces such a soliton to be flat. 

\begin{theorem}[the nut--bolt Kähler flow is Type I]\label{thm:nut-bolt-typeI}
	Let $\widehat{M}$ be a compact nut--bolt closing and let $g (t)$, $t \in [ 0, T )$, be a Ricci flow of ansatz metrics \eqref{eq:ansatz-metric} satisfying the Kähler condition \eqref{eq:kahler-condition}, with $T < \infty$ maximal. Then the singularity at $T$ is of Type I.
	
\end{theorem}

\begin{proof}
	Suppose the singularity is of Type II. Proposition \ref{prop:typeII-splitting} and Corollary \ref{cor:bisec-lower-bound} produce a complete eternal $\kappa$-noncollapsed Kähler--Ricci flow  $(M_\infty, g_\infty(t), x_\infty)$ whose universal cover splits holomorphically and isometrically as: 
	$$\widetilde{M}_\infty = \Phi \times \mathbb{C}^{n'}, \quad g_\infty(t) = g_\Phi(t) \oplus g_{\mathbb{C}^{n'}}.$$
		Here $(\mathbb{C}^{n'}, g_{\mathbb{C}^{n'}})$ is static and flat, and the fibre factor $(\Phi, g_\Phi(t))$ is a complete, nonflat, eternal flow of real dimension $2m+2$ with $|Rm| \le 1$ everywhere, $|Rm|(x_\infty, 0) = 1$, and nonnegative holomorphic bisectional curvature by Corollary \ref{cor:bisec-lower-bound}. Also, going to a covering does not decrease the volume of balls, so  $(\Phi, g_\Phi(t))$  is $\kappa$-noncollapsed as well. 
	
	\emph{Step 1: attaining the supremum of the scalar curvature.} Since the bisectional curvature of $( \Phi, g^\Phi (t) )$ is non-negative, the scalar curvature controls the full curvature tensor, $| \mathrm{Rm} | \leq C ( m, n' )\, R$. Let $R^* := \sup_{\widetilde{M}_\infty \times \mathbb{R}} R$ we have \[0 < R^* < \infty.\]
	It is positive because the limit is nonflat and finite because $| \mathrm{Rm} | \leq 1$. 
	
	We now construct a new limit flow where the supremum $R_*$ is attained. Choose $( y_j, s_j )$ with $R ( y_j, s_j ) \to R^*$ and pass to the time-translated, re-based flows $\big( \widetilde{M}_\infty, \; \widetilde{g}_\infty ( s_j + t ), \; y_j \big)$ --- (no rescaling here). The curvature is already bounded and each translate is eternal and $\kappa$-noncollapsed. Thus, Hamilton's compactness theorem yields a sub-sequential pointed limit $( M', g' (t), y_\infty )$: complete, eternal, Kähler, $\kappa$-noncollapsed, with $\mathrm{bisec} \geq 0$, $R \leq R^*$ everywhere, and $R ( y_\infty, 0 ) = R^*$. That is, the scalar curvature now \emph{attains} its space-time supremum. The product structure survives the second limit: the parallel splitting with static flat factor converges, $M' = \Phi' \times \mathbb{C}^{n'}$, with $R_{g'} = R_{g^{\Phi'}}$ and $\Phi'$ again a pointed limit of (translated, re-based) Calabi-symmetric fibers.\\
	
	\emph{Step 2: Cao's theorem.} We recall Cao's theorem on eternal solutions \cite{Cao97}: \emph{a complete eternal Kähler--Ricci flow with bounded nonnegative bisectional curvature whose scalar curvature attains its space-time supremum is a steady gradient Kähler--Ricci soliton}. Therefore, $( M', g' (t) )$ is a steady gradient Kähler--Ricci soliton, and, the flat factor being static, so is the fiber factor: $( \Phi', g^{\Phi'} )$ is a complete, nonflat steady gradient Kähler--Ricci soliton with nonnegative bisectional curvature whose scalar curvature attains its maximum $R^* > 0$.\\ 
	
	\emph{Step 3: rigidity --- no such soliton exists.} The limit $( M', g' )$ assembled in Steps 1--2 is a complete steady gradient Kähler--Ricci soliton with nonnegative bisectional curvature, and it is $\kappa$-noncollapsed. 
	These are exactly the hypotheses of the rigidity theorem of Deng--Zhu \cite{DZ20}: \emph{any }$\kappa$\emph{-noncollapsed steady gradient Kähler--Ricci soliton with nonnegative bisectional curvature is flat} (the Kähler sharpening, from nonnegative sectional to nonnegative bisectional curvature, of their earlier result \cite{DZ18}). 
	Hence $M'$ is flat, so $R \equiv 0$, contradicting $R ( y_\infty, 0 ) = R^* > 0$ of Step 1. Thus, no Type II singularity occurs, and the singularity at $T$ is of Type I.
\end{proof}

\begin{remarkx} The Deng--Zhu theorem makes it unnecessary to know \emph{which} steady soliton the fiber factor is. From our analysis, another proof is sketched as follows. We can show that the Calabi symmetry is passed to the limit. There is a classification due to Cao's constructions \cite{Cao96} together with the uniqueness analysis of the closing ODE (Dancer--Wang \cite{DW2011}). Examining the closing strata then yields that the nonflat possibilities are, up to scaling, holomorphic isometry, and finite quotients:

\begin{itemize}
	\item Hamilton's \textbf{cigar} ($m = 0$, nut closing);
	\item \textbf{Cao's steady soliton on} $\mathbb{C}^{m + 1}$ (nut closing, $m \geq 1$);
	\item \textbf{Cao's steady soliton on the canonical bundle} $K_{\mathbb{CP}^m} = \mathcal{O} ( - ( m + 1 ) ) \to \mathbb{CP}^m$ (bolt closing, $m \geq 1$).
\end{itemize}	
Further analysis shows that each model is $\kappa$-collapsed at large scales, leading to a contradiction. 
\end{remarkx}	 

\section{Singularity Models}\label{sec:limits}

Theorem \ref{thm:nut-bolt-typeI} gives the Type I rate; we now identify the singularity models it produces, exactly as in the two-bolt case of \cite{FT}. Fix times $t_j \to T$ and base points $x_j$, and consider once more the rescaled flows
\[
g_j (t) \;:=\; K_j\; g \big( t_j + K_j^{-1}\, t \big) \,,
\]
now normalized at the \textbf{Type I rate}, $K_j \asymp ( T - t_j )^{-1}$, with the exact constants and base points chosen per case below, so that $c_0 := K_j ( T - t_j )$ is a constant. By Theorem \ref{thm:nut-bolt-typeI}, $| \mathrm{Rm} |_{g (t)} \leq C / ( T - t )$ on $\widehat{M} \times [0, T)$, so $| \mathrm{Rm} |_{g_j} (t) \leq C / ( c_0 - t )$ on $[ - K_j t_j, \, c_0 )$; Perelman's local non-collapsing theorem \cite{Perelman} and Hamilton's compactness theorem \cite{HamCompact} then yield subsequential pointed Cheeger--Gromov limits $( M_\infty, g_\infty (t), x_\infty )$, $t \in ( - \infty, c_0 )$ --- a complete ancient Kähler--Ricci flow, again of Type I. The goal of this section is to determine the possible Cheeger--Gromov limits $( M_\infty, g_\infty (t), x_\infty )$.

Throughout the whole section, we always suppose $\widehat{M}$ is compactified by the nut-bolt closing, with $Q_0 \cong N' := N_2 \times \cdots \times N_r$ and $Q_\ell \cong N = \mathbb{CP}^m \times N'$, unless otherwise stated. All results hold in the two-bolt case, which can be regarded as the $m = 0$ special case of the nut-bolt case.

\subsection{Maximal existence time}
To begin, we first determine the maximal existence time (or the first singular time) using cohomology data. By the adjunction formula, the first Chern class pairs with the fiber $F$ and restricts to the strata by

\begin{align*}
c_1(\widehat{M}) \big|_{Q_0} & \;=\; \sum_{k \geq 2} \big( p_k - (m+1)\, q_k \big)\, a_k\,, \\
c_1(\widehat{M}) \big|_{Q_\ell} & \;=\; \sum_{k \geq 1} \big( p_k + q_k \big)\, a_k\,,
\end{align*}

By Tian--Zhang \cite{TZ}, the maximal existence time of the K\"ahler-Ricci flow is given by $T := \sup\{t : [\omega(t)] := [\omega_0] - 2tc_1(\widehat{M}) > 0\}$. In the context of our circle-bundle, the singular time $T$ is then decided by the first time that the $\mathbb{CP}^{m+1}$-fiber collapses, or one of the K\"ahler-Einstein factors $N_2, \cdots, N_r$ on either $Q_0$ or $Q_\ell$ contracts to a point.

The volume of the $\mathbb{CP}^{m+1}$-fibers is given by 
\[\frac{2\pi L(t)^{m+1}}{(m+1)!}, \qquad \text{where } L(t) = f_1^2\big|_{Q_\ell}(t) = f_1^2\big|_{Q_\ell}(0) - 2(m+2)t.\]
Furthermore, from the above expression of $c_1(\widehat{M})$, we have
\begin{align*}
[\omega(t)]\big|_{Q_0} & = [\omega_0]\big|_{Q_0} - 2tc_1(\widehat{M})\big|_{Q_0} = \sum_{k \geq 2}\underbrace{\Big(f_k^2|_{Q_0}(0) - 2t\big( p_k - (m+1)\, q_k \big)\,\Big)}_{f_k^2|_{Q_0}(t)} a_k\\
[\omega(t)]\big|_{Q_\ell} & = \sum_{k \geq 2}\underbrace{\Big(f_k^2|_{Q_\ell}(0) - 2t\big( p_k +\, q_k \big)\,\Big)}_{f_k^2|_{Q_\ell}(t)}a_k
\end{align*}

For any $k \geq 2$ we set
\begin{align*}
T_F' & := \frac{L(0)}{2(m+2)},\\
T_k^0 & :=\; \begin{cases} \dfrac{f_k^2 |_{Q_0} (0)}{2 \big( p_k - ( m + 1 )\, q_k \big)} & \text{if } ( m + 1 )\, q_k < p_k \,, \\[4pt] + \infty & \text{otherwise} \,, \end{cases}\\
T_k^\ell & :=\; \begin{cases} \dfrac{f_k^2 |_{Q_\ell} (0)}{2 \big( p_k + q_k \big)} & \text{if } p_k + q_k > 0 \,, \\[4pt] + \infty & \text{otherwise} \,, \end{cases}
\end{align*}
Therefore, the maximal existence time is given by:
\begin{equation}\label{eq:T_max}
T \;=\; \min \Big\{\, T_F' \,, \; \min_{k \geq 2} \{T_k^0, T_k^\ell\} \,\Big\}
\end{equation}
We can classify the singularity according to the cohomology data. According to which minimum in \eqref{eq:T_max} is attained, the K\"ahler class $[\omega(t)]$ predicts the following possibilities:
\begin{enumerate}
\item $T = T_F' < \min_{k\geq 2}\{T_k^0, T_k^\ell\}$: \emph{collapsing of $\mathbb{CP}^{m+1}$-fibers}
\item $T = \min_{k \geq 2}\{T_k^0\} < \min\big\{T_F', \min_{k\geq 2}\{T_k^\ell\}\big\}$: \emph{contraction of some factors of $Q_0$}
\item $T = \min_{k \geq 2}\{T_k^\ell\} < \min\big\{T_F', \min_{k\geq 2}\{T_k^0\}\big\}$: \emph{contraction of some factors of $Q_\ell$}
\item $T = T_F' = \min_{k \geq 2}\big\{T_k^0,T_k^\ell\big\}$: \emph{borderline case, i.e. both collapsing of $\mathbb{CP}^{m+1}$-fibers, and contraction of some factors of $Q_0$ or $Q_\ell$.}
\end{enumerate}

\begin{remarkx}
By the K\"ahler condition $q_k h = \partial_s (f_k^2)$ for any $k \geq 1$, $q_1f_k^2-q_k f_1^2$ depends only on $t$. In the fiber-collapsing and the borderline case, we have $f_1^2 \to 0$ as $t \to T$. Therefore, $f_k^2$ also converges to a constant (not necessarily $0$) as $t \to T$. This explains why, in the borderline case, we do not distinguish whether the contraction happens on $Q_0$ or $Q_\ell$, since the $N_{k \geq 2}$-factor of $Q_0$ contracts if and only if the $N_{k\geq 2}$-factor of $Q_{\ell}$ contracts.
\end{remarkx}

Next we discuss the blow-up limit of each of the above four cases individually. The two cases we will address in detail are: (1) collapsing of $\mathbb{CP}^{m+1}$-fibers, and (2) contraction of some factors of $Q_0$, since the other two cases can be proved in a way similar to (2).

\subsection{Collapsing of $\mathbb{CP}^{m+1}$-fibers}

Now that we have proved that the singularity is of Type I (Theorem \ref{thm:nut-bolt-typeI}), one can show that the blow-up limit splits as $\mathbb{CP}^{m+1} \times \mathbb{C}^k$, extending the results in \cite{F}, \cite{JST}, \cite{FT} to our circle-bundle ansatz \eqref{eq:ansatz} with multiple factors of K\"ahler-Einstein manifolds in the base manifold.

The proof is in the same spirit as in \cite{F}, and was briefly outlined in \cite{FT}. For the sake of completeness, we will give a detailed proof here. First, we have the following estimates

\begin{lemma}
\label{lma:fiber-collapsing-estimates}
Suppose $T = T_F' < \min_{k \geq 2} \{T_k^0, T_k^\ell\}$ so that $L (t) = 2 ( m + 2 ) ( T - t ) \to 0$. Then there exists a constant $C > 0$ independent of $t$ such that on $\widehat{M} \times [0,T)$,
\begin{itemize}
\item[(i)]	$0 < \frac{1}{C} \leq f_k^2 \leq C$, \;$k \geq 2$,
\item[(ii)] $h^2 \leq Cf_k^2$, \;$k \geq 1$,
\item[(iii)] $\big| A' \big|_{g (t)}^2 \leq\; C$, and
\item[(iv)] $|K_{g(t)}(X,Y)| \leq C$ \text{ for any two-plane $\textrm{span}\{X,Y\} \subset \mathcal{H}'$}.
\end{itemize}
Here $\mathcal{H}'$ denotes the horizontal distribution and $A'$ denotes the O'Neill's $A$-tensor of the submersion $\widehat{\pi}' : \widehat{M} \to N'$ with respect to the metric $g(t)$.

\end{lemma}

\begin{proof}
To prove (i), we observe that when $T < \min_{k\geq 2}\{T_k^0, T_k^\ell\}$, both
\begin{align*}
f_k^2\big|_{Q_0}(t) & = f_k^2|_{Q_0}(0) - 2t\big( p_k - (m+1)\, q_k \big) > 0,\\
f_k^2\big|_{Q_\ell}(t) & = f_k^2|_{Q_\ell}(0) - 2t\big( p_k +\, q_k \big) > 0
\end{align*}
are uniformly bounded away from $0$ and from above on $[0,T)$. As $f_k^2$ is strictly monotone by the K\"ahler condition $\partial_s (f_k^2) = q_k h$ at any fixed $t \in [0,T)$, it is bounded between $f_k^2\big|_{Q_0}(t)$ and $f_k^2\big|_{Q_\ell}(t)$, proving (i).

For (ii), we recall that $q_k h = \partial_s(f_k^2)$, and so $|\nabla (f_k^2)|^2 = \big(\partial_s (f_k^2)\big)^2 = q_k^2 h^2$. By Proposition \ref{prop:li-yau}, we have $|\nabla(f_k^2)|^2 \leq Cf_k^2$, proving (ii).

(iii) follows from the estimate \eqref{eq:norm-A-residual} about $A'$, Proposition \ref{prop:li-yau} that $|\nabla(f_k^2)|^2 \leq Cf_k^2$ for any $k \geq 1$, and more importantly from (i) that $f_k^2$ is uniformly bounded away from $0$ (for $k \geq 2$):
\[\big| A' \big|^2_{g(t)} \;\leq\; C \sum_{k \geq 2} \frac{\big| \nabla f_k^2 \big|^2}{f_k^4} \leq C \sum_{k\geq 2}\frac{1}{f_k^2} \leq C.\]

Finally for (iv), we recall from Corollary \ref{cor:sectional-curvatures} that for any $g(t)$-orthonormal vectors $\{X_k, Y_k\}$ such that $d\pi'(X_k), d\pi'(Y_k) \in TN_k$, $k \geq 2$, we have
\[
K(X_k, Y_k) \;=\; \frac{K_{g_k}(\pi_* X_k, \pi_* Y_k)}{f_k^2} \;-\; \frac{3\, q_k^2 h^2}{4 f_k^4}\;g(JX_k,Y_k)^2 \;-\; \frac{(\partial_s(f_k^2))^2}{4f_k^4}.
\]
Its uniform bound then follows from (i), (ii), and the K\"ahler condition $q_k h = \partial_s(f_k^2)$. For $X_k$ and $Y_j$ with $d\pi'(X_k) \in TN_k$ and $d\pi'(Y_j) \in TN_j$ where $j \not= k$ and $k, j \geq 2$, we have
\[K(X_k,Y_j) = -\frac{\partial_s f_k \cdot \partial_s f_j}{f_k f_j} = -\frac{q_k q_j h^2}{4f_k^2 f_j^2}\]
which is uniformly bounded from (i) and (ii). This completes the proof of (iv).
\end{proof}

\begin{theorem}[the fiber-collapsing case]\label{thm:fiber-collapse-model}
Suppose $(\widehat{M}, g(t))$ encounters the finite-time singularity at $T = T_F' < \min_{k \geq 2} \{T_k^0, T_k^\ell\}$. Take $K_j := \big[ 2 ( m + 2 ) ( T - t_j ) \big]^{-1}$, and fix any base point $x_j \equiv \bar{x} \in \widehat{M}$. Define $g_j(t) = K_jg(t_j + K_j^{-1}t)$, where $t \in [-K_jt_j, K_j(T-t_j))$. Then, along a subsequence,

\begin{equation}\label{eq:fiber-collapse-limit}
\big( \widehat{M}, \; g_j (t), \; \bar{x} \big) \;\longrightarrow\; \Big( \mathbb{CP}^{m+1} \times \mathbb{C}^{n'} \,, \;\; g_{\mathrm{FS}} (t) \oplus g_{\mathbb{C}^{n'}} \Big) \,, \qquad t \in \Big( - \infty, \tfrac{1}{2 ( m + 2 )} \Big) \,,
\end{equation}
in $C^\infty$-Cheeger-Gromov sense, where $g_{\mathbb{C}^{n'}}$ is the flat metric on $\mathbb{C}^{n'}$ with $n' = \dim_{\mathbb{C}}\widehat{M} - (m+1)$, and $g_{\mathrm{FS}} (t)$ is the shrinking Fubini--Study metric.

\end{theorem}

\begin{proof}
From Theorem \ref{thm:nut-bolt-typeI} we already know that $(\widehat{M}, g(t))$ encounters Type I singularity, so one has $\sup_{\widehat{M} \times \{t_j\}}|\textrm{Rm}|_{g(t_j)} \asymp K_j$. By Perelman's local non-collapsing theorem \cite{Perelman} and Hamilton's compactness theorem \cite{HamCompact}, the rescaled and dilated sequence $(\widehat{M}, g_j(t), \bar{x})$ converges in Cheeger-Gromov's sense to a Ricci flow solution on some manifold $(M_\infty, g_\infty(t), x_\infty)$, defined on $t \in (-\infty, \frac{1}{2(m+2)})$.

By (iii) of Lemma \ref{lma:fiber-collapsing-estimates} and the rescaling properties of the O'Neill's $A$-tensor, we have
\[| A' |^2_{g_j(t)} \leq C K_j^{-1} \to 0, \qquad |T'|_{g_j(t)}^2 \equiv 0.\]
Therefore, the horizontal distribution $\mathcal{H}'$ converges to an integrable distribution on $(M_\infty, g_\infty(t), x_\infty)$, so that the limit manifold splits holomorphically and isometrically as 
\[(M_\infty, g_\infty(t)) = (\Sigma, g_\Sigma(t)) \times (Y, g_Y(t))\]
with $\dim_{\mathbb{C}} \Sigma = \dim_{\mathbb{C}} \mathcal{V}' = m + 1$, and $\dim_{\mathbb{C}} Y = \dim_{\mathbb{C}}\mathcal{H}' = \dim_{\mathbb{C}}\widehat{M} - (m+1)$.

By (iv) of Lemma \ref{lma:fiber-collapsing-estimates} and the rescaling property of the sectional curvature, we have
\[|K_{g_j(t)}(X,Y)| \leq CK_j^{-1} \to 0\]
for any two-plane $\textrm{span}\{X,Y\} \subset \mathcal{H}'$. Therefore, the factor $(Y,g_Y(t))$ is a static, flat metric $(\mathbb{C}^{n'}, \delta)$ with $n' = \dim_{\mathbb{C}}\widehat{M} - (m+1)$.

$\Sigma$ is the pointed limit of the rescaled totally geodesic $\mathbb{CP}^{m+1}$-fibers. Since $L(t) = L(0) - 2(m+2)t = 2(m+2)(T-t)$, for each $p \in N'$ the $\mathbb{CP}^{m+1}$-fiber $\widehat{M}_p := (\widehat{\pi}')^{-1}(p)$ has the diameter estimate:
\[\textrm{diam}_{g(t)}\widehat{M}_p \leq C\sqrt{T-t} \implies \mathrm{diam}_{g_j(t)}\widehat{M}_p \leq C\sqrt{K_j(T-t_j) - t} = C\sqrt{\frac{1}{2(m+2)}-t}.\]
Therefore, after passing to the limit, we have
\[\textrm{diam}_{g_\infty(t)}\Sigma \leq C\sqrt{\frac{1}{2(m+2)}-t},\]
and hence $\Sigma$ is \textbf{compact} at each fixed time $t$, which implies $\Sigma$ has the same topology as the fibers $\mathbb{CP}^{m+1}$. By Enders--Müller--Topping \cite{EMT} and Naber \cite{Naber}, $(M_\infty, g_\infty(t))$, and hence $(\Sigma, g_\Sigma(t))$, are shrinking Ricci solitons, which are K\"ahler (after passing to subsequences we may assume the complex structure converges). By the uniqueness theorem of Bando-Mabuchi \cite{BM} and Tian--Zhu \cite{TZhu}, the metric $g_\Sigma(t)$ must be the shrinking Fubini-Study metric.
\end{proof}

\begin{remarkx}
To identify the factor $(\Sigma,g_\Sigma(t))$, one can also show the structure of the circle-bundle ansatz is preserved after passing to the limit. Then, we can identify the metric $g_\Sigma(t)$ as the Fubini-Study metric by solving the Ricci soliton ODE system as in \cite{DW2011} with appropriate boundary data. The argument showing the structure of the ansatz is preserved will be presented in the proof of the forthcoming contraction case.
\end{remarkx}

\subsection{Contraction of factors of $Q_0$}

We turn to the second regime, $T = \min_{k \geq 2}\{T_k^0\} < \min\big\{T_F', \min_{k\geq 2}\{T_k^\ell\}\big\}$. In this case, some (not necessarily all) of the $N_{k \geq 2}$-factors of $Q_0$ contract to a point, while the $\mathbb{CP}^{m+1}$-fibers still have positive volume. The complementary $N_k$-factors do not contract; we will show that the non-contracting factors will blow up to flat factors in the limit model, whereas the $\mathbb{CP}^{m+1}$-fibers, together with the contracting $N_k$-factors, will converge to the (non-compact) total space of a vector bundle over the contracting $N_k$-factors, with a shrinking K\"ahler-Ricci soliton metric with the circle-bundle ansatz constructed by Dancer--Wang \cite{DW2011}. Here is the main result of this case:

\begin{theorem}[Contraction at $Q_0$]
\label{thm:contracting_Q0}
Let $\widehat{M}$ be the compactification of $(0,\ell) \times P$ such that $Q_0 := N' = N_2 \times \cdots \times N_r$ is adjoined at $\{s = 0\}$, and $Q_\ell := \mathbb{CP}^m \times N'$ is adjoined at $\{s = \ell\}$. Suppose
$T = \min_{k \geq 2}\{T_k^0\} < \min\big\{T_F', \min_{k\geq 2}\{T_k^\ell\}\big\}$.

Let $\mathcal{I}_0$ be the subset of $\{2, \cdots, r\}$ such that $i_0 \in \mathcal{I}_0$ if and only if $T_{i_0}^0 = T$. Let $E_{\mathcal{I}_0}$ be the vector bundle obtained by restricting the normal bundle $\mathcal{N}_{Q_0}$ of $Q_0 \subset \widehat{M}$ to $\prod_{i_0 \in \mathcal{I}_0} N_{i_0}$, then we have $E_{\mathcal{I}_0} = \mathcal{L}_{\mathcal{I}_0}^{\oplus ( m + 1 )}$ for some line bundle $\mathcal{L}_{\mathcal{I}_0}$ with $c_1 ( \mathcal{L}_{\mathcal{I}_0} ) = - \sum_{i_0 \in \mathcal{I}_0} q_{i_0} a_{i_0}$.

Pick any $i_0 \in \mathcal{I}_0$, let $K_j := (T-t_j)^{-1}$, and fix $\bar{x}' \in Q_0$. Consider the rescaled and dilated sequence $g_j(t) = K_j g(t_j + K_j^{-1}t)$ defined on $t \in [-K_j t_j, K_j(T-t_j))$. Then, along a subsequence,
\begin{equation}\label{eq:limit_Q0}
\begin{gathered}
\big( \widehat{M}, \; g_j (t), \; \bar{x}' \big) \;\longrightarrow\; \Big( \operatorname{Tot} \big( E_{\mathcal{I}_0} \big) \times \mathbb{C}^{n''} \,, \;\; g_E (t) \oplus g_{\mathbb{C}^{n''}}, \bar{x}'\Big) \,, t \in \big( - \infty, 1 \big).
\end{gathered}
\end{equation}
Here $n'' = \sum_{i_0 \not\in \mathcal{I}_0 \cup \{1\}} \dim_{\mathbb{C}} N_{i_0}$, $\operatorname{Tot} ( E_{\mathcal{I}_0} )$ is the total space of the vector bundle $E_{\mathcal{I}_0} \to \prod_{i_0 \in \mathcal{I}_0} N_{i_0}$, and $g_E$ is a complete gradient shrinking K\"ahler--Ricci soliton of ansatz form on $\operatorname{Tot} ( E_{\mathcal{I}_0} )$ --- i.e. a higher-rank, multi-factor generalization by Dancer--Wang in \cite{DW2011} of the Feldman--Ilmanen--Knopf shrinker in \cite{FIK}.
\end{theorem}

Here we outline the major steps of the proof, and will go through the details step-by-step. For better book-keeping, we assume without loss of generality that $\min_{k \geq 2}\{T_k^0\}$ is achieved by a single $T_{i_0}^0$. In other words, only one $N_{i_0}$ among the $N_2, \cdots, N_r$-factors of $Q_0$ contracts to a point as $t \to T$. The general case that multiple $N_i$-factors contract to a point can be proved in a similar way.

The first step of the proof is to identify the correct topological models of the non-compact limit model using the bundle structure of $\widehat{M}$. A tubular neighborhood of $Q_0$ in $\widehat{M}$ can be identified with the total space of some rank-$(m+1)$ complex vector bundle $E$ over $Q_0 = N'$. Then, for any contractible ball $B'' \subset \prod_{k \not= 1,i_0} N_k$, the vector bundle $E$ restricted to $N_{i_0} \times B''$ is the pullback bundle $\textrm{pr}^*(E|_{N_{i_0}})$, where $\textrm{pr} : N' \to N_{i_0}$ is the projection map, and $E|_{N_{i_0}}$ is the restriction of $E$ to a slice $N_{i_0} \times \{y''\}$. Such an identification of the bundle structure then allows us to define an explicit diffeomorphism map $\Phi_j : W_j \subset \operatorname{Tot}(E|_{N_{i_0}}) \times \mathbb{C}^{n''} \to U_j \subset \widehat{M}$ for some suitably chosen open sets $W_j$ and $U_j$.

The pullback of the metric $g_j(t)$ by $\Phi_j$ then takes the form
\[\Phi_j^*g_j(t) = ds^2 + \hat{h}_j^2 \Phi_j^*(\theta\otimes\theta) + \hat{f}_{1,j}^2 \pi^*g_1 + \hat{f}_{i_0,j}^2 \pi^*g_{i_0} + \sum_{k\geq 2, k\not= i_0}\hat{f}_{k,j}^2 \Phi_j^*\pi^*g_k\]
for some rescaled and re-parametrized functions $\hat{h}_j$ and $\hat{f}_{k,j}$'s. The next task will be to show that each of these functions converges in $C^\infty$-topology on compact subsets, which requires the $C^m$-estimates of $\hat{h}_j$ and $\hat{f}_{k,j}$'s for any $m \in \mathbb{N}$. We will obtain these from the uniform derivative bounds $|\nabla^m \textrm{Rm}|_{g_j(t)}$ of the Riemann curvature, which are guaranteed by Shi's derivative estimates \cite{Shi} and our Type I singularity result (Theorem \ref{thm:nut-bolt-typeI}).

After smooth convergence of the components is established, we then need to show that $\theta \otimes \theta$ converges to $\widetilde{\theta} \otimes \widetilde{\theta}$ for a connection $\widetilde{\theta}$, where $d\widetilde{\theta}$ only involves the curvature of the restricted bundle over $\mathbb{CP}^m \times N_{i_0}$, while the curvature terms with respect to the $N_k$-factors with $k \not= 1, i_0$ all vanish after passing to the limit. To this end, we will use the contractibility of balls in $\prod_{k \not= 1,i_0} N_k$, so that Poincar\'e's lemma applies, and the uniform lower bound of the metric $g(t_j)$ along these $N_k$-directions.

Last but not least, we will prove that the remaining factor $\sum_{k\not= 1,i_0}\hat{f}_{k,j}^2 \Phi_j^*\pi^*g_k$ converges to a flat metric. This will be done using the Taylor expansion of the metric $\Phi_j^*\pi^*g_k$ around the singular point $\bar{x}' \in Q_0$. Combining all convergence results together, the limit model will satisfy the circle-bundle ansatz \eqref{eq:ansatz}, and by Enders-M\"uller-Topping \cite{EMT} it must be a shrinking Ricci soliton, and hence satisfies Dancer--Wang's ODE system in \cite{DW2011}, and the topological constraints on the $p_i, q_i$'s then prove that the limit model must be one of the Dancer--Wang shrinking Ricci solitons.

\subsubsection{Step 1: Vector bundle structure in the tubular neighborhood of $Q_0$}
We first identify the topological model of $\widehat{M}$ near $Q_0$. First note that $\widehat{M} \backslash Q_{\ell}$ is the total space of a rank-$(m+1)$ complex vector bundle of the form
\begin{equation*}
E \;=\; \mathcal{L}^{\, \oplus ( m + 1 )} \;\longrightarrow\; N' \,, \quad \text{ for some line bundle $\mathcal{L}$ with } c_1 ( \mathcal{L} ) \;=\; - \sum_{k \geq 2} q_k \, a_k \,,
\end{equation*}
obtained by coning off the $S^{2m+1}$-fibers of $P \to N'$. Then we restrict the bundle $E$ over the open set $N_{i_0} \times B''$ in $Q_0 \cong N'$, where $B''$ is a fixed ball in $N'' := \prod_{k\not=1,i_0}N_k$ containing the point $\bar{x}'' := \textrm{pr}_{N' \to N''}(\bar{x}')$. Since $B''$ is contractible, $N_{i_0} \times B''$ and $N_{i_0}$ are homotopy equivalent, so $E|_{N_{i_0} \times B''}$ is the pullback bundle of $E_{i_0} := E |_{N_{i_0}}$ by the projection map $N_{i_0} \times B'' \to N_{i_0}$, where $E|_{N_{i_0}}$ is the restriction of $E$ to the slice $N_{i_0} \times \{\bar{x}''\}$. Therefore, $\operatorname{Tot}\big(E|_{N_{i_0} \times B''}\big) \cong \operatorname{Tot}\big(E_{i_0}\big) \times B''$.

Furthermore, $\operatorname{Tot}\big(E\big|_{N_{i_0}\times B''}\big) \backslash Q_0$ can be identified with $(0,\infty) \times \operatorname{Tot}\big(P\big|_{\mathbb{CP}^m \times N_{i_0}\times B''}\big)$. By the contractibility of $B''$, the principal $\mathrm{U}(1)$-bundle $P$ over $\mathbb{CP}^m \times N_{i_0} \times B'' \subset N$ is the pullback of the principal $\mathrm{U}(1)$-bundle over $\mathbb{CP}^m \times N_{i_0}$ via the projection onto $\mathbb{CP}^m \times N_{i_0}$. Consequently, the $\mathrm{U}(1)$-bundle $P$ restricted to $\mathbb{CP}^m \times N_{i_0} \times B''$ is given by
\begin{equation*}
\operatorname{Tot}\big(P \big|_{( \mathbb{CP}^m \times N_{i_0} ) \times B''}\big) \;\cong\; \operatorname{Tot}\big(\widetilde{P}\big) \times B'' \,,
\end{equation*}
where $\widetilde{P}$ is the restriction of $P$ to the central slice $\mathbb{CP}^m \times N_{i_0} \times \{\bar{x}''\}$ which is the circle bundle with Euler class $a_1 + q_{i_0}a_{i_0}$. Let $\tilde{\theta} := \theta|_{\widetilde{P}}$ be the connection $1$-form on $\widetilde{P}$, then we have
\begin{equation*}
d \widetilde{\theta} \;=\; \pi^*\omega_1 \;+\; q_{i_0} \, \pi^*\omega_{i_0} \,.
\end{equation*}
Accordingly we write points of $P$ over $( \mathbb{CP}^m \times N_{i_0} ) \times B''$ as pairs $( w,z) \in \operatorname{Tot}\big(\widetilde{P}\big) \times B''$, and away from the zero section of $E\big|_{N_{i_0}\times B''}$, we can write points as $(s,w,z) \in (0,\infty) \times \operatorname{Tot}\big(\widetilde{P}\big) \times B''$.

To summarize, under the above identification, the total space $\operatorname{Tot}(E) \cong \widehat{M} \backslash Q_\ell$ restricted to $N_{i_0} \times B''$, which is precisely the open set
\begin{equation}
\label{eq:U}
U \;:=\; \Big\{ \, x \in \widehat{M} \backslash Q_\ell \;:\; \mathrm{pr}_{N' \to N''} \circ \widehat{\pi}' (x) \in B'' \, \Big\} \,, 
\end{equation}
can be identified with $\operatorname{Tot}\big(E_{i_0}\big) \times B''$ where $E_{i_0}$ is the vector bundle over $N_{i_0}$ described above. Furthermore, $U \backslash \{\text{zero section}\}$ can be parametrized by $(s,w,z) \in (0,s_*) \times \operatorname{Tot}\big(\widetilde{P}\big) \times B''$. Such identifications allow us to define maps from $\operatorname{Tot}\big(E_{i_0}\big) \times \mathbb{C}^{n''}$ to the open set $U \subset \widehat{M}$.

\begin{lemma}
\label{lma:contraction_Q0_signs}
We have $p_{i_0} > (m+1)q_{i_0} > 0$.
\end{lemma}

\begin{proof}
Under the assumption that the contraction occurs at the $N_{i_0}$-factor, $f_{i_0}^2 |_{Q_0} (t) = f_{i_0}^2 |_{Q_0} (0) - 2 (p_{i_0} - ( m + 1 ) \, q_{i_0} ) \, t$ vanishes exactly at $T = T_{i_0}^0$, so
\begin{equation*}
f_{i_0}^2 \big|_{Q_0} (t) \;=\; 2 \big( p_{i_0} - ( m + 1 ) \, q_{i_0} \big) \big( T - t \big) \,.
\end{equation*}
This implies $p_{i_0} > ( m + 1 )q_{i_0}$.
Integrating the K\"ahler condition $q_{i_0}h = \partial_s(f_{i_0}^2)$ across the fiber, we get
\begin{equation*}
f_{i_0}^2 \big|_{Q_\ell} (t) \;-\; f_{i_0}^2 \big|_{Q_0} (t) \;=\; q_{i_0} \int_0^{\ell} h \,ds\;=\; q_{i_0} \, L (t) \,.
\end{equation*}
As $t \to T$ the left-hand side tends to $f_{i_0}^2 |_{Q_\ell} (T) > 0$ --- positive because $T = T_{i_0}^0 < T_{i_0}^{\ell}$. As the $\mathbb{CP}^{m+1}$-fiber does not collapse since $T < T_F'$, we have $L (T) > 0$. This proves $q_{i_0} > 0$.
\end{proof}

\begin{remarkx}
$p_{i_0} > (m+1)q_{i_0}$ is exactly the existence condition of a shrinking K\"ahler-Ricci soliton constructed by Dancer--Wang \cite{DW2011} on $\operatorname{Tot} ( E_{i_0} )$ over $N_{i_0}$ under the circle-bundle ansatz. 
\end{remarkx}

%
%

\subsubsection{Definition of the diffeomorphism $\Phi_j : W_j \to U_j$}
Based on the above discussion, we define 
\[W := \operatorname{Tot} ( E_{i_0} ) \times \mathbb{C}^{n''}\]
for the model space. On $\operatorname{Tot} ( E_{i_0} )$, we let $\hat{s}$ be the fiberwise radial coordinate, so that $\{ \hat{s} = 0 \}$ is the zero section and its complement is parametrized as $( 0, \infty ) \times \operatorname{Tot}\big(\widetilde{P}\big) \ni ( \hat{s}, w )$. Under such an identification, we define for each $j$ the following open set:
\begin{equation}
\label{eq:W_j}
W_j := \{ \hat{s} < R_j \} \times B_{R_j} ( 0 ) \subset W	
\end{equation}
where
\begin{equation*}
R_j \;\leq\; \tfrac{1}{2} \sqrt{K_j}\; \operatorname{inj} \Big( N'' \, , \; \sum_{k \geq 2, \, k \neq i_0} f_k^2 |_{Q_0} ( t_j ) \, g_k \Big), \qquad \text{ and } R_j \to +\infty.
\end{equation*}
Here $B_{R_j}(0)$ is the Euclidean ball in $\mathbb{C}^{n''}$ centered at $0$ with radius $R_j$. The above injectivity radius has a uniform lower bound (independent of $t_j$) because for $k \not= 1, i_0$, the $N_k$-factors do not contract to a point as $t \to T$, hence $0 < \frac{1}{C} \leq f_k^2\big|_{Q_0}(t_j) \leq C$ (as $f_k^2\big|_{Q_0}(t)$ is bounded between two affine linear functions of $t$ which are uniformly bounded away from zero and from above). It is easy to check that
\[\bigcup_{j=1}^\infty W_j = \operatorname{Tot}(E_{i_0}) \times \mathbb{C}^{n''} = W.\]

Recall that the open set $U$ of $\widehat{M}$ can be identified with $\operatorname{Tot}\big(E_{i_0}\big) \times B''$, and $U \backslash Q_0$ can be identified with $(0,s_*) \times \operatorname{Tot}(\widetilde{P}) \times B''$. Using this identification, we will define a diffeomorphism $\Phi_j : W_j \subset W \to U_j \subset U$ as follows. If $p \not\in \{\hat{s} = 0\}$, so that $p = (\hat{s},w,z) \in (0,\infty) \times \operatorname{Tot}(\widetilde{P}) \times B_{R_j}(0)$, then we define
\begin{equation}
\Phi_j(\hat{s},w,z) = \left(K_j^{-1/2}\hat{s}, \; w, \;\exp_{\bar{x}'',j}\big(K_j^{-1/2}z\big)\right)	
\end{equation}
where $\exp_{\bar{x}'',j} : B_{R_j}(0) \subset \mathbb{C}^{n''} \cong T_{\bar{x}''}B'' \to B'' \subset N''$ is the exponential map at the point $\bar{x}'' := \textrm{pr}_{N' \to N''}(\bar{x}')$ of the manifold $\Big( N'', \; \sum_{k \geq 2, \, k \neq i_0} f_k^2 |_{Q_0} ( t_j ) \, g_k \Big)$.
 If $p \in \{\hat{s} = 0\}$, then we express $p$ as the point $(\vec{0}, z) \in W_j = \operatorname{Tot}(E_{i_0}) \times B_{R_j}(0)$, and define
 \begin{equation}
 \Phi_j(\vec{0},z) := \left(\vec{0}, \;\exp_{\bar{x}'',j}\big(K_j^{-1/2}z\big)\right)
 \end{equation}
For any $z \in B_{R_j}(0)$, we have $\exp_{\bar{x}'',j}(K_j^{-1/2}z) \in B_{K_j^{-1/2}R_j}(\bar{x}'') \subset N''$ which is still within the injectivity radius from $\bar{x}''$ under the metric $\sum_{k \geq 2, \, k \neq i_0} f_k^2 |_{Q_0} ( t_j ) \, g_k$ by our choice of $R_j$. Furthermore, we may assume that $\exp_{\bar{x}'',j}(K_j^{-1/2}z) \in B''$ for any $j$ and $z \in B_{R_j}(0)$, and $K_j^{-1/2}R_j$ does not exceed the distance between $Q_0$ and $Q_\ell$ at $t = T$, which is positive as the $\mathbb{CP}^{m+1}$-fibers do not collapse. All these combine to show that $\Phi_j$ is a diffeomorphism from $W_j$ onto $U_j \subset U$, where $U_j$ is given by
\begin{equation}
U_j \;:=\; \Big\{ \, x \in \widehat{M} \;:\; \operatorname{dist}_{g ( t_j )} ( x, Q_0 ) < K_j^{-1/2} R_j \;\; \text{and} \;\; \mathrm{pr} \big( \widehat{\pi}' ( x ) \big) \in B_{K_j^{-1/2}R_j}(\bar{x}'') \, \Big\} \,,
\end{equation}
where $\mathrm{pr} : N' \to N''$ is the projection map, and the ball $B_{K_j^{-1/2}R_j}(\bar{x}'')$ is with respect to the metric $\sum_{k \geq 2, \, k \neq i_0} f_k^2 |_{Q_0} ( t_j ) \, g_k$.

\subsubsection{$C^\infty$-convergence of components of $\Phi_j^*g_j(t)$}
Recall that $g_j(t) := K_j g(t_j + K_j^{-1}t)$. By pulling back $g_j(t)$ by $\Phi_j$, we can see that its components are given by:
\begin{align}\label{eq:pullback_by_Phi}
\Phi_j^* \, g_j (t) & \;=\; d \hat{s} \otimes d \hat{s} \;+\; \hat{h}_j^2 \, \Phi_j^*(\theta \otimes \theta) \;+\; \hat{f}_{1, j}^2 \, \pi^* g_1 \;+\; \hat{f}_{i_0, j}^2 \, \pi^* g_{i_0} \;+\; \sum_{k \neq 1,i_0} K_j \, \Phi_j^*(f_k^2 \pi^* g_k) \,,
\end{align}
where 
\begin{align*}
	\hat{h}_j ( \hat{s}, t ) & := K_j^{1/2} \, h \big( K_j^{-1/2} \hat{s}, \, t_j + K_j^{-1} t \big),\\
	\hat{f}_{i, j} ( \hat{s}, t ) & := K_j^{1/2} \, f_i \big( K_j^{-1/2} \hat{s}, \, t_j + K_j^{-1} t \big), & i & = 1, i_0.
\end{align*}

Our next task is to establish the $C^m$-estimates for $\hat{h}_j^2$ and $\hat{f}_{i,j}^2$. For this we need Shi's derivative estimates \cite{Shi}. By Theorem \ref{thm:nut-bolt-typeI},
\[| \mathrm{Rm} ( g (t) ) |_{g(t)} \leq \frac{C}{T - t }\]
on all of $\widehat{M}$. Recall that $K_j := (T-t_j)^{-1}$, so under the parabolic rescaling $g_j ( t ) = K_j \, g ( t_j + K_j^{-1} t )$, $t \in [-K_jt_j, K_j(T-t_j))$, this becomes
\[|\mathrm{Rm} ( g_j (t) ) |_{g_j (t)} \leq \frac{1}{K_j} \cdot \frac{C}{T-(t_j + K_j^{-1}t)} = \frac{C}{1 - t} \leq \frac{C}{1 - t_2}\]
for $t \leq t_2$, where $t_1 < t_2 < 1$ are arbitrary but fixed. Since the rescaled flows exist on $[ - K_j t_j, K_j(T - t_j) ) \supseteq [ t_1 - 1, t_2 ]$ for $j$ large (as $- K_j t_j \to - \infty$), we may apply Shi's interior derivative estimates \cite{Shi} on $[ t_1 - 1, t_2 ]$, so that for every $m \geq 1$,
\begin{equation}\label{eq:Shi_m}
| \nabla^m \mathrm{Rm} |_{g_j (t)} \;\leq\; C(m, S, t_1, t_2 ) \qquad \text{on } \{ 0 \leq \hat{s} \leq S \} \times [ t_1, t_2 ] \text{ for any fixed } S > 0.
\end{equation}
We will use Shi's estimates to recover the uniform $C^m$-bounds of $\hat{h}_j^2$ and $\hat{f}_{i,j}^2$:
\begin{lemma}
\label{lma:Cm-estimates}
For any $m,k \geq 0$, there exists $C = C(m,k,S,t_1,t_2) > 0$ such that
\begin{equation}\label{eq:Cmk}
\big\| \hat{h}_j \big\|_{C^{m,k}([0,S]\times[t_1,t_2])} + \big\| \hat{f}_{i_0, j}^2 \big\|_{C^{m,k}([0,S]\times[t_1,t_2])} + \big\| \hat{f}_{1, j}^2 \big\|_{C^{m,k}([0,S]\times[t_1,t_2])} \;\leq\; C.
\end{equation}
\end{lemma}

\begin{proof}
The Li--Yau-type estimate \eqref{eq:li-yau} (after \cite{LY}) applied to $u = f_1^2$ shows $h^2 \leq C \, f_1^2$. After rescaling, we get
\[\hat{h}_j \leq \sqrt{C} \, \hat{f}_{1, j}.\]
By the K\"ahler condition $q_ih = \partial_s(f_i^2)$ for any $i \geq 1$, we have $\partial_{\hat{s}} \hat{f}_{1, j}^2 = \hat{h}_j \leq \sqrt{C} \, \hat{f}_{1, j}$ and hence $\partial_{\hat{s}}\hat{f}_{1,j} \leq \frac{\sqrt{C}}{2}$ on $[0,S]$. By integration from $\hat{s} = 0$ and recalling that $\hat{f}_{1, j} ( 0, t) = 0$ and $\partial_{\hat{s}}\hat{f}_{1,j}(0,t) = \frac{1}{\sqrt{2}}$, and the mean value theorem, we get
\begin{equation}\label{eq:C0-estimates}
\hat{f}_{1, j} \;\leq\; \tfrac{\sqrt{C}}{2} \, \hat{s} \,, \qquad \hat{h}_j \;\leq\; \tfrac{C}{2} \, \hat{s} \,, \qquad \big| \, \hat{f}_{i_0, j}^2 ( \hat{s}, t ) - \hat{f}_{i_0, j}^2 ( 0, t ) \, \big| \;\leq\; \tfrac{q_{i_0} C}{4} \, \hat{s}^2
\end{equation}
for any $(\hat{s},t) \in [0,S] \times [t_1,t_2]$. The first two inequalities prove $\|\hat{f}_{1,j}^2\|_{C^0([0,S])} + \|\hat{h}_j\|_{C^0([0,S])} \leq C$ for some $C = C(S) > 0$. For $\hat{f}_{i_0,j}^2$, we observe that from
\[f_{i_0}^2(0,t) = 2(p_{i_0}-(m+1)q_{i_0})(T-t),\]
we have
\begin{equation*}
\hat{f}_{i_0, j}^2(0,t) \;=\; K_j \, f_{i_0}^2 \big(0, t_j + K_j^{-1} t \big) \;=\; 2 \big( p_{i_0} - ( m + 1 ) \, q_{i_0} \big) \,(1 - t) \geq \frac{1}{C} > 0
\end{equation*}
on $t \in [t_1,t_2]$. Since $\hat{f}_{i_0,j}^2$ is strictly increasing on $[0,S]$, combining with \eqref{eq:C0-estimates} we get
\begin{equation}\label{eq:C0-fi0}
0 < \frac{1}{C} \leq \hat{f}_{i_0,j}^2 (\hat{s},t) \leq C
\end{equation}
where $C$ depends on $S, t_1, t_2$, proving $\big\|\hat{f}_{i_0,j}^2\big\|_{C^0([0,S])} \leq C$.

Next we move on to the higher-order estimate for $\hat{f}_{i_0,j}^2$. Using the K\"ahler condition $q_k h = \partial_s (f_k^2)$, we have
\begin{equation*}
\partial_{\hat{s}} \hat{f}_{i_0, j}^2 \;=\; q_{i_0} \, \hat{h}_j \,, \qquad \partial_{\hat{s}} \hat{f}_{1, j}^2 \;=\; \hat{h}_j \,.
\end{equation*}
Hence, to prove the $C^m([0,S])$-estimates of $\hat{h}_j$ and $\hat{f}_{1,j}^2$, it suffices to prove $\big\|\hat{f}_{i_0,j}^2\big\|_{C^m([0,S])} \leq C$ for any $m \geq 1$. The K\"ahler condition above and the uniform bounds of $\|\hat{h}_j\|_{C^0([0,S])}$ and $\|\hat{f}_{i_0,j}^2\|_{C^0([0,S])}$ already imply the uniform bound for $\big\|\hat{f}_{i_0,j}^2\big\|_{C^1([0,S])}$. From \eqref{eq:C0-fi0}, we have a uniform bound for $\big\|\hat{f}_{i_0,j}\big\|_{C^1([0,S])}$ too.

For $m \geq 2$, we will make use of the sectional curvature expression from Corollary \ref{cor:sectional-curvatures}:
\[K_{\Phi_j^*g_j(t)}(\partial_{\hat{s}}, \widetilde{X}) = -\frac{\partial_s^2 f_{i_0}}{f_{i_0}}\]
for any $\widetilde{X} \in TW_j$ such that $\pi_*(\widetilde{X}) \in TN_{i_0}$ and $|\pi_*(\widetilde{X})|_{g_{i_0}} = 1$. After rescaling and dilation, we get
\begin{equation}\label{eq:partial-f-K}
\partial^2_{\hat{s}}\hat{f}_{i_0,j} = -\hat{f}_{i_0,j} \cdot K_{\Phi_j^*g_j(t)}(\partial_{\hat{s}}, \widetilde{X}).	
\end{equation}
Under the rescaled metric $\Phi_j^*g_j$ the orthonormal frame $\{ \partial_{\hat{s}}, \widetilde{X} / \hat{f}_{i_0, j} \}$ is \emph{parallel} along the normal geodesics: indeed, $\nabla_{\partial_{\hat{s}}}\partial_{\hat{s}} = 0$, and by the shape-operator formula \eqref{eq:shape-operator},
\[\nabla_{\partial_{\hat{s}}}\widetilde{X} = \nabla_{\widetilde{X}}\partial_{\hat{s}} + [\partial_{\hat{s}}, \widetilde{X}] = \frac{\partial_{\hat{s}}\hat{f}_{i_0,j}}{\hat{f}_{i_0,j}}\widetilde{X},\]
and so
\begin{equation*}
\nabla_{\partial_{\hat{s}}} \left( \frac{\widetilde{X}}{\hat{f}_{i_0, j}} \right) \;=\; \frac{ \partial_{\hat{s}} \hat{f}_{i_0, j} }{ \hat{f}_{i_0, j} } \, \frac{\widetilde{X}}{\hat{f}_{i_0, j}} \;-\; \frac{ \partial_{\hat{s}} \hat{f}_{i_0, j} }{ \hat{f}_{i_0, j}^2 } \, \widetilde{X} \;=\; 0.
\end{equation*}
Note that $\hat{f}_{i_0, j}$ is uniformly bounded away from $0$ on $[t_1,t_2]$. Consequently, we have
\[\partial_{\hat{s}}^m \left(\textrm{Rm}_{\Phi_j^*g_j(t)}(\partial_{\hat s},\tilde X/\hat f_{i_0,j},\tilde X/\hat f_{i_0,j},\partial_{\hat s})\right) = \big(\nabla^m_{\partial_{\hat s},\dots,\partial_{\hat s}}\mathrm{Rm}_{\Phi_j^*g_j(t)}\big)(\partial_{\hat s},\tilde X/\hat f_{i_0,j},\tilde X/\hat f_{i_0,j},\partial_{\hat s})\]
as there are no covariant derivative terms of the frame $\{ \partial_{\hat{s}}, \widetilde{X} / \hat{f}_{i_0, j} \}$. Since
\[K_{\Phi_j^*g_j(t)}(\partial_{\hat{s}},\widetilde{X}) = \frac{\textrm{Rm}_{\Phi_j^*g_j(t)}(\partial_{\hat{s}}, \widetilde{X}, \widetilde{X}, \partial_{\hat{s}})}{|\partial_{\hat{s}}|^2_{\Phi_j^*g_j(t)}|\widetilde{X}|_{\Phi_j^*g_j(t)}^2} = \textrm{Rm}_{\Phi_j^*g_j(t)}(\partial_{\hat{s}}, \widetilde{X}/\hat{f}_{i_0,j}, \widetilde{X}/\hat{f}_{i_0,j}, \partial_{\hat{s}}),\]
we get
\[\partial_{\hat s}^{\,m}\,\left(K_{\Phi_j^*g_j(t)}(\partial_{\hat s},\widetilde{X})\right)=\big(\nabla^m_{\partial_{\hat s},\dots,\partial_{\hat s}}\mathrm{Rm}_{\Phi_j^*g_j(t)}\big)(\partial_{\hat s},\tilde X/\hat f_{i_0,j},\tilde X/\hat f_{i_0,j},\partial_{\hat s}).\]
Shi's estimates \eqref{eq:Shi_m} then imply
\begin{equation}\label{eq:partialk-K}
\big|\partial_{\hat s}^{\,m}\,(K_{\Phi_j^*g_j(t)}(\partial_{\hat s},\widetilde{X}))\big|\;\leq\;|\nabla^m\mathrm{Rm}|_{\Phi_j^*g_j(t)} = |\nabla^m\textrm{Rm}|_{g_j(t)}\;\leq\;C(m,S,t_1,t_2)
\end{equation}
for any $m \geq 0$.

From \eqref{eq:partial-f-K}, we have for any $m \geq 2$:
\begin{equation}\label{eq:Cm-fi0}
\partial_{\hat s}^{\,m}\hat f_{i_0,j}=-\sum_{a=0}^{m-2}\binom{m-2}{a} \cdot \partial_{\hat s}^{\,a}(K_{\Phi_j^*g_j(t)}(\partial_{\hat s},\widetilde{X}))\cdot \partial_{\hat s}^{\,m-2-a}\hat f_{i_0,j},
\end{equation}
so combining with \eqref{eq:partialk-K} and the uniform bound of $\|\hat{f}_{i_0,j}\|_{C^1([0,S])}$ that we have already proved, one can prove by induction that $\|\hat{f}_{i_0,j}\|_{C^m([0,S])} \leq C(m,S,t_1,t_2)$ for any $m \geq 0$.

Then, the K\"ahler condition $\partial_{\hat{s}}\hat{f}_{i_0,j}^2 = q_{i_0}\hat{h}_j$ immediately shows $\|\hat{h}_j\|_{C^m([0,S])} \leq C(m,S,t_1,t_2)$ for any $m \geq 0$.  We observe that
\[\partial_{\hat{s}}\hat{f}_{1,j}^2 = \hat{h}_j \;\;\text{ and }\;\; \hat{f}_{1,j}(0,t) = 0 \implies \hat{f}_{1,j}^2(\hat{s},t) = \int_0^{\hat{s}}\hat{h}_j(\sigma,t)\,d\sigma,\]
so the uniform $C^{m+1}([0,S])$-bound for $\hat{h}_j$ implies $C^m([0,S])$-bound for $\hat{f}_{1,j}^2$.

Now we have proved the $C^m([0,S])$-estimates for $\hat{h}_j$, $\hat{f}_{i_0,j}^2$, and $\hat{f}_{1,j}^2$ for any $m \geq 0$ for the spatial derivatives. To complete the proof of \eqref{eq:Cmk}, we need to bound their time derivatives of all orders as well. Recall that $\widetilde{X}$ is a static vector field such that $\pi_*(\widetilde{X}) \in TN_{i_0}$ and $|\pi_*(\widetilde{X})|_{g_{i_0}} = 1$.

Therefore, combining with \eqref{eq:Cm-fi0}, to obtain a uniform bound on $\partial_t^k \partial_{\hat{s}}^m \hat{f}_{i_0,j}$, it suffices to obtain uniform bounds of $\partial_t^k \partial_{\hat{s}}^a\big(K_{\Phi_j^*g_j(t)}(\partial_{\hat s},\widetilde{X})\big)$ and $\partial_t^k\partial_{\hat{s}}^a \hat f_{i_0,j}^2$ for any $a = 0, 1, \cdots, m-2$. The first term can be expressed as 
\[\partial_t^k\Big[\big(\nabla^a_{\partial_{\hat s},\dots,\partial_{\hat s}}\mathrm{Rm}_{\Phi_j^*g_j(t)}\big)(\partial_{\hat s},\tilde X/\hat f_{i_0,j},\tilde X/\hat f_{i_0,j},\partial_{\hat s})\Big],\]
and using the evolution equation of $\textrm{Rm}$:
\begin{equation}
\label{eq:Box_Rm}
\partial_t \textrm{Rm} = \Delta\textrm{Rm} + \textrm{Rm} * \textrm{Rm} \implies \partial_t^k \textrm{Rm} = \Delta^k \textrm{Rm} + \nabla^{i_1}\textrm{Rm} * \cdots * \nabla^{i_l}\textrm{Rm},
\end{equation}
the first term can be estimated by $|\nabla^l\textrm{Rm}|$, which is uniformly bounded by Shi's estimates \eqref{eq:Shi_m}, and by the term $\partial_t^k \hat{f}_{i_0,j}^2$. Hence, by induction, it suffices to prove the base cases when $m = 0, 1$, i.e. the uniform bounds for $\partial_t^k \hat{f}_{i_0,j}^2$ and $\partial_t^k \partial_{\hat{s}}\hat{f}_{i_0,j}^2$ for all $k \geq 0$.

Observe that
\[\partial_t \hat{f}_{i_0,j}^2 = \partial_t |\widetilde{X}|^2_{\Phi_j^*g_j(t)}= -2\textrm{Ric}_{\Phi_j^*g_j(t)}(\widetilde{X},\widetilde{X}) \implies \partial_t^k \hat{f}_{i_0,j}^2 =-2\partial_t^{k-1}\big(\textrm{Ric}_{\Phi_j^*g_j(t)}(\widetilde{X},\widetilde{X})\big).\]
From the evolution equation \eqref{eq:Box_Rm} of $\textrm{Rm}$, one can then express $\partial_t^{k-1}\textrm{Ric}$ as a linear combination of $\nabla^l\textrm{Rm}$ terms with $l \leq 2k-2$. Note that $\widetilde{X}$ is static, so from Shi's estimates \eqref{eq:Shi_m} we get a bound on $\partial_t^k\hat{f}_{i_0,j}^2$ uniform on $[0,S] \times [t_1,t_2]$. For $\partial_t^k\partial_{\hat{s}}\hat{f}_{i_0,j}^2$, recalling that $\widetilde{X}/\hat{f}_{i_0,j}$ is parallel along $\partial_{\hat{s}}$, we express it as
\begin{align}
& \partial_t^k \partial_{\hat{s}}\hat{f}_{i_0,j}^2 \label{eq:d_td_sf_i0}\\
 & = -2\partial_t^{k-1}\partial_{\hat{s}}\big(\textrm{Ric}_{\Phi_j^*g_j(t)}(\widetilde{X},\widetilde{X})\big) = -2\partial_t^{k-1}\partial_{\hat{s}}\big[\textrm{Ric}_{\Phi_j^*g_j(t)}(\widetilde{X}/\hat{f}_{i_0,j},\widetilde{X}/\hat{f}_{i_0,j})\cdot \hat{f}_{i_0,j}^2\big]\nonumber\\
& = -2\partial_t^{k-1}\left\{(\nabla_{\partial_{\hat{s}}}\textrm{Ric}_{\Phi_j^*g_j(t)})(\widetilde{X}/\hat{f}_{i_0,j},\widetilde{X}/\hat{f}_{i_0,j}) \cdot \hat{f}_{i_0,j}^2 + \textrm{Ric}_{\Phi_j^*g_j(t)}(\widetilde{X}/\hat{f}_{i_0,j},\widetilde{X}/\hat{f}_{i_0,j}) \cdot \partial_{\hat{s}}\hat{f}_{i_0,j}^2\right\}\nonumber\\
& = -2\partial_t^{k-1}\left\{(\nabla_{\partial_{\hat{s}}}\textrm{Ric}_{\Phi_j^*g_j(t)})(\widetilde{X},\widetilde{X})  + \textrm{Ric}_{\Phi_j^*g_j(t)}(\widetilde{X},\widetilde{X}) \cdot \frac{\partial_{\hat{s}}\hat{f}_{i_0,j}^2}{\hat{f}_{i_0,j}^2}\right\}.\nonumber
\end{align}
Since $\partial_t^{k-1}(\nabla_{\partial_{\hat{s}}}\textrm{Ric}_{\Phi_j^*g_j(t)})(\widetilde{X},\widetilde{X})$ consists of linear combinations of $\nabla^l \textrm{Rm}$ terms with $l \leq 2k-1$ from the evolution equation of $\textrm{Rm}$, it is uniformly bounded according to Shi's estimates \eqref{eq:Shi_m}. The same holds for $\partial_t^{k-1}\big[\textrm{Ric}_{\Phi_j^*g_j(t)}(\widetilde{X},\widetilde{X})\big]$. Therefore, we obtain the uniform bound on $\partial_t^k\partial_{\hat{s}}\hat{f}_{i_0,j}^2$ for all $k \geq 0$ using induction on \eqref{eq:d_td_sf_i0}. Note that we have already proved the base case $k=0$, i.e. the uniform bound on $\partial_{\hat{s}}\hat{f}_{i_0,j}^2$.

Therefore, by induction on $m$ we have a bound on $\partial_t^k \partial_{\hat{s}}^m \hat{f}_{i_0,j}^2$ for any $k, m \geq 0$. Note also that
\[\big[\partial_t,\; \partial_{\hat s}\big] \;=\; \operatorname{Ric}_{\Phi_j^*g_j}\!\big(\partial_{\hat s}, \partial_{\hat s}\big)\;\partial_{\hat s}\,,\]
which has a uniform $C^m$-bound, so we have established a uniform bound 
\[\big\|\hat{f}_{i_0,j}^2\big\|_{C^{m,k}([0,S]\times[t_1,t_2])} \leq C(m,k,S,t_1,t_2).\]

Finally, recall that
\[\hat{h}_j = \frac{1}{q_{i_0}}\partial_{\hat{s}}\hat{f}_{i_0,j}^2 \qquad \text{ and } \qquad \hat{f}_{1,j}^2(\hat{s},t) = \int_0^{\hat{s}}\hat{h}_j(\sigma,t)\,d\sigma.\]
We have completed the proof of \eqref{eq:Cmk} for all $m, k \geq 0$.
\end{proof}

\begin{corollary}
\label{cor:h_and_f}
There exist functions $\hat{h}_\infty$, $\hat{f}_{i_0, \infty}$ and $\hat{f}_{1,\infty}$ on $(\hat{s},t) \in [0,\infty) \times (-\infty, 1)$ such that, after passing to subsequences,  we have
\[\big(\hat{h}_j, \hat{f}_{i_0,j}^2, \hat{f}_{1,j}^2\big) \to \big(\hat{h}_\infty, \hat{f}_{i_0,\infty}^2, \hat{f}_{1,\infty}^2\big)\ \qquad \text{as $j \to \infty$ in }\;\; C^\infty_{\mathrm{loc}}([0,\infty)\times(-\infty,1)).\]
\end{corollary}

\begin{proof}
It is an immediate consequence of Lemma \ref{lma:Cm-estimates}, the Arzel\`a--Ascoli theorem, and the standard diagonalization argument.
\end{proof}

After passing to the limit, we can only guarantee that $\hat{h}_\infty \geq 0$ and $\hat{f}_{1, \infty} \geq 0$ on $\{ \hat{s} > 0 \}$, and strict positivity of both is needed to make the limit $g_\infty (t)$ a Riemannian metric. Note that $\hat{f}_{i_0, \infty} > 0$ since $\hat{f}_{i_0,j} \geq \frac{1}{C} > 0$ where $C$ is independent of $j$.

\begin{lemma}
\label{lma:h-f-positive}
\[\hat{h}_\infty > 0 \text{ and } \hat{f}_{1,\infty} > 0 \quad \text{ on $\{\hat{s} > 0\}$}.\]	
\end{lemma}

\begin{proof}
 For each $j$ the rescaled metric is genuine, so $\hat{h}_j > 0$ and $\hat{f}_{1, j} > 0$ on $\{ \hat{s} > 0 \}$. They are related to the sectional curvatures by
 \[\partial_{\hat{s}}^2 \hat{h}_j = - K_{\Phi_j^*g_j(t)} ( \partial_{\hat{s}}, \xi / \hat{h}_j )\, \hat{h}_j, \qquad\partial_{\hat{s}}^2 \hat{f}_{1, j} = - K_{\Phi_j^*g_j(t)} ( \partial_{\hat{s}}, \widetilde{X}_1 / \hat{f}_{1, j} ) \, \hat{f}_{1, j}.\]
Both sectional curvatures are bounded on $[0,S] \times [t_1,t_2]$. Suppose $\Lambda = \Lambda(S,t_1,t_2) > 0$ is a constant such that 
\[| \partial_{\hat{s}}^2 u | \leq \Lambda \, u, \qquad \text{ on $[ 0, S ] \times [ t_1, t_2 ]$ for $u \in \{ \hat{h}_\infty, \hat{f}_{1, \infty} \}$}.\] If $u ( \hat{s}_0, t ) = 0$ for some $\hat{s}_0 > 0$, the zero is an interior minimum of a nonnegative function, so $\partial_{\hat{s}} u ( \hat{s}_0, t ) = 0$ as well.

The energy $E := ( \partial_{\hat{s}} u )^2 + u^2$ satisfies 
\[| \partial_{\hat{s}} E | \leq C(\Lambda) \, E \implies -C(\Lambda)\,E \leq \partial_{\hat{s}}E \leq C(\Lambda)\,E.\]
With $E ( \hat{s}_0, t ) = 0$, using Gronwall's inequality we can show $E(\hat{s},t) = 0$ for any $\hat{s} \in [0,S]$, forcing $u \equiv 0$ on $[ 0, S ]$. This contradicts the closing slopes of Proposition \ref{prop:nut-closing}, $\partial_{\hat{s}} \hat{h}_\infty ( 0, t ) = 1$ and $\partial_{\hat{s}} \hat{f}_{1, \infty} ( 0, t ) = 1 / \sqrt{2}$. Hence $\hat{h}_\infty > 0$ and $\hat{f}_{1, \infty} > 0$ on $\{ \hat{s} > 0 \}$.
\end{proof}

\subsubsection{Convergence of gauge terms}
We need to handle the gauge term $\Phi_j^*(\theta \otimes \theta)$ in \eqref{eq:pullback_by_Phi} and show that it converges to $\widetilde{\theta} \otimes \widetilde{\theta}$. Recall that the vector bundle $E_{i_0}$ is built upon the circle-bundle $\widetilde{P}$ whose Euler class is $a_1 + q_{i_0}a_{i_0}$ with connection $1$-form $\tilde{\theta} := \theta\big|_{\widetilde{P}}$, so that
\[d \widetilde{\theta} \;=\; \pi^*\omega_1 \;+\; q_{i_0} \, \pi^*\omega_{i_0} \,.\]
\begin{lemma}
\label{lma:gauge}
As $j \to \infty$, we have
\begin{equation}
\label{eq:Phi_jtheta}
\Phi_j^*\theta = \widetilde{\theta} + \Phi_j^*\alpha \to \widetilde\theta
\end{equation}
in $C^\infty_{\textup{loc}}$-topology on $W$.	
\end{lemma}
\begin{proof}
Let $\alpha := \theta - \widetilde{\theta}$ so that $\theta = \widetilde{\theta} + \alpha$. Then we have
\[d\alpha = \sum_{k \neq 1, i_0} q_k \, \pi^* \omega_k\]
which only involves the $B''$-directions of the $N''$-factor. By the Poincar\'e Lemma, any closed form on the contractible ball $B''$ is exact with primitives given in geodesic normal coordinates centered at $\bar{x}''$ by
\begin{equation*}
\alpha \big|_y \;=\; \int_0^1 t \; (\iota_y  d\alpha)( t y ) \, d t.
\end{equation*}
On the ball $B_R(\bar{x}'')$, we then have
\begin{equation*}
| \alpha | ( z ) \;\leq\; \tfrac{| z |}{2} \, \sup |d\alpha|  \,, \qquad | \partial_z^{\mu} \alpha | \;\leq\; C_{\mu} ( R ) \, \textstyle\sum_{l \leq \mu} \sup \sum_{k\not=1,i_0}| \partial_z^{l} \, (d\alpha) | \,.
\end{equation*}
As the $N''$-factors do not contract along the unrescaled flow $g(t)$, so after rescaling when we measure the rescaled metric $g_j(t)$, the $N''$-components become $K_j f_k^2 \, g_k$ with $K_j f_k^2 \geq c \, K_j \to \infty$, so we have the following decay estimates:
\[\|\alpha\|_{C^\mu(B'',g_j(t))} \leq C(B'')\| d \alpha \|_{C^\mu(B'',g_j(t))} \leq C(B'')\sum_{k\not=1,i_0} \|q_k\pi^*\omega_k\|_{C^\mu(B'', K_j f_k^2 g_k)} \leq O ( K_j^{-1} ).\]
Therefore, after pulling back by $\Phi_j$, we have $\Phi_j^*\theta = \widetilde{\theta} + \Phi_j^*\alpha \to \widetilde\theta$ in $C^\infty_{\textrm{loc}}$-topology on $W$, completing the proof.
\end{proof}

\subsubsection{Remaining $N''$-factors} We now show that the remaining term of $\Phi_j^*g_j(t)$ goes to a flat metric.  The last group of terms is given by
\[\sum_{k\not=1,i_0} K_j \Phi_j^*(f_k^2 \pi^*g_k)\]
where $\Phi_j$ takes $z$ in the $\mathbb{C}^{n''}$-factor to the point $\exp_{\bar{x}'',j}(K_j^{-1/2}z)$ in $N''$, and $\exp_{\bar{x}'',j} : B_{R_j}(0) \to N''$ is the exponential map based at $\bar{x}'' \in N''$ with respect to the metric $\sum_{k \neq 1, i_0} f_k^2 |_{Q_0} ( t_j ) \, g_k$. Therefore, via $\Phi_j$ the $B_{R_j}(0)$-factor of $W_j$ is mapped to a geodesic ball in $N''$ -- with respect to the metric $\sum_{k \neq 1, i_0} f_k^2 |_{Q_0} ( t_j ) \, g_k$ -- with radius $K_j^{-1/2}R_j$ centered at $\bar{x}''$. By our choice of $R_j$, we have
\[K_j^{-1/2}R_j \leq \frac{1}{2}\textrm{inj}\,\Big( N'' \, , \; \sum_{k \neq 1, i_0} f_k^2 |_{Q_0} ( t_j ) \, g_k \Big).\]
Therefore, $\Phi_j$ maps $B_{R_j}(0)$ injectively onto its image. We claim:
\begin{lemma}
\label{lma:flat}
As $j \to \infty$, we have
\[\sum_{k\not=1,i_0}K_j\Phi_j^*(f_k^2\pi^*g_k) \to g_{(\mathbb{C}^{n''}, \text{flat})}\]
in $C^\infty_{\textup{loc}}(W)$-topology.
\end{lemma}

\begin{proof}
For each $g_k$, we express its components in geodesic normal coordinates centered at $\bar{x}''$ (for simplicity we suppress the index $k$ for a moment):
\[g_{ab}(z) = \delta_{ab} - \frac{1}{3}R_{acbd}(\bar{x}'')z^cz^d + O\big(|z|^3\big),\]
where $R_{acbd}$ is the Riemann curvature of $g_k$, so it is independent of $j$ and $t$. Then by the pullback, we have
\begin{align*}
K_j\Phi_j^*\pi^*g_k & = K_j\Phi_j^*\pi^*\big(g_{a b}(z)dz^a \otimes dz^b\big)\\
& = K_jg\big( K_j^{-1/2} z \big)_{ab} \cdot K_j^{-1/2}dz^a \otimes K_j^{-1/2}dz^b\\
& =\left(\delta_{a b} \;-\; \tfrac{1}{3} \, K_j^{-1} \, R_{a c b d} ( \bar{x}'' ) \, z^c z^d \;+\; O \big( K_j^{-3/2} | z |^3 \big)\right)\,dz^a \otimes dz^b\,,
\end{align*}
so $K_j \Phi_j^*\pi^*g_k \to$ the Euclidean metric by the boundedness of Riemann curvature. The convergence is in $C^\infty$ on compact balls because
\[\big|\partial_z^\alpha\big[(K_j\Phi_j^*\pi^* g_k)_{ab} - \delta_{ab}\big]\big| \;=\; K_j^{-|\alpha|/2}\,\big|\partial_x^\alpha g_{ab}\big|\big(K_j^{-1/2}z\big) \;=\; O\big(K_j^{-1}\big) \,.\]
Finally we need to show the coefficients $\Phi_j^*f_k^2$ converge to spatial constants. We observe that from the K\"ahler condition $q_k h = \partial_s(f_k^2)$ and $f_1\big|_{Q_0}(t) = 0$, we have 
\[f_k^2(s, t_j + K_j^{-1}t) = f_k^2\big|_{Q_0} (t_j + K_j^{-1}t) +q_k \int_0^s h(\sigma,t_j+K_j^{-1}t)\,d\sigma.\]
After pulling back by $\Phi_j$ we get
\begin{align*}
f_k^2 \circ \Phi_j (\hat{s},t_j+K_j^{-1}t) & \;=\; f_k^2 \big|_{Q_0} \big( t_j + K_j^{-1} t \big) \;+\; q_k \, \int_0^{K_j^{-1/2}\hat{s}}h(\sigma,t_j+K_j^{-1}t)\,d\sigma\\
& = f_k^2 \big|_{Q_0} \big( t_j + K_j^{-1} t \big) + \frac{q_k}{K_j}\int_0^{\hat{s}}\hat{h}_j(\hat\sigma, t)\,d\hat\sigma\\
& = A_k + B_k (T-t_j - K_j^{-1}t) + \frac{q_k}{K_j}\int_0^{\hat{s}}\hat{h}_j(\hat\sigma, t)\,d\hat\sigma,
\end{align*}
for some constants $A_k > 0, B_k \in \mathbb{R}$. The positivity of $A_k$ follows from the fact that the $N_k$-factor in $Q_0$, for $k \not=1,i_0$, does not contract to a point. By the uniform $C^{m,k}_{\textrm{loc}}([0,\infty) \times (-\infty,1))$-estimates of $\hat{h}_j$ (see Lemma \ref{lma:Cm-estimates}), we conclude that for any $m, k \geq 0$,
\[f_k^2 \circ \Phi_j(\hat{s},t_j+K_j^{-1}t) \to A_k \text{ in $C^{m,k}_{\textrm{loc}}(W)$-sense}.\]
Finally, the limit is given by
\[\lim_{j\to\infty} \sum_{k\not=1,i_0}K_j\Phi_j^*(f_k^2\pi^*g_k) = \bigoplus_{k\not=1,i_0}A_k\,g_{(\mathbb{C}^{\dim_{\mathbb{C}}N_k}, \text{flat})}\]
which is a flat metric on $\mathbb{C}^{n''}$ after a change of variables, completing the proof.
\end{proof}

\subsubsection{Completion of the proof}
Finally, we can collect all convergence lemmas and give the proof of Theorem \ref{thm:contracting_Q0}.
\begin{proof}[Proof of Theorem \ref{thm:contracting_Q0}]
Recall that
\[\Phi_j^* \, g_j (t) \;=\; d \hat{s} \otimes d \hat{s} \;+\; \hat{h}_j^2 \, \Phi_j^*(\theta \otimes \theta) \;+\; \hat{f}_{1, j}^2 \, \pi^* g_1 \;+\; \hat{f}_{i_0, j}^2 \, \pi^* g_{i_0} \;+\; \sum_{k \neq 1,i_0} K_j \, \Phi_j^*(f_k^2 \pi^* g_k) \,.\]
From Corollary \ref{cor:h_and_f} and Lemma \ref{lma:h-f-positive}, we know -- after passing to a subsequence -- that $\hat{h}_j^2 \to \hat{h}_\infty^2$, $\hat{f}_{1,j}^2 \to \hat{f}_{1,\infty}^2$, and $\hat{f}_{i_0,j}^2 \to \hat{f}_{i_0,\infty}^2$ in $C^\infty(W)$-sense as $j \to \infty$ for some positive functions $\hat{h}_\infty$, $\hat{f}_{1,\infty}$ and $\hat{f}_{i_0,\infty}$. From Lemma \ref{lma:gauge}, we know that $\Phi_j^*(\theta\otimes\theta) \to \widetilde{\theta} \otimes \widetilde{\theta}$ in $C^\infty_{\textup{loc}}(W)$-topology where $\widetilde{\theta} = \theta\big|_{\widetilde{P}}$ for the circle-bundle $\widetilde{P}$ over $\mathbb{CP}^m \times N_{i_0}$ with Euler class $a_1 + q_{i_0}a_{i_0}$. Then, from Lemma \ref{lma:flat} we proved that $\sum_{k\not=1,i_0}K_j \Phi_j^*(f_k^2\pi^*g_k) \to \text{flat metric}$ in $C^\infty_{\textup{loc}}(W)$-topology.

Combining everything, the pullback rescaled metric
\[\Phi_j^*g_j(t) \to \underbrace{\left(d\hat{s}\otimes d\hat{s} + \hat{h}_\infty^2 \widetilde{\theta}\otimes\widetilde{\theta} + \hat{f}_{1,\infty}^2\pi^*g_1 + \hat{f}_{i_0,\infty}^2 \pi^*g_{i_0}\right)}_{g_E(t)} \oplus g_{\mathbb{C}^{n''}}\]
in $C^\infty_{\textup{loc}}(W)$-topology. By \cite{EMT, Naber}, the pointed Cheeger-Gromov limit $\Big(\operatorname{Tot}(E_{i_0}) \times \mathbb{C}^{n''}, g_E(t) \oplus g_{\mathbb{C}^{n''}}\Big)$ must be a shrinking K\"ahler-Ricci soliton as we have a Type I blow-up. Given the splitting structure and the flatness of the second factor, $\big(\operatorname{Tot}(E_{i_0}),g_E(t)\big)$ is also a K\"ahler-Ricci shrinker. Since the metric $g_E(t)$ satisfies the circle-bundle ansatz \eqref{eq:ansatz}, the profile functions $\hat{h}_\infty$, $\hat{f}_{1,\infty}$, and $\hat{f}_{i_0,\infty}$ then satisfy the ODE system (3.2)--(3.4) in Dancer--Wang's work \cite{DW2011}. By our contraction condition, we have $0 < (m+1)q_{i_0} < p_{i_0}$ from Lemma \ref{lma:contraction_Q0_signs}, so $g_E(t)$ is the K\"ahler--Ricci shrinker constructed in \cite[Theorem 3.36]{DW2011}. This completes the proof of the theorem in the case $\mathcal{I}_0 = \{i_0\}$.
\end{proof}

\begin{remarkx}
If there is more than one element in $\mathcal{I}_0$, the argument is largely similar. We just need to replace $N_{i_0}$ by $N_0 :=\prod_{k \in \mathcal{I}_0}N_k$, and define $N''$ as $\prod_{k\not=1,k\not\in\mathcal{I}_0}N_k$. The vector bundle $E_{i_0}$ would be the restriction of $E$ to the slice over $N_0$, and $\widetilde{P}$ becomes the $U(1)$-bundle over $\mathbb{CP}^m \times N_0$ with Euler class $a_1 + \sum_{k \in \mathcal{I}_0} q_k a_k$. The map $\Phi_j$ is defined in a similar way since each hypersurface $\{\hat{s} = s_0\}$ is invariant under the map, but the pullback of $g_j(t)$ would become
\[\Phi_j^* \, g_j (t) \;=\; d \hat{s} \otimes d \hat{s} \;+\; \hat{h}_j^2 \, \Phi_j^*(\theta \otimes \theta) \;+\; \hat{f}_{1, j}^2 \, \pi^* g_1 \;+\; \sum_{k \in \mathcal{I}_0}\hat{f}_{k, j}^2 \, \pi^* g_k \;+\; \sum_{k \neq 1, k\not\in \mathcal{I}_0} K_j \, \Phi_j^*(f_k^2 \pi^* g_k).\]
The proof of uniform $C^{m,k}$-estimates of each $\{\hat{f}_{k,j}^2\}_j$ for each $k \in \mathcal{I}_0$ is similar to that of $\hat{f}_{i_0,j}^2$. The other components including convergence of $\Phi_j^*(\theta\otimes\theta)$ and $K_j\Phi_j^*(f_k^2\pi^*g_k)$ with $k \not= 1$ and $k \in \mathcal{I}_0$ are identical to the $\mathcal{I}_0 = \{i_0\}$ case.
\end{remarkx}

\subsection{Contraction of factors of $Q_\ell$}

We now turn to the third regime, $T = \min_{k \geq 2}\{T_k^\ell\} < \min\big\{T_F', \min_{k\geq 2}\{T_k^0\}\big\}$, in which some (not necessarily all) of the $N_{k \geq 2}$-factors of the bolt $Q_\ell$ contract to a point, while the $\mathbb{CP}^{m+1}$-fibers still have positive volume at the singular time. The result and its proof closely parallel Theorem \ref{thm:contracting_Q0}, with two structural differences. First, by Proposition \ref{prop:bolt-closing}, a tubular neighborhood of $Q_\ell$ is modeled on the total space of the complex \emph{line} bundle $L := P \times_{\mathrm{U}(1)} \mathbb{C}$ over $N = \mathbb{CP}^m \times N'$ instead of a rank-$(m+1)$ vector bundle, so the limit model is now the total space of a line bundle. Second, the $\mathbb{CP}^m$-factor of $Q_\ell$ does not contract in this regime --- indeed $f_1^2\big|_{Q_\ell}(t) = L(t) \geq L(T) > 0$ since $T < T_F'$ --- so under the rescaling the $\mathbb{CP}^m$-directions blow up and flatten into a factor of $\mathbb{C}^m$, which joins the non-contracting $N_k$-factors in the Euclidean factor of the limit.

\begin{theorem}[Contraction at $Q_\ell$]
\label{thm:contracting_Ql}
Let $\widehat{M}$ be the compactification of $(0,\ell) \times P$ such that $Q_0 := N' = N_2 \times \cdots \times N_r$ is adjoined at $\{s = 0\}$, and $Q_\ell := \mathbb{CP}^m \times N'$ is adjoined at $\{s = \ell\}$. Suppose
$T = \min_{k \geq 2}\{T_k^\ell\} < \min\big\{T_F', \min_{k\geq 2}\{T_k^0\}\big\}$.

Let $\mathcal{I}_\infty$ be the subset of $\{2, \cdots, r\}$ such that $i_\infty \in \mathcal{I}_\infty$ if and only if $T_{i_\infty}^\ell = T$. Let $\mathcal{L}_{\mathcal{I}_\infty}$ be the line bundle obtained by restricting the normal bundle $\mathcal{N}_{Q_\ell}$ of $Q_\ell \subset \widehat{M}$ --- the line bundle $L = P \times_{\mathrm{U}(1)} \mathbb{C}$ of Proposition \ref{prop:bolt-closing} --- to $\prod_{i_\infty \in \mathcal{I}_\infty} N_{i_\infty}$; then $c_1 ( \mathcal{L}_{\mathcal{I}_\infty} ) = \sum_{i_\infty \in \mathcal{I}_\infty} q_{i_\infty} a_{i_\infty}$, with each $q_{i_\infty} < 0$.

Let $K_j := (T-t_j)^{-1}$, and fix $\bar{x}' \in Q_\ell$. Consider the rescaled and dilated sequence $g_j(t) = K_j g(t_j + K_j^{-1}t)$ defined on $t \in [-K_j t_j, K_j(T-t_j))$. Then, along a subsequence,
\begin{equation}\label{eq:limit_Ql}
\begin{gathered}
\big( \widehat{M}, \; g_j (t), \; \bar{x}' \big) \;\longrightarrow\; \Big( \operatorname{Tot} \big( \mathcal{L}_{\mathcal{I}_\infty} \big) \times \mathbb{C}^{n''} \,, \;\; g_L (t) \oplus g_{\mathbb{C}^{n''}}, \bar{x}'\Big) \,, t \in \big( - \infty, 1 \big).
\end{gathered}
\end{equation}
Here $n'' = m + \sum_{k \geq 2, \, k \not\in \mathcal{I}_\infty} \dim_{\mathbb{C}} N_{k}$, $\operatorname{Tot} ( \mathcal{L}_{\mathcal{I}_\infty} )$ is the total space of the line bundle $\mathcal{L}_{\mathcal{I}_\infty} \to \prod_{i_\infty \in \mathcal{I}_\infty} N_{i_\infty}$, and $g_L$ is a complete gradient shrinking K\"ahler--Ricci soliton of ansatz form on $\operatorname{Tot} ( \mathcal{L}_{\mathcal{I}_\infty} )$ --- the Feldman--Ilmanen--Knopf shrinker in \cite{FIK} when the base is a projective space, and its multi-base generalization by Dancer--Wang in \cite{DW2011} in general.
\end{theorem}

\begin{remarkx}
When $m = 0$, Theorem \ref{thm:contracting_Ql} contains the contracting-bolt regime of the two-bolt flow of Theorem \ref{thm:two-bolt-typeI}: the limit is the FIK-type shrinker on the line bundle $\mathcal{L}_{i_0}$ of \cite{FT}, whose first Chern class $-|q_{i_0}|\,a_{i_0}$ agrees with the formula above since $q_{i_\infty} < 0$ at a contracting bolt.
\end{remarkx}

\begin{proof}[Sketch of proof]
The proof follows the same steps as that of Theorem \ref{thm:contracting_Q0}; we indicate the necessary modifications. As there, we assume $\mathcal{I}_\infty = \{i_\infty\}$; the general case follows as in the remark after the proof of Theorem \ref{thm:contracting_Q0}, replacing $N_{i_\infty}$ by $\prod_{k \in \mathcal{I}_\infty} N_k$.

\emph{Signs.} The analogue of Lemma \ref{lma:contraction_Q0_signs} is
\[0 \;<\; -q_{i_\infty} \;<\; p_{i_\infty}\,, \qquad \text{i.e. } p_{i_\infty} + q_{i_\infty} > 0 > q_{i_\infty}\,.\]
Indeed, $f_{i_\infty}^2\big|_{Q_\ell}(t) = f_{i_\infty}^2\big|_{Q_\ell}(0) - 2(p_{i_\infty} + q_{i_\infty})\,t$ vanishes exactly at $T = T_{i_\infty}^\ell$, so
\[f_{i_\infty}^2\big|_{Q_\ell}(t) \;=\; 2\big(p_{i_\infty} + q_{i_\infty}\big)\big(T-t\big)\]
and hence $p_{i_\infty} + q_{i_\infty} > 0$. Integrating the K\"ahler condition across the fiber as in Lemma \ref{lma:contraction_Q0_signs},
\[f_{i_\infty}^2 \big|_{Q_\ell} (t) \;-\; f_{i_\infty}^2 \big|_{Q_0} (t) \;=\; q_{i_\infty} \, L (t) \,.\]
As $t \to T$ the left-hand side tends to $-f_{i_\infty}^2\big|_{Q_0}(T) < 0$ --- negative because $T = T_{i_\infty}^\ell < T_{i_\infty}^0$ --- while $L(T) > 0$ since $T < T_F'$. This proves $q_{i_\infty} < 0$. By adjunction along the zero section, $c_1\big(\operatorname{Tot}(\mathcal{L}_{i_\infty})\big)\big|_{N_{i_\infty}} = (p_{i_\infty} + q_{i_\infty})\,a_{i_\infty} > 0$, which is exactly the existence condition for the Dancer--Wang shrinker on $\operatorname{Tot}(\mathcal{L}_{i_\infty})$.

\emph{Step 1: topological model near $Q_\ell$.} Here $\widehat{M} \backslash Q_0$ is the total space of the line bundle $L = P \times_{\mathrm{U}(1)} \mathbb{C} \to N$, obtained by coning off the circle fibers of $P \to N$ (Proposition \ref{prop:bolt-closing}). Set
\[N'' := \mathbb{CP}^m \times \prod_{k \geq 2, \, k \neq i_\infty} N_k\,,\]
which now includes the $\mathbb{CP}^m$-factor, and let $B'' \subset N''$ be a contractible ball containing $\bar{x}'' := \mathrm{pr}_{N \to N''}(\bar{x}')$. As in Step 1 of the proof of Theorem \ref{thm:contracting_Q0}, the contractibility of $B''$ identifies $L\big|_{N_{i_\infty} \times B''}$ with the pullback of $\mathcal{L}_{i_\infty} := L\big|_{N_{i_\infty} \times \{\bar{x}''\}}$, and correspondingly identifies $P$ restricted over $N_{i_\infty} \times B''$ with $\widetilde{P} \times B''$, where $\widetilde{P}$ is the circle bundle over $N_{i_\infty}$ with Euler class $q_{i_\infty} a_{i_\infty}$ and connection $1$-form $\widetilde{\theta} := \theta\big|_{\widetilde{P}}$ satisfying $d\widetilde{\theta} = q_{i_\infty}\,\pi^*\omega_{i_\infty}$. Throughout, we use the inward arclength coordinate $u$ at $Q_\ell$ and its rescaled counterpart $\hat{u} := K_j^{1/2}u$, in which the K\"ahler condition \eqref{eq:kahler-condition} reads $\partial_u (f_i^2) = -q_i h$.

\emph{Step 2: the maps $\Phi_j$.} Set $W := \operatorname{Tot}(\mathcal{L}_{i_\infty}) \times \mathbb{C}^{n''}$ and $W_j := \{\hat{u} < R_j\} \times B_{R_j}(0)$ with
\[R_j \;\leq\; \tfrac{1}{2}\sqrt{K_j}\;\operatorname{inj}\Big(N''\,, \; L(t_j)\,g_1 \oplus \sum_{k \geq 2,\, k \neq i_\infty} f_k^2\big|_{Q_\ell}(t_j)\,g_k\Big), \qquad \text{ and } R_j \to +\infty\,;\]
the injectivity radius is uniformly bounded from below since $L(t_j) \geq L(T) > 0$ and $\frac{1}{C} \leq f_k^2\big|_{Q_\ell}(t_j) \leq C$ for $k \neq 1, i_\infty$. The map $\Phi_j(\hat{u}, w, z) := \big(K_j^{-1/2}\hat{u}, \; w, \; \exp_{\bar{x}'',j}(K_j^{-1/2}z)\big)$ is then defined verbatim as in the proof of Theorem \ref{thm:contracting_Q0}, with $\exp_{\bar{x}'',j}$ the exponential map of the above metric on $N''$ at $\bar{x}''$. The pullback metric becomes
\[\Phi_j^* \, g_j (t) \;=\; d \hat{u} \otimes d \hat{u} \;+\; \hat{h}_j^2 \, \Phi_j^*(\theta \otimes \theta) \;+\; \hat{f}_{i_\infty, j}^2 \, \pi^* g_{i_\infty} \;+\; \sum_{k \geq 1, \, k \neq i_\infty} K_j \, \Phi_j^*(f_k^2 \pi^* g_k) \,;\]
compared with \eqref{eq:pullback_by_Phi}, the $g_1$-block has moved from the soliton block to the flattening group of terms.

\emph{Step 3: $C^{m,k}$-estimates.} Lemma \ref{lma:Cm-estimates} holds for $\hat{h}_j$ and $\hat{f}_{i_\infty,j}^2$ with the same proof, with $(\hat{s}, i_0)$ replaced by $(\hat{u}, i_\infty)$, and with the nut-closing conditions of Proposition \ref{prop:nut-closing} replaced by the bolt-closing conditions of Proposition \ref{prop:bolt-closing}: $\hat{h}_j(0,t) = 0$, $\partial_{\hat{u}}\hat{h}_j(0,t) = 1$, and $\hat{f}_{i_\infty,j}$ extends to a smooth even function with
\[\hat{f}_{i_\infty, j}^2(0,t) \;=\; K_j \, f_{i_\infty}^2\big|_{Q_\ell}\big(t_j + K_j^{-1}t\big) \;=\; 2\big(p_{i_\infty} + q_{i_\infty}\big)(1-t) \;\geq\; \frac{1}{C} \;>\; 0 \quad \text{on } [t_1,t_2]\,.\]
For the $C^0$-estimates, the Li--Yau-type bound of Proposition \ref{prop:li-yau} applied to $u = f_{i_\infty}^2$ gives $q_{i_\infty}^2 h^2 = |\nabla(f_{i_\infty}^2)|^2 \leq Cf_{i_\infty}^2$, hence $\hat{h}_j \leq C\hat{f}_{i_\infty,j}$; moreover $\partial_{\hat{u}}\hat{f}_{i_\infty,j}^2 = -q_{i_\infty}\hat{h}_j = |q_{i_\infty}|\,\hat{h}_j \geq 0$, so $\hat{f}_{i_\infty,j}^2$ is nondecreasing in $\hat{u}$ and uniformly bounded on $[0,S] \times [t_1,t_2]$. The higher-order estimates are verbatim: the identity $\partial^2_{\hat{u}}\hat{f}_{i_\infty,j} = -K_{\Phi_j^*g_j(t)}\big(\partial_{\hat{u}}, \widetilde{X}/\hat{f}_{i_\infty,j}\big)\,\hat{f}_{i_\infty,j}$ in the parallel frame, Shi's derivative estimates \cite{Shi} combined with the Type I bound (Theorem \ref{thm:nut-bolt-typeI}), and the scale-invariant relation $\hat{h}_j = |q_{i_\infty}|^{-1}\partial_{\hat{u}}\big(\hat{f}_{i_\infty,j}^2\big)$ recover all space-time derivative bounds, so Corollary \ref{cor:h_and_f} carries over. Positivity of $\hat{h}_\infty$ on $\{\hat{u} > 0\}$ follows from the energy/Gronwall argument of Lemma \ref{lma:h-f-positive}, now applied to the single function $\hat{h}_\infty$, using the closing slope $\partial_{\hat{u}}\hat{h}_\infty(0,t) = 1$; positivity of $\hat{f}_{i_\infty,\infty}$ is immediate from the uniform lower bound above.

\emph{Step 4: gauge.} With $\alpha := \theta - \widetilde{\theta}$ we now have
\[d\alpha \;=\; \pi^*\omega_1 \;+\; \sum_{k \geq 2, \, k \neq i_\infty} q_k \, \pi^*\omega_k\,,\]
which only involves the $B''$-directions. The decay argument of Lemma \ref{lma:gauge} applies once we note that on $U_j$ --- which lies within unrescaled distance $K_j^{-1/2}R_j \to 0$ of $Q_\ell$ --- we have $f_1^2 \geq L(t_j) - o(1) \geq c > 0$, so the rescaled metric coefficients along all the $N''$-directions satisfy $K_jf_k^2 \geq cK_j$ for $k \geq 1$ and $k \not= i_\infty)$, giving $\|\alpha\|_{C^\mu(B'', g_j(t))} = O(K_j^{-1})$ as before.

\emph{Step 5: flattening.} Lemma \ref{lma:flat} now covers the group of terms indexed by $k \in \{1\} \cup \{k \geq 2 : k \neq i_\infty\}$. The Taylor-expansion argument is unchanged; for the coefficients, integrating the K\"ahler condition from $Q_\ell$ gives
\[f_1^2 \circ \Phi_j \;=\; L\big(t_j + K_j^{-1}t\big) \;-\; \frac{1}{K_j}\int_0^{\hat{u}}\hat{h}_j(\hat\sigma,t)\,d\hat\sigma \;\longrightarrow\; A_1 := L(T) > 0\,,\]
and similarly $f_k^2 \circ \Phi_j \to A_k := f_k^2\big|_{Q_\ell}(T) > 0$ for $k \geq 2$, $k \neq i_\infty$, so this group of terms converges in $C^\infty_{\textup{loc}}$ to a flat metric on $\mathbb{C}^{n''}$.

\emph{Completion.} Combining the steps, $\Phi_j^*g_j(t)$ converges in $C^\infty_{\textup{loc}}(W)$-topology to
\[\underbrace{\Big(d\hat{u}\otimes d\hat{u} \;+\; \hat{h}_\infty^2\,\widetilde{\theta}\otimes\widetilde{\theta} \;+\; \hat{f}_{i_\infty,\infty}^2\,\pi^*g_{i_\infty}\Big)}_{g_L(t)} \oplus \; g_{\mathbb{C}^{n''}}\]
on $\operatorname{Tot}(\mathcal{L}_{i_\infty}) \times \mathbb{C}^{n''}$. By \cite{EMT, Naber} and the Type I property, the limit is a shrinking K\"ahler-Ricci soliton; it satisfies the circle-bundle ansatz \eqref{eq:ansatz}, so the profile functions solve the ODE system (3.2)--(3.4) in \cite{DW2011}. The sign condition $0 < -q_{i_\infty} < p_{i_\infty}$ then identifies $g_L(t)$ as the shrinking K\"ahler--Ricci soliton constructed in \cite[Theorem 3.36]{DW2011} on the total space of the negative line bundle $\mathcal{L}_{i_\infty}$, which is the Feldman--Ilmanen--Knopf shrinker \cite{FIK} when $N_{i_\infty}$ is a projective space.
\end{proof}

\subsection{Borderline cases}

Finally, we consider the borderline regime
\[T \;=\; T_F' \;=\; \min_{k \geq 2}\big\{T_k^0, T_k^\ell\big\}\,,\]
in which the $\mathbb{CP}^{m+1}$-fibers collapse and, at the same first singular time, some of the $N_{k\geq 2}$-factors contract to a point. Note that when $k \geq 3$, and some of $N_2, \cdots, N_r$-factors do not contract to a point, it is not the Fano canonical case as in $\mathbb{CP}^{m+1}$-bundle over a \emph{single} K\"ahler-Einstein manifolds (i.e. $k = 2$). Furthermore, by the Remark following \eqref{eq:T_max}, for each $k \geq 2$ the difference $f_k^2\big|_{Q_\ell}(t) - f_k^2\big|_{Q_0}(t) = q_kL(t) \to 0$ as $t \to T$, so the $N_k$-factor of $Q_0$ contracts if and only if the $N_k$-factor of $Q_\ell$ contracts, and there is no need to distinguish the contraction into the $Q_0$ and $Q_\ell$ cases. The essential new feature is that the two closing strata $Q_0$ and $Q_\ell$ now remain within bounded distance of each other with respect to the rescaled metrics $g_j(t)$, and the limit model is compact in the fiber directions, and is the total space of a $\mathbb{CP}^{m+1}$-bundle.

\begin{theorem}[Borderline case]
\label{thm:borderline}
Let $\widehat{M}$ be the compactification of $(0,\ell) \times P$ as in Theorem \ref{thm:contracting_Q0}. Suppose
$T = T_F' = \min_{k \geq 2}\big\{T_k^0, T_k^\ell\big\}$.

Let $\mathcal{I} := \big\{k \in \{2,\cdots,r\} : T_k^0 = T\big\} = \big\{k \in \{2,\cdots,r\} : T_k^\ell = T\big\}$, the two sets being equal by the Remark following \eqref{eq:T_max}. Let $E_{\mathcal{I}} = \mathcal{L}_{\mathcal{I}}^{\oplus(m+1)}$ be the vector bundle obtained by restricting the normal bundle $\mathcal{N}_{Q_0}$ of $Q_0 \subset \widehat{M}$ to $\prod_{k \in \mathcal{I}} N_k$, as in Theorem \ref{thm:contracting_Q0}, so that $c_1(\mathcal{L}_{\mathcal{I}}) = -\sum_{k \in \mathcal{I}} q_k a_k$, and let
\[\widehat{Z} \;:=\; \mathbb{P}\big(\mathcal{O} \oplus E_{\mathcal{I}}\big)\]
be the $\mathbb{CP}^{m+1}$-bundle over $\prod_{k \in \mathcal{I}} N_k$ obtained by projectivization. 

Let $K_j := (T-t_j)^{-1}$, and fix any $\bar{x} \in \widehat{M}$. Consider the rescaled and dilated sequence $g_j(t) = K_j g(t_j + K_j^{-1}t)$ defined on $t \in [-K_j t_j, K_j(T-t_j))$. Then, along a subsequence,
\begin{equation}\label{eq:limit_borderline}
\big( \widehat{M}, \; g_j (t), \; \bar{x} \big) \;\longrightarrow\; \Big( \widehat{Z} \times \mathbb{C}^{n''} \,, \;\; g_{\widehat{Z}} (t) \oplus g_{\mathbb{C}^{n''}} \Big) \,, \qquad t \in \big( - \infty, 1 \big)\,,
\end{equation}
where $n'' = \sum_{k \geq 2, \, k \not\in \mathcal{I}} \dim_{\mathbb{C}} N_{k}$, and $g_{\widehat{Z}}(t)$ is the unique compact shrinking K\"ahler--Ricci soliton of ansatz form on $\widehat{Z}$ --- constructed by Koiso \cite{Koiso} and Cao \cite{Cao96} when the base is a single projective space, and by Dancer--Wang \cite{DW2011} in general, with uniqueness given by Tian--Zhu \cite{TZhu}.

In the extreme case $\mathcal{I} = \{2, \cdots, r\}$, we have $n'' = 0$ and $\widehat{Z} = \widehat{M}$: then $[\omega(t)] = 2(T-t)\,c_1(\widehat{M})$, which is the Fano canonical-class case, and the blow-up limit is $\widehat{M}$ itself equipped with its compact shrinking K\"ahler-Ricci soliton.
\end{theorem}

\begin{proof}[Sketch of proof]
The proof combines the constructions in the proofs of Theorems \ref{thm:contracting_Q0} and \ref{thm:contracting_Ql}; we indicate the modifications.

\emph{Scale bookkeeping.} Under the rescaling $K_j = (T-t_j)^{-1}$, all fiber quantities now live at the same bounded scale: $K_jL(t_j + K_j^{-1}t) = 2(m+2)(1-t)$, and for each $k \in \mathcal{I}$,
\[K_jf_k^2\big|_{Q_0} \;=\; 2\big(p_k - (m+1)\,q_k\big)(1-t)\,, \qquad K_jf_k^2\big|_{Q_\ell} \;=\; 2\big(p_k + q_k\big)(1-t)\,,\]
so by the monotonicity of $f_k^2$ in $s$ we get $\frac{1}{C} \leq K_jf_k^2 \leq C$ on $\widehat{M} \times [t_1,t_2]$. Moreover, the fibers of $\widehat{\pi}'$ satisfy $\mathrm{diam}_{g(t)} \leq C\sqrt{T-t}$ as in the proof of Theorem \ref{thm:fiber-collapse-model}, while the Li--Yau-type estimate (Proposition \ref{prop:li-yau}) applied to $f_1^2$ gives $h^2 \leq Cf_1^2 \leq CL$, whence the arclength satisfies $\ell(t) \geq L(t)/\max h \geq c\sqrt{T-t}$. Therefore, the rescaled arclength $\hat\ell_j := K_j^{1/2}\,\ell(t_j + K_j^{-1}t)$ satisfies $\frac{1}{C} \leq \hat\ell_j \leq C$: the two strata $Q_0$ and $Q_\ell$ stay within bounded $g_j(t)$-distance of each other, and any fixed base point $\bar{x}$ stays within bounded $g_j(t)$-distance of both. This is why the base point may be chosen arbitrarily.

\emph{Topological model.} By Propositions \ref{prop:bolt-closing} and \ref{prop:nut-closing}, $\widehat{M}$ is globally the projectivization $\mathbb{P}(\mathcal{O}\oplus E)$ of the rank-$(m+1)$ bundle $E = \mathcal{L}^{\oplus(m+1)} \to N'$ from the proof of Theorem \ref{thm:contracting_Q0}, with $Q_0$ the zero section and $Q_\ell$ the section at infinity. Set $N'' := \prod_{k \geq 2, \, k\not\in\mathcal{I}}N_k$, and fix a contractible ball $B'' \subset N''$ containing $\bar{x}'' := \mathrm{pr}_{N' \to N''}\big(\widehat{\pi}'(\bar{x})\big)$. Restricting over $\big(\prod_{k\in\mathcal{I}}N_k\big) \times B''$ and using the contractibility of $B''$ identifies this restriction with $\widehat{Z} \times B''$. Away from both strata, it is parametrized as $(0,\ell) \times \operatorname{Tot}(\widetilde{P}) \times B''$, where $\widetilde{P}$ is the circle bundle over $\mathbb{CP}^m \times \prod_{k\in\mathcal{I}}N_k$ with Euler class $a_1 + \sum_{k\in\mathcal{I}}q_ka_k$ and $d\widetilde\theta = \pi^*\omega_1 + \sum_{k\in\mathcal{I}}q_k\,\pi^*\omega_k$ --- exactly the $\widetilde{P}$ of the multi-factor case in the proof of Theorem \ref{thm:contracting_Q0}.

\emph{The maps $\Phi_j$.} Set $W := \widehat{Z} \times \mathbb{C}^{n''}$ and $W_j := \widehat{Z} \times B_{R_j}(0)$: no truncation is needed in the $\widehat{Z}$-directions, which are compact. On the $\mathbb{C}^{n''}$-factor, $\Phi_j$ is $z \mapsto \exp_{\bar{x}'',j}(K_j^{-1/2}z)$ as before; on the $\widehat{Z}$-factor, $\Phi_j$ is the bundle identification above composed with the fiberwise radial reparametrization matching the radial coordinate of $\widehat{Z}$ with the rescaled arclength $\hat{s} \in [0,\hat\ell_j]$ --- a diffeomorphism for each $j$, since $\hat\ell_j$ is bounded between two positive constants. The pullback metric takes the ansatz form
\[\Phi_j^*g_j(t) \;=\; d\hat{s}\otimes d\hat{s} \;+\; \hat{h}_j^2\,\Phi_j^*(\theta\otimes\theta) \;+\; \hat{f}_{1,j}^2\,\pi^*g_1 \;+\; \sum_{k\in\mathcal{I}}\hat{f}_{k,j}^2\,\pi^*g_k \;+\; \sum_{k\geq 2,\,k\not\in\mathcal{I}}K_j\,\Phi_j^*(f_k^2\pi^*g_k)\,,\]
where now the $g_1$-block \emph{remains} in the soliton block: $\hat{f}_{1,j}^2 \leq K_jL = 2(m+2)(1-t)$ stays bounded.

\emph{$C^{m,k}$-estimates and convergence.} On $[0, \hat\ell_j]$ we run the estimates of Lemma \ref{lma:Cm-estimates} from both ends: from the nut end $\hat{s} = 0$ verbatim, using the closing conditions of Proposition \ref{prop:nut-closing} ($\hat{h}_j(0,t) = 0$, $\partial_{\hat{s}}\hat{h}_j(0,t)=1$, $\hat{f}_{1,j}(0,t) = 0$, $\partial_{\hat{s}}\hat{f}_{1,j}(0,t) = \frac{1}{\sqrt{2}}$), and from the bolt end $\hat{s} = \hat\ell_j$ in the inward coordinate $\hat{u} := \hat\ell_j - \hat{s}$, using the closing conditions of Proposition \ref{prop:bolt-closing} as in Step 3 of the proof of Theorem \ref{thm:contracting_Ql}. The $C^0$-bounds for $\hat{f}_{k,j}^2$, $k \in \mathcal{I}$, follow from the scale bookkeeping above; the higher-order bounds follow from the parallel-frame curvature identities, Shi's derivative estimates \cite{Shi}, and the Type I bound (Theorem \ref{thm:nut-bolt-typeI}), as before. After passing to a subsequence, $\hat\ell_j \to \hat\ell_\infty \in [\frac{1}{C}, C]$, and the profile functions converge in $C^\infty$ to limits $\big(\hat{h}_\infty, \hat{f}_{1,\infty}, \{\hat{f}_{k,\infty}\}_{k\in\mathcal{I}}\big)$ on $[0,\hat\ell_\infty] \times (-\infty,1)$ satisfying the closing conditions at both ends; interior positivity of $\hat{h}_\infty$ and $\hat{f}_{1,\infty}$ follows from the energy/Gronwall argument of Lemma \ref{lma:h-f-positive} run from either end. The limit profiles therefore define a smooth one-parameter family of ansatz metrics on the \emph{compact} manifold $\widehat{Z}$.

\emph{Gauge and flattening.} With $\alpha := \theta - \widetilde\theta$ we now have $d\alpha = \sum_{k\geq 2,\,k\not\in\mathcal{I}}q_k\,\pi^*\omega_k$ --- the $\omega_1$-term is absorbed into $d\widetilde\theta$ --- and the proof of Lemma \ref{lma:gauge} gives $\Phi_j^*\theta \to \widetilde\theta$, since the non-contracting factors satisfy $K_jf_k^2 \geq cK_j$. Lemma \ref{lma:flat} applies verbatim to the last group of terms, with $A_k = f_k^2\big|_{Q_0}(T) > 0$ for $k \geq 2$, $k \not\in \mathcal{I}$, yielding a flat limit metric on $\mathbb{C}^{n''}$.

\emph{Completion.} Combining the steps, the limit $\big(\widehat{Z}\times\mathbb{C}^{n''}, \, g_{\widehat{Z}}(t)\oplus g_{\mathbb{C}^{n''}}\big)$ is a shrinking K\"ahler-Ricci soliton by \cite{EMT, Naber} and the Type I property, so $g_{\widehat{Z}}(t)$ is a \emph{compact} shrinking K\"ahler-Ricci soliton of ansatz form on $\widehat{Z}$. The sign conditions $-p_k < q_k < \frac{p_k}{m+1}$ for $k \in \mathcal{I}$ say precisely that $c_1(\widehat{Z})$ is positive on both strata of $\widehat{Z}$:
\begin{align*}
c_1(\widehat{Z})\Big|_{\text{zero section}} & \;=\; \sum_{k\in\mathcal{I}}\big(p_k - (m+1)\,q_k\big)\,a_k \;>\; 0\,,\\
c_1(\widehat{Z})\Big|_{\text{infinity section}} & \;=\; (m+2)\,a_1 + \sum_{k\in\mathcal{I}}\big(p_k+q_k\big)\,a_k \;>\; 0\,,
\end{align*}
consistent with $\widehat{Z}$ being Fano. The profile functions satisfy the Dancer--Wang ODE system (3.2)--(3.4) in \cite{DW2011} with the two-ended closing boundary conditions, and the corresponding compact shrinking K\"ahler--Ricci soliton is the one constructed in \cite{Koiso}, \cite{Cao96}, and \cite{DW2011}; by the uniqueness theorem of Tian--Zhu \cite{TZhu}, $g_{\widehat{Z}}(t)$ must be this soliton. When $\mathcal{I} = \{2,\cdots,r\}$ there are no flattening directions, $\widehat{Z} = \widehat{M}$, and the statement reduces to the convergence of the canonical-class Fano K\"ahler-Ricci flow on $\widehat{M}$ to its unique compact shrinking K\"ahler-Ricci soliton.
\end{proof}

\section{Two-nut closing}\label{sec:twonut}
Finally, we discuss the case of two-nut closing, i.e. both ends are adjoined with one $\mathbb{CP}^m$ with collapsing $S^{2m+1}$. Under the K\"ahler condition, each $f_k$'s, $k \geq 1$, is strictly monotone. Therefore, it is necessary that to have two $\mathbb{CP}^m$-factors (with possibly different dimensions), and their charges $q$'s must have different signs.

After relabelling, we let $( N_1, g_1 ) = ( \mathbb{CP}^{m_0}, 2\, g_{FS} )$ with $p_1 = m_0 + 1$, whose sphere $S^{2 m_0 + 1}$ collapses at $s = 0$, and $( N_2, g_2 ) = ( \mathbb{CP}^{m_\ell}, 2\, g_{FS} )$ with $p_2 = m_\ell + 1$, collapsing at $s = \ell$, where $m_0, m_\ell \geq 1$ and $r \geq 2$; write $N'' := \prod_{k \geq 3} N_k$. The closing strata are $Q_0 \cong \mathbb{CP}^{m_\ell} \times N''$ and $Q_\ell \cong \mathbb{CP}^{m_0} \times N''$, and we define the bundle map to be the following submersion:
\[
\widehat{\pi}'' \;:\; \widehat{M} \;\longrightarrow\; N'' \,.
\]
The fibers are then $\mathbb{CP}^{m_0 + m_\ell + 1}$, and precisely $\widehat{M}$ is a projectivized bundle:
\[
\widehat{M} \;\cong\; \mathbb{P} \big( \mathcal{O}^{\oplus ( m_0 + 1 )} \oplus ( \mathcal{L}'' )^{\oplus ( m_\ell + 1 )} \big) \;\longrightarrow\; N'' \,, \qquad \text{ with } c_1 ( \mathcal{L}'' ) \;=\; \sum_{k \geq 3} q_k\, a_k \,,
\]
and the lower and upper ends are adjoined by $Q_\ell = \mathbb{P} \big( \mathcal{O}^{\oplus ( m_0 + 1 )} \big)$ and $Q_0 = \mathbb{P} \big( ( \mathcal{L}'' )^{\oplus ( m_\ell + 1 )} \big)$ respectively. For $r = 2$ this is $\mathbb{CP}^{m_0 + m_\ell + 1}$ itself, while $m_\ell = 0$ formally returns the nut--bolt closing $\mathbb{P} ( \mathcal{O} \oplus E )$, and $m_0 = m_\ell = 0$ the two-bolt one.

\begin{lemma}\label{lem:twonut}
Let $g (t)$, $t \in [ 0, T )$, be a K\"ahler-Ricci flow under the circle-bundle ansatz \eqref{eq:ansatz} on compact two-nut closing, with $s$ the distance to $Q_0$. Then:

\noindent (i) We must have $q_1 = + 1$ and $q_2 = - 1$

\noindent (ii) About the two collapsing components $f_1$ and $f_2$:

\[
f_1^2 ( s ) \;=\; \int_0^s h\, d\sigma \,, \qquad f_2^2 ( s ) \;=\; \int_s^\ell h\, d\sigma \,,
\]
so that, on all of $\widehat{M}$, we have
\begin{equation}\label{eq:twonut_L}
f_1^2 \;+\; f_2^2 \;=\; L (t) \;:=\; \int_0^\ell h\, d\sigma \,.
\end{equation}
Furthermore, $f_1$ is related to $f_{k \geq 3}$ by the relations $f_k^2 = f_k^2 \big|_{Q_0} + q_k\, f_1^2$ for $k \geq 3$.

\noindent (iii) At the closings $Q_0$ and $Q_\ell$, the components $f_i$'s are given by:
\begin{align*}
f_i^2 \big|_{Q_0} (t) & \;=\; f_i^2 \big|_{Q_0} (0) \;+\; 2 \big( ( m_0 + 1 )\, q_i - p_i \big)\, t \,,\\
 f_i^2 \big|_{Q_\ell} (t) & \;=\; f_i^2 \big|_{Q_\ell} (0) \;-\; 2 \big( ( m_\ell + 1 )\, q_i + p_i \big)\, t \,.
\end{align*}
In particular, for $i = 2$ at $Q_0$ --- equivalently, by \eqref{eq:twonut_L}, for $i = 1$ at $Q_\ell$ --- they give the ``height'' of the $\mathbb{CP}^{m_0 + m_\ell +1}$-fibers:
\begin{equation}\label{eq:twonut_Lrate}
L (t) \;=\; L (0) \;-\; 2 \big( m_0 + m_\ell + 2 \big)\, t \,,
\end{equation}
matching $\big\langle c_1 ( \widehat{M} ), [ \Lambda ] \big\rangle = ( m_0 + 1 ) + ( m_\ell + 1 )$ for a projective line $\Lambda$ in a fiber (\S\ref{subsec:periods-chern}).

\noindent (iv) There is $c > 0$ with

\[
f_k^2 \;\geq\; c\, ( T - t ) \quad ( k \geq 3 ) \qquad \text{and} \qquad L (t) \;\geq\; c\, ( T - t ) \qquad \text{on } \widehat{M} \times [ 0, T ) \,,
\]
each being a positive affine function of $t$, exactly as in Proposition \ref{prop:affine-laws}(iii). Note that no such bound holds for $f_1^2$ or $f_2^2$, as each of which vanishes on $Q_0$ and $Q_\ell$ respectively.

\end{lemma}

\subsection{The two-nut closing: Type I rate and singularity models}\label{subsec:twonut}

\begin{theorem}[the two-nut K\"ahler flow is Type I]\label{thm:twonut_typeI}
Let $\widehat{M}$ be a compact two-nut closing and let $g (t)$, $t \in [ 0, T )$, be a Ricci flow of ansatz metrics \eqref{eq:ansatz-metric} satisfying the K\"ahler condition \eqref{eq:kahler-condition}, with $T < \infty$ maximal. Then the singularity at $T$ is of Type I.

\end{theorem}

\begin{proof}[Sketch of proof]
Most of the proof of Theorem \ref{thm:nut-bolt-typeI} carries over with the two-nut case, with some slight modifications:

\emph{The residual bundle.} For $\widehat{\pi}''$ the vertical distribution is $\mathcal{V}'' = \mathbb{R} \partial_s \oplus \mathbb{R} \xi \oplus \mathcal{H}_1 \oplus \mathcal{H}_2$, of real rank $2 ( m_0 + m_\ell + 1 )$, and $\mathcal{H}'' = \bigoplus_{k \geq 3} \mathcal{H}_k$; the residual O'Neill tensor $A''$ is given by \eqref{eq:A-residual} with the sum running over $k \geq 3$. Only the blocks carrying the linear lower bound of Lemma \ref{lem:twonut}(iv) enter, so the Li--Yau estimate \eqref{eq:li-yau} gives $| A'' |^2 \leq C \sum_{k \geq 3} | \nabla f_k^2 |^2 / f_k^4 \leq C / ( T - t )$, and Propositions \ref{prop:A-residual-decay} and \ref{prop:typeII-splitting} apply verbatim: a Type II rescaling produces an eternal limit splitting as $\Phi \times \mathbb{C}^{n''}$, $n'' = \sum_{k \geq 3} n_k$, with all curvature carried by a limit of the totally geodesic fibers of $\widehat{\pi}''$.

\emph{Boundary conditions.} The boundary conditions of \S\ref{subsec:fiber-bisectional} persists: $f_1^2$ is strictly increasing in $s$, the pair $( f_1^2, t )$ is independent, and $h^2$ is a smooth function of $( f_1^2, t )$ on $[ 0, L (t) ]$, with $' = \partial / \partial ( f_1^2 )$. The boundary data are \emph{unchanged}: $h^2 = 0$, $( h^2 )' = 2$ at $f_1^2 = 0$, and $h^2 = 0$, $( h^2 )' = - 2$ at $f_1^2 = L (t)$ --- the second nut closes with $h_s \to - 1$, exactly the boundary behaviour the bolt had. In the mirror variable $f_2^2 = L (t) - f_1^2$ one has $d / d ( f_2^2 ) = - d / d(f_1^2)$ and $\partial_{f_2^2}^2 = \partial_{f_1^2}^2$, so the second derivative $( h^2 )''$ is smooth on the closed manifold $\widehat{M}$ including on $Q_\ell$, exactly as in the proof of Proposition \ref{prop:h2pp-evolution}.

\emph{Fiber curvature.} The most notable difference is the tangent space of fibers are much larger than the two-bolt and nut-bolt cases, as now the fibers are $\mathbb{CP}^{m_0+m_\ell+1}$'s. Proposition \ref{prop:fiber-curvature} holds as stated for planes drawn from $\mathbb{R} \partial_s \oplus \mathbb{R} \xi \oplus \mathcal{H}_1$; for $W, \widetilde{W} \in \mathcal{H}_2$ the same computation in the mirror variable gives the entries of \eqref{eq:fiber-curvature} under $( f_1^2 , ( h^2 )' ) \mapsto ( f_2^2 , - ( h^2 )' )$,

\begin{align*}
K ( \nu, W ) & \;=\; K ( \xi^*, W ) \;=\; \frac{1}{4 f_2^2} \left( \frac{h^2}{f_2^2} + ( h^2 )' \right) ,\\
 K ( W, J W ) & \;=\; \frac{2}{f_2^2} - \frac{h^2}{f_2^4} \,, \qquad K ( W, \widetilde{W} ) \;=\; \frac{1}{4 f_2^2} \left( 2 - \frac{h^2}{f_2^2} \right) ,
\end{align*}
while the mixed entries between the two collapsing blocks, $X \in \mathcal{H}_1$ and $W \in \mathcal{H}_2$, are
\[
K ( X, W ) \;=\; K ( X, J W ) \;=\; - \frac{q_1 q_2\, h^2}{4\, f_1^2 f_2^2} \;=\; \frac{h^2}{4\, f_1^2 f_2^2} \;\geq\; 0
\]
by Lemma \ref{lem:twonut}(i) --- hence the cross-term $K(X,W)$ and $K(X,JW)$ has favorable sign. For other bisectional curvature terms, we can repeat the integration argument of Lemma \ref{lem:h2pp-controls-bisec} using variables $f_2^2$ instead --- by noting that $\partial_{f_2^2} = -\partial_{f_1^2}$ and $\partial_{f_2^2}^2 = \partial_{f_1^2}^2$ --- the upper bound $( h^2 )'' \leq K (t)$ would also imply bisectional curvatures $K(\nu, W)$, $K(W,JW)$ and $K(W,\widetilde{W})$ are all bounded below by $- K (t)$ too with the same proof as $K(\nu,X)$, $K(X,JX)$ and $K(X,\widetilde{X})$.

\emph{Estimates of $(h^2)''$.} Under the variables $( f_1^2, t )$, the scalar equation \eqref{eq:reduced-flow} keeps its form, with $m = m_0$ and the $k = 2$ term written out via \eqref{eq:twonut_L}:

\begin{align}\label{eq:twonut_reduced}
\partial_t ( h^2 ) \big|_{f_1^2} & \;=\; h^2 ( h^2 )'' \;+\; 2 ( m_0 + 1 ) ( h^2 )' \;-\; \big( ( h^2 )' \big)^2 \\
& \qquad \;-\; \frac{m_0\, h^4}{f_1^4} \;-\; \frac{m_\ell\, h^4}{\big( L (t) - f_1^2 \big)^2} \;-\; \sum_{k \geq 3} \frac{n_k\, q_k^2\, h^4}{\big( a_k (t) + q_k f_1^2 \big)^2} \,, \nonumber
\end{align}
where $a_k (t) := f_k^2 \big|_{Q_0} (t)$; rewriting \eqref{eq:twonut_reduced} in the variables $( f_2^2, t )$ --- noting that
\[\partial_t\big|_{f_2^2}(h^2) - \partial_t\big|_{f_1^2}(h^2) = \frac{dL}{dt}\cdot(h^2)' = -2(m_0+m_\ell+2)(h^2)',\]
we would have
\begin{align}\label{eq:twonut_reduced_f_2}
\partial_t ( h^2 ) \big|_{f_2^2} & \;=\; h^2\partial_{f_2^2}^2  ( h^2 ) \;+\; 2 ( m_\ell + 1 ) \partial_{f_2^2}( h^2 ) \;-\; \big( \partial_{f_2^2}( h^2 ) \big)^2 \\
& \qquad \;-\; \frac{m_0\, h^4}{(L(t) - f_2^2)^2} \;-\; \frac{m_\ell\, h^4}{f_2^4} \;-\; \sum_{k \geq 3} \frac{n_k\, q_k^2\, h^4}{\big( \tilde{a}_k (t) + q_k f_2^2 \big)^2} \,, \nonumber
\end{align}
The proof of Theorem \ref{thm:h2pp-upper-bound} now runs symmetrically: Differentiating \eqref{eq:twonut_reduced} twice as in Proposition \ref{prop:h2pp-evolution} gives the same reaction--diffusion equation $$\partial_t ( h^2 )'' = \Delta ( h^2 )'' - \big( ( h^2 )'' \big)^2 + R,$$ with the sources now split as $R = R_1 + R_2 + R_{\mathrm{res}}$, where $R_1 := - m_0\, \partial_{f_1^2}^2 \big[ h^4 / f_1^4 \big]$, $R_2 := - m_\ell\, \partial_{f_1^2}^2 \big[ h^4 / f_2^4 \big]$, and $R_{\mathrm{res}} := - \sum_{k \geq 3} n_k q_k^2\, \partial_{f_1^2}^2 \big[ h^4 / f_k^4 \big]$.

Lemma \ref{lem:residual-source-bound} holds similarly for our $R_{\textrm{res}}$ here, as the $R_{\textrm{res}}$ involves only the second derivatives.

The proof of Theorem \ref{thm:h2pp-upper-bound} will be modified as follows: at an \emph{interior} spatial maximum of $( h^2 )''$ the integral identity for the ratio holds from \emph{each} nut --- integrate $h^2 / f_1^2$ from $f_1^2 = 0$ for $R_1$, and $h^2 / f_2^2$ from $f_2^2 = 0$ for $R_2$, the second derivative $\partial_{f_2^2}^2 = \partial_{f_1^2}^2$ being the same --- so $R_1 \leq 0$ \emph{and} $R_2 \leq 0$ there. At a maximum on $Q_0$, the one-sided expansion applies to $R_1$, whose unfavorable term $- \tfrac{4 m_0}{3} \big( ( h^2 )'' \big)' ( 0^+ )$ is compensated by $\Delta ( f_1^2 ) \big|_{Q_0} = 2 ( m_0 + 1 ) > \tfrac{4 m_0}{3}$, exactly as in the proof of Theorem \ref{thm:h2pp-upper-bound}; $R_2$ is now of residual type: $f_2^2 = L (t) - f_1^2$ is affine in $f_1^2$, so the completed square of Lemma \ref{lem:residual-source-bound} applies to this block verbatim, and $f_2^2 \big|_{Q_0} = L (t) \geq c\, ( T - t )$ gives $R_2 \leq \beta' / ( T - t )^2$. A maximum on $Q_\ell$ is the mirror case: $\Delta (f_2^2)\big|_{Q_\ell} = 2 ( m_\ell + 1 )$ compensates the $-\tfrac{4 m_\ell}{3}$ term in $R_2$, while $f_1^2 \big|_{Q_\ell} = L (t) \geq c\, ( T - t )$ makes $R_1$ residual. The Riccati comparison for $y (t) := \max_{\widehat{M}} ( h^2 )'' ( \cdot, t )$ is unchanged, so \eqref{eq:h2pp-upper-bound} holds with a constant $\gamma = \gamma ( m_0, m_\ell, n'', c, C_0, g_0 )$, and as in Corollary \ref{cor:bisec-lower-bound} the bisectional curvature satisfies $\mathrm{bisec} \geq - \gamma / ( T - t )$: a Type II limit has nonnegative bisectional curvature.

\emph{Final Step.} Cao's theorem \cite{Cao97} and the Deng--Zhu rigidity theorem \cite{DZ20} now apply word for word as in \S\ref{subsec:no-typeII}: no Type II singularity occurs, and the singularity at $T$ is of Type I.
\end{proof}

By Tian--Zhang \cite{TZ}, exactly as in \eqref{eq:T_max}, the maximal existence time is computed from the affine laws of Lemma \ref{lem:twonut}(iii): set

\[
T_F'' \;:=\; \frac{L (0)}{2 ( m_0 + m_\ell + 2 )} \,,
\]
and, for $k \geq 3$,
\begin{align*}
T_k^0 & \;:=\; \begin{cases} \dfrac{f_k^2 |_{Q_0} (0)}{2 \big( p_k - ( m_0 + 1 )\, q_k \big)} & \text{if } ( m_0 + 1 )\, q_k < p_k \,, \\[4pt] + \infty & \text{otherwise} \,, \end{cases}\\
T_k^\ell & \;:=\; \begin{cases} \dfrac{f_k^2 |_{Q_\ell} (0)}{2 \big( p_k + ( m_\ell + 1 )\, q_k \big)} & \text{if } ( m_\ell + 1 )\, q_k > - p_k \,, \\[4pt] + \infty & \text{otherwise} \,, \end{cases}
\end{align*}
so that $T = \min \big\{ T_F'' , \, \min_{k \geq 3} \{ T_k^0, T_k^\ell \} \big\}$. Since $f_k^2 \big|_{Q_\ell} (t) - f_k^2 \big|_{Q_0} (t) = q_k\, L (t)$ for $k \geq 3$, the Remark following \eqref{eq:T_max} applies verbatim: in the borderline case $T = T_F''$, an $N_k$-factor of $Q_0$ contracts if and only if the $N_k$-factor of $Q_\ell$ does.

\begin{theorem}[singularity models of the two-nut case]\label{thm:twonut_models}
Let $\big(\widehat{M}$, $g (t)\big)$ be as in Theorem \ref{thm:twonut_typeI}, rescaled at the Type I rate as in \S\ref{sec:limits}. According to the first singular time, the pointed Cheeger--Gromov limits are:

\emph{(i) Fiber collapse}, $T = T_F'' < \min_{k \geq 3} \{ T_k^0, T_k^\ell \}$: with arbitrary base points,

\[
\big( \mathbb{CP}^{m_0 + m_\ell + 1} \times \mathbb{C}^{n''} , \; g_{FS} (t) \oplus g_{\mathbb{C}^{n''}} \big) \,, \qquad n'' \;=\; \sum_{k \geq 3} \dim_{\mathbb{C}} N_k \,,
\]

the shrinking Fubini--Study soliton on the fiber times a static flat factor.

\emph{(ii) Contraction at $Q_0$}, $T = \min_{k \geq 3} \{ T_k^0 \} < \min \big\{ T_F'', \min_{k \geq 3} \{ T_k^\ell \} \big\}$: let $\mathcal{I}_0 := \{ k \geq 3 : T_k^0 = T \}$; then $0 < ( m_0 + 1 )\, q_k < p_k$ for every $k \in \mathcal{I}_0$, and with base points on $Q_0$ the limit is

\[
\big( \operatorname{Tot} ( E_{\mathcal{I}_0} ) \times \mathbb{C}^{n''} , \; g_E (t) \oplus g_{\mathbb{C}^{n''}} \big) \,, \qquad E_{\mathcal{I}_0} \;=\; \mathcal{L}_{\mathcal{I}_0}^{\oplus ( m_0 + 1 )} \,, \quad c_1 ( \mathcal{L}_{\mathcal{I}_0} ) \;=\; - \sum_{k \in \mathcal{I}_0} q_k\, a_k \,,
\]
where $E_{\mathcal{I}_0}$ is the restriction of the normal bundle $\mathcal{N}_{Q_0}$ to $\prod_{k \in \mathcal{I}_0} N_k$, $g_E (t)$ is the Dancer--Wang shrinking K\"ahler--Ricci soliton of ansatz form \cite[Theorem 3.36]{DW2011}, and $n'' = m_\ell + \sum_{k \geq 3, \, k \not\in \mathcal{I}_0} \dim_{\mathbb{C}} N_k$: the $\mathbb{CP}^{m_\ell}$-factor of $Q_0$ does not contract, and flattens into $\mathbb{C}^{m_\ell}$.

\emph{(iii) Contraction at $Q_\ell$}: the mirror of (ii) under $( s, q_k, m_0 ) \leftrightarrow ( \ell - s, - q_k, m_\ell )$, with $\mathcal{I}_\infty := \{ k \geq 3 : T_k^\ell = T \}$, $0 < - ( m_\ell + 1 )\, q_k < p_k$ for every $k \in \mathcal{I}_\infty$, and limit $\operatorname{Tot} \big( \mathcal{L}_{\mathcal{I}_\infty}^{\oplus ( m_\ell + 1 )} \big) \times \mathbb{C}^{n''}$, where $\mathcal{L}_{\mathcal{I}_\infty}^{\oplus ( m_\ell + 1 )}$ is the restriction of $\mathcal{N}_{Q_\ell}$ to $\prod_{k \in \mathcal{I}_\infty} N_k$, $c_1 ( \mathcal{L}_{\mathcal{I}_\infty} ) = \sum_{k \in \mathcal{I}_\infty} q_k\, a_k$, and $n'' = m_0 + \sum_{k \geq 3, \, k \not\in \mathcal{I}_\infty} \dim_{\mathbb{C}} N_k$.

\emph{(iv) Borderline}, $T = T_F'' = \min_{k \geq 3} \{ T_k^0, T_k^\ell \}$: let $\mathcal{I} := \{ k \geq 3 : T_k^0 = T \} = \{ k \geq 3 : T_k^\ell = T \}$; then $- \frac{p_k}{m_\ell + 1} < q_k < \frac{p_k}{m_0 + 1}$ for every $k \in \mathcal{I}$, and with arbitrary base points the limit is $\widehat{Z} \times \mathbb{C}^{n''}$, where

\[
\widehat{Z} \;:=\; \widehat{M} \Big|_{\prod_{k \in \mathcal{I}} N_k} \;=\; \mathbb{P} \big( \mathcal{O}^{\oplus ( m_0 + 1 )} \oplus \mathcal{L}_{\mathcal{I}}^{\oplus ( m_\ell + 1 )} \big) \,, \qquad c_1 ( \mathcal{L}_{\mathcal{I}} ) \;=\; \sum_{k \in \mathcal{I}} q_k\, a_k \,,
\]
carrying its unique compact shrinking K\"ahler--Ricci soliton of ansatz form (uniqueness by Tian--Zhu \cite{TZhu}), and $n'' = \sum_{k \geq 3, \, k \not\in \mathcal{I}} \dim_{\mathbb{C}} N_k$. When $\mathcal{I} = \{ 3, \cdots, r \}$, we have $\widehat{Z} = \widehat{M}$ and the flow is the Fano canonical-class case.

\end{theorem}

\begin{proof}[Sketch of proof]
All four cases can be proved as in \S\ref{sec:limits} with slight modifications. With Lemma \ref{lem:twonut} in place of the nut--bolt data of \S\ref{subsec:fiber-bisectional} and Proposition \ref{prop:affine-laws}; we indicate the modifications.

\emph{(i)} Lemma \ref{lma:fiber-collapsing-estimates} holds with the sums running over $k \geq 3$, and the Li--Yau estimate (Proposition \ref{prop:li-yau}) applies to $f_1^2$ and to $f_2^2$ alike; by \eqref{eq:twonut_L} and \eqref{eq:twonut_Lrate}, $f_1^2 + f_2^2 = L (t) = 2 ( m_0 + m_\ell + 2 ) ( T - t )$, so the fibers of $\widehat{\pi}''$ have diameter $\leq C \sqrt{T - t}$ and the proof of Theorem \ref{thm:fiber-collapse-model} goes through: the limit is compact in the fiber directions, is a shrinking K\"ahler--Ricci soliton by \cite{EMT, Naber}, and is identified with the Fubini--Study soliton on $\mathbb{CP}^{m_0 + m_\ell + 1}$ by \cite{BM}, \cite{TZhu}, as there.

\emph{(ii)} The sign condition follows from $f_k^2 \big|_{Q_\ell} - f_k^2 \big|_{Q_0} = q_k\, L (t)$: as $t \to T$ the left side tends to $f_k^2 \big|_{Q_\ell} (T) > 0$ (since $T < T_k^\ell$) while $L (T) > 0$ (since $T < T_F''$), so $q_k > 0$ for $k \in \mathcal{I}_0$, and $( m_0 + 1 )\, q_k < p_k$ is the finiteness of $T_k^0$. The proof of Theorem \ref{thm:contracting_Q0} is then run near the nut $Q_0$, using the closing conditions of Proposition \ref{prop:nut-closing} with $m_0$ in place of $m$, together with one modification borrowed from the proof of Theorem \ref{thm:contracting_Ql}: the $\mathbb{CP}^{m_\ell}$-factor of $Q_0$ does not contract --- $f_2^2 \big|_{Q_0} = L (t) \geq L (T) > 0$ --- so its $g_2$-block joins the non-contracting factors, and flattens into $\mathbb{C}^{m_\ell}$ by the Taylor-expansion argument of Lemma \ref{lma:flat}.

\emph{(iii)} is (ii) run from the other end, in the inward coordinate $\ell - s$ and the mirror gauge.

\emph{(iv)} The scale bookkeeping of the proof of Theorem \ref{thm:borderline} carries over: $K_j L = 2 ( m_0 + m_\ell + 2 ) ( 1 - t )$, $K_j f_k^2 \asymp 1$ for $k \in \mathcal{I}$, and the Li--Yau estimate applied to $f_1^2$ (or to $f_2^2$) gives $\hat\ell_j \in [ \frac{1}{C}, C ]$. The only structural change is that the $C^{m,k}$-estimates are run from both ends with the \emph{nut} closing conditions of Proposition \ref{prop:nut-closing}, and the limit profiles define ansatz metrics on the compact total space $\widehat{Z}$; the soliton identification via \cite{EMT, Naber}, Dancer--Wang \cite{DW2011}, and Tian--Zhu \cite{TZhu} is unchanged.
\end{proof}

\section{Appendix}

\begin{proof}[Proof of Lemma \ref{lem:bolt-nut-dichotomy}]
Let $\widehat{\xi}$ denote the (smooth) Killing field extending $\xi$; along the leaf $P_s$ we have
$|\widehat{\xi}\,|_{\widehat{g}} = H(s)$. Also, the geodesics perpendicular to each $P_s$ yield a smooth fiber bundle $\sigma : P \to Q$, $\sigma(p)$ with fibers $S_x := \sigma^{-1}(x) \cong S^{c-1}$, $c := \operatorname{codim} Q \geq 2$. Moreover, by the Fermi-coordinate expansion of a smooth metric around a submanifold
\cite[Ch.~9]{GrayTubes} (cf.\ \cite[\S1]{EschenburgWang}), the rescaled induced
metrics converge:
\begin{equation}\label{eq:roundlimit}
	s^{-2}\, g_s \big|_{S_x} \;\longrightarrow\; \text{the unit round metric on }
	S^{c-1} \qquad (s \to 0^+),
\end{equation}
uniformly on $S_x$.


\emph{Claim: $H(0) = 0$.}

\emph{Proof of the claim:} Suppose that $H(0)>0$ then, since there is a compactification at $0$, $J \neq \varnothing$ and one fixes $j \in J$. Since $\omega_j$ is symplectic and locally exact, for every $\phi \in [0, 2\pi)$ there is a contractible loop $\gamma_\phi$ bounding a disc $D_\phi$ in $N_j$ and such that
 $$q_j \int_{D_\phi} \omega_j = -\phi \ (\mathrm{mod}\ 2\pi) \text{ and } |\gamma_\phi|_{g_j}\leq L_j $$
 where the constant $L_j$ is independent of $\phi$. The $\theta$-horizontal (basic) lift of $\gamma_\phi$ through $p \in P$ moves only in the $\mathcal{H}_j$-directions and, since $d\theta = \sum_i q_i \pi^* \omega_i$, joins $p$ to $e^{i\phi} p$ within the
leaf, where $e^{i\phi}$ represents an action of an element of the group $\mathrm{U}(1)$ (the action is generated by the Killing field $\xi$) on $N_i$. Furthermore, its $g_s$-length is at most $F_j(s)\, L_j$. Hence
\[
d_{\widehat{g}}\big( \Phi(s,p),\; e^{i\phi}\, \Phi(s,p) \big)
\;\leq\; F_j(s)\, L_j \;\longrightarrow\; 0 \qquad (s \to 0^+).
\]
Letting $s \to 0^+$ gives $e^{i\phi} x = x$ for $x = \sigma(p)$. Therefore, the extended circle action fixes $Q$ pointwise and
$\widehat{\xi}\,|_Q = 0$, and by continuity of the smooth field $\widehat{\xi}$,
\[
H(0) \;=\; 
\big| \widehat{\xi}\,(x) \big| \;=\; 0 .
\]
That is a contradiction and the claim is proved.\\


Next, by continuity, 
\[
Q \;\cong\; \prod_{k \notin J} N_k , \qquad
S_x \;\cong\; S_J \;\cong\; S^{c-1}, \qquad c - 1 = 1 + 2 \textstyle\sum_{j \in J} n_j .
\]
If $J = \varnothing$ this reads $Q \cong N$ with $c = 2$ yielding the first case. It remains to show the following.\\ 

\emph{Claim: $J \neq \varnothing$ forces the Hopf structure.}

\emph{Proof of the claim:} We apply \eqref{eq:roundlimit} to the invariant connection metrics
$s^{-2}\big( H(s)^2\, \theta \otimes \theta + \sum_{j \in J} F_j(s)^2 \pi^* g_j \big)$
on $S_J$, we have that the limits $h_0 := \lim_{s\to0} H(s)/s$ and
$a_j := \lim_{s\to0} F_j(s)/s$ exist in $(0,\infty)$ 
the limit metric
$$\widehat{g}_\infty = h_0^2\, \theta \otimes \theta + \sum_{j \in J} a_j^2\, \pi^* g_j$$
is the unit round metric on $S^{c-1}$. The bundle projection
$S_J \to \prod_{j \in J} \big( N_j, a_j^2 g_j \big)$ is then a Riemannian
submersion of the unit round sphere with connected totally geodesic
one-dimensional fibres. 
By the classification of such submersions \cite{Escobales, Ranjan}, the base is isometric to
$\big( \mathbb{CP}^{m}, g_{FS} \big)$, $m = \sum_{j \in J} n_j$, with the
Fubini--Study metric 
and the submersion is equivalent to the Hopf fibration.  Since
$(\mathbb{CP}^m, g_{FS})$ is de Rham irreducible 
$J = \{ j_0 \}$ to be a singleton, and
$(N_{j_0}, a_{j_0}^2 g_{j_0})$ is isometric to
$(\mathbb{CP}^{n_{j_0}}, g_{FS})$. 

Finally, the bundle equivalence with the Hopf fibration
identifies $S_J$ --- the circle bundle over $\mathbb{CP}^{n_{j_0}}$ with Euler
class $q_{j_0} a_{j_0}$ --- with $S^{2 n_{j_0} + 1}$. By the Gysin sequence
\cite[\S 12]{MilnorStasheff}, a circle bundle over $\mathbb{CP}^{n_{j_0}}$ has
simply connected total space iff its Euler class is a generator of
$H^2(\mathbb{CP}^{n_{j_0}}; \mathbb{Z})$, whence $|q_{j_0}| = 1$. 
\end{proof}
\begin{proof} [Proof of Proposition \ref{prop:residual-oneill}]
	We work on the dense open set $M^0$; 
	identities established on $M^0$ extend by continuity across $Q$.
	
	\emph{(i)} It follows directly from Koszul formula for $2g(\nabla_U V, X_k)$.  
	
	\emph{(ii)} The components of $A^{\widehat{\pi}'}_X Y = \mathcal{V}'(\nabla_X Y)$ along $\partial_s$ and $\xi$ are computed verbatim as for the total submersion in \S\ref{subsec:oneill}, with the sums restricted to $k \geq 2$ because $\pi^*g_1(X, Y) = \pi^*\omega_1(X, Y) = 0$; substituting $(F_k^2)' = q_k H + \kappa_k$ splits the $\partial_s$-component as in \eqref{eq:A-total}. The $\mathcal{H}_1$-component vanishes by the Koszul formula for $2g(\nabla_X Y, X_1)$. Thus,
	\begin{equation}\label{eq:A-residual-split}
			A^{\widehat{\pi}'}_X Y \;=\; -\frac{H}{2}\sum_{k \geq 2} q_k\,\pi^*g_k(X,Y)\;\partial_s \;-\; \frac{1}{2}\sum_{k \geq 2} \kappa_k\,\pi^*g_k(X,Y)\;\partial_s \;-\; \frac{H}{2}\sum_{k \geq 2} q_k\,\pi^*\omega_k(X,Y)\;\xi^* \, .
	\end{equation}
\end{proof}

\begin{proof}[Proof of Corollary \ref{cor:norm-A-residual}]
	Recall that the $A'$-tensor has no $\mathcal{H}_1$-component. So the proof is almost verbatim to that of Corollary 2.2 of \cite{FT}. In particular, with respect to metric $g$, every component of $A'$ --- symmetric part and skew part (\ref{eq:A-residual}), and their metric duals (\ref{eq:A-cartesian-mixed})--- is an algebraic combination of the functions $q_k H / f_k^2$, $k \geq 2$; hence, since $| \nabla f_k^2 | = | q_k |\, H$ the result follows.
	
	
\end{proof}

\bibliographystyle{amsplain}
\bibliography{bio}

\end{document}